\documentclass[reqno]{amsart}
\usepackage{fullpage}

\usepackage{amsmath} 

\usepackage{amsthm}
\usepackage{amssymb}
\usepackage{graphics}
\usepackage{latexsym}
\usepackage{comment}
\usepackage{extarrows}
\usepackage{wrapfig}
\usepackage{comment}
\usepackage {colortbl,array,xcolor}
\usepackage{hhline}
\usepackage{tikz}
\usetikzlibrary{intersections, calc, math, patterns}

\numberwithin{equation}{section}
\newtheorem{thm}{Theorem}[section]
\newtheorem{prop}[thm]{Proposition}
\newtheorem{lem}[thm]{Lemma}
\newtheorem{cor}[thm]{Corollary}
\newtheorem{rem}[thm]{Remark}
\newtheorem{defn}[thm]{Definition}
\theoremstyle{remark}
\newtheorem{ass}{Assumption}

\newtheorem{manualassinner}{Assumption}
\newenvironment{manualass}[1]{%
  \renewcommand\themanualassinner{#1}%
  \manualassinner
}{\endmanualassinner}

\newcommand{\R}{{\mathbb R}}
\newcommand{\Z}{{\mathbb Z}}

\newcommand{\N}{{\mathbb N}}
\newcommand{\C}{{\mathbb C}}
\newcommand{\LR}[1]{{\langle {#1} \rangle }}
\newcommand{\cross}{\times}
\newcommand{\e}{\varepsilon}
\newcommand{\F}{\mathcal{F}}

\newcommand{\ha}{\widehat}

\newcommand{\supp}{\operatorname{supp}}

\title[Subcritical well-posedness results for the ZK equation]{Subcritical well-posedness results for the Zakharov-Kuznetsov equation in dimension three and higher}

\author[S.~Herr]{Sebastian Herr}
\address[Sebastian Herr]{Universit\"{a}t Bielefeld
Fakult\"{a}t f\"{u}r Mathematik
Postfach 10 01 31
33501 Bielefeld
Germany}
\email[Sebastian Herr]{herr@math.uni-bielefeld.de}

\author[S.  Kinoshita]{Shinya Kinoshita}
\address[Shinya Kinoshita]{Universit\"{a}t Bielefeld
Fakult\"{a}t f\"{u}r Mathematik
Postfach 10 01 31
33501 Bielefeld
Germany}
\email[Shinya Kinoshita]{kinoshita@math.uni-bielefeld.de}


\begin{document}
\begin{abstract}
The Zakharov-Kuznetsov equation in space dimension $d\geq 3$ is
considered. It is proved that the Cauchy problem is locally well-posed
in $H^s(\R^d)$ in the full subcritical range $s>(d-4)/2$, which is
optimal up to the endpoint. As a corollary, global well-posedness in $L^2(\R^3)$ and, under a smallness condition, in $H^1(\R^4)$, follow.
\end{abstract}
\maketitle

\section{Introduction}
We consider the Zakharov-Kuznetsov equation with the quadratic nonlinearity
\begin{equation}\label{eq:ZK}
  \begin{split}
  \partial_t u+\partial_{x}\Delta u={}&\partial_{x}u^2 \text{ in
  }(-T,T) \times \R^d\\
  u(0,\cdot)={}&u_0 \in H^s(\R^d)
\end{split}
\end{equation}
where $u=u(t,x,\mathbf{y})$ is real-valued and
$\Delta$ denotes the Laplacian with respect to $(x,\mathbf{y})\in
\R\times \R^{d-1}$ .
The equation \eqref{eq:ZK} arises as an asymptotic model wave propagation in a magnetized
plasma \cite{B06,SS99}. It was introduced in \cite{ZK74} in $d=2,3$, see also \cite{LS82} for a formal derivation.
More recently, it was rigorously derived from the Euler-Poisson system as a long-wave and small-amplitude limit, see
\cite[Section 10.3.2.6]{LLS13}. The Zakharov-Kuznetsov equation \eqref{eq:ZK} generalizes the
Korteweg-de Vries equation (which is the case $d=1$). In particular,
it has solitary wave solutions. Recently, their asymptotic stability has been proven in \cite{CMPS16}.

Real-valued solutions of \eqref{eq:ZK} conserve the $L^2$-norm and the energy
\[
\frac12 \int_{\R^d}|\nabla_{x,\mathbf{y}}u(t,x,\mathbf{y})|^2dxd\mathbf{y} +\frac13\int_{\R^d}u(t,x,\mathbf{y})^3 dxd\mathbf{y}.
\]

If $u$ is a solution, then for any $\lambda>0$ the function
\[
u_\lambda(t,x,\mathbf{y})=\lambda^{2}u(\lambda^3 t, \lambda x,\lambda \mathbf{y})
\]
also solves \eqref{eq:ZK}. This implies that $s_c:=(d-4)/2$ is the critical
Sobolev regularity for \eqref{eq:ZK} in the sense that the corresponding (homogeneous) Sobolev norm is invariant under the rescaling described above.

In this paper, we will focus on the case of spatial dimensions $d\geq
3$ and prove local well-posedness of the Cauchy problem associated with \eqref{eq:ZK} in the full sub-critical
range. Let $H^s(\R^d)$ denote the Sobolev space of tempered distributions on $\R^d$ all derivatives up to order $s$ in $L^2(\R^d)$, see Subsection \ref{subsec:not} for a precise definition.

\begin{thm}\label{thm:main-zk}
 Let $d\geq 3$. For any $s > (d-4)/2$, the Cauchy problem for \eqref{eq:ZK} is locally
 well-posed in $H^s(\R^d)$.
\end{thm}

Note that the energy-subcritical dimensions are $d\leq 5$, and the $L^2$-subcritical dimensions are $d\leq 3$.
From the conservation laws mentioned above and the Gagliardo-Nirenberg inequality we deduce

\begin{cor}\label{cor:gwp}
If $d=3$, the Cauchy problem for \eqref{eq:ZK} is globally
well-posed for real-valued initial data in $L^2(\R^3)$.
If $d=4$, the Cauchy problem for \eqref{eq:ZK} is globally well-posed
for real-valued initial data in $H^1(\R^4)$ with sufficiently small $L^2(\R^4)$-norm.
\end{cor}

\subsection*{Previous results}\label{subsec:prev}

The Cauchy problem for the Zakharov-Kuznetsov equation and the
so-called generalized Zakharov-Kuznetsov equation
\[
\partial_t u+\partial_{x}\Delta u={}\partial_{x}u^{k+1}, \quad ( k \in
\N),
\]
have been studied extensively. In dimension $d=2$ and for the quadratic
nonlinearity ($k=1$),
 Faminsk\u{\i} \cite{Fa95} proved
global well-poseness in  $H^1(\R^2)$,
Linares and Pastor \cite{LP09} proved local well-posedness for $s >
3/4$, Gr\"{u}nrock and Herr \cite{GH14} and Molinet and Pilod
\cite{MP15},  proved local well-posedness for $s > 1/2$,
and recently the second named author \cite{Ki2019} proved local
well-posedness for $s>-1/4$, which is optimal up to the end-point. In
dimension $d=2$ and for the cubic nonlinearity ($k=2$) Biagioni and
Linares \cite{BL03} proved global well-posedness in  $H^1(\R^2)$,
Linares and Pastor \cite{LP09} proved local well-posedness for $s >
3/4$ and \cite{LP11} global well-posedness for $s > 53/63$,
Ribaud and Vento \cite{RV12-gZK} proved local well-posedness for $s >
1/4$, and recently Bhattacharya, Farah and Roudenko \cite{BFR19arxiv} proved global well-posedness for $s>3/4$,
and the second named author \cite{Ki2019mZK} proved local well-posedness for $s=1/4$.
In dimension $d=3$ and for the quadratic nonlinearity ($k=1$)
Linares and Saut \cite{LS09} proved local well-posedness for $s >
9/8$, Ribaud and Vento \cite{RV12-ZK} proved local well-posedness for $s > 1$ and in $B_2^{1,1}(\R^3)$,
and Molinet and Pilod \cite{MP15} proved global well-posedness for $s
>1$.
In dimension $d \geq 3$ and for the cubic nonlinearity ($k=2$)
 Gr\"{u}nrock \cite{Gru14} proved local well-posedness in the full
 sub-critical range  $s > d/2-1$ and the second named author
 \cite{Ki2019mZK} proved small data global well-posedness at the scaling critical regularity $s=d/2-1$.
For more results in the case $k \geq 3$, we refer to the papers \cite{FLP12}, \cite{Gru15arxiv}, \cite{LP11}, \cite{RV12-gZK}, \cite{Gru15arxiv}.

\subsection*{Strategy of proof}\label{subsec:str}
Theorem \ref{thm:main-zk} will be proved by a contraction
argument in Fourier restriction spaces $X^{s,b}$, which is based on
the bilinear estimate provided in Theorem \ref{mainthm} below. Since
the general strategy of proof is standard, we
refer to \cite[Section 2]{GTV97} for details and
precise formulations and focus on the key bilinear estimate.

In the low-regularity analysis of the Korteweg-de Vries equation,
the set of time-resonances plays a crucial role, see
\cite{B93,KPV96}. It is the set of spatial frequencies allowing the product of
two solutions to the homogeneous equations to form another solution
to the homogeneous equation. In the case of the KdV equation, it is the set
\[
\{(\xi,\xi_1)\in \R^2: 3\xi(\xi-\xi_1)\xi_1=0\},
\]
while in the case of the ZK equation it is significantly
more complex. More precisely, it is the set
\[
\{(\xi,\boldsymbol{\eta},\xi_1,\boldsymbol{\eta}_1)\in \R^{2d}: 3\xi(\xi-\xi_1)\xi_1+\xi\boldsymbol{\eta}-\xi_1\boldsymbol{\eta}_1-(\xi-\xi_1)(\boldsymbol{\eta}-\boldsymbol{\eta}_1)=0\}.
\]
First well-posedness results \cite{GH14,MP15} did not rely on the structure of this set.
Recently, the second author \cite{Ki2019} established well-posedness of the
Zakharov-Kuznetsov equation in $d=2$ for $s>-1/4$, which turned out to
be optimal within the purely perturbative regime (up to the
endpoint). The key observation is that the bulk of
frequencies close to the resonant set are transversal. This allows to
invoke the Loomis-Whitney inequality \cite{LW49}, more
precisely a nonlinear generalization thereof
\cite{BCW05,BHT10,KS15}. This strategy has been used previously in the case of the
low-dimensional Zakharov system, see \cite{BHHT09,BH11}, in which case the resonant set is easier to understand.

\subsection*{Outline of the paper}\label{subsec:outline}
In Section \ref{sec:pre} we will discuss preliminaries concerning notation, Strichartz and transversal estimates. The key bilinear estimate is stated as Theorem \ref{mainthm} in Section \ref{sec:key-bil} and first reductions are performed. The proof of  Theorem \ref{mainthm} is then split into cases, which are treated in the following Sections \ref{sec:proof-bil-1} -- \ref{sec:proof-bil-3}. In an appendix we include a proof of the well-known transversal $L^2$ estimate from Section \ref{sec:pre}.

  \section{Preliminaries}\label{sec:pre}
\subsection{Notation}\label{subsec:not}
We write $A{\ \lesssim \ } B$ if there exists $C>0$ such that $A \le CB$ and $A\ll B$ if $A \le cB$ for some small enough $c<1$. Also, $A\sim B$ means $A{\ \lesssim \ } B$ and $B{\ \lesssim \ } A.$

Let $N$, $L \geq 1$ be dyadic numbers, i.e. there exist $n_1$, $n_2 \in \N_{0}$ such that 
$N= 2^{n_1}$ and $L=2^{n_2}$, and $\psi \in C^{\infty}_{0}((-2,2))$ be an even, non-negative function which satisfies $\psi (t)=1$ for $|t|\leq 1$ and letting 
$\psi_N (t):=\psi (t N^{-1})-\psi (2t N^{-1})$, 
$\psi_1(t):=\psi (t)$, 
the equality $\displaystyle{\sum_{N}\psi_{N}(t)=1}$ holds. 
Here and in the sequel we use the convention that capitalized summation indices run over $2^{\N_0}$.
Let $d \geq 3$. We distinguish the first and the other spatial variables. 
To be precise, $(x, \mathbf{y}) \in \R \times \R^{d-1}$ denotes the space variables. 
Similarly, $(\xi, \boldsymbol{\eta}) \in \R \times \R^{d-1}$ denotes the frequency variables of $(x, \mathbf{y})$.

Let $u=u(t,x,\mathbf{y}).\ \F_t u,\ \F_{x,\mathbf{y}} u$ denote the Fourier transform of $u$ in time, space, respectively. 
$\F_{t, x,\mathbf{y}} u = \ha{u}$ denotes the Fourier transform of $u$ in space and time.
For $s\in \R$ define $H^s(\R^d)$ to be the space of all tempered distributions $f=f(x, \mathbf{y})$ on $\R^d$ satisfying
\[
\|f\|_{H^s}:=\Big(\int_{\R^d}(1+|\xi|^2+|\boldsymbol{\eta}|^2)^s|\mathcal{F}_{x,\mathbf{y}}f(\xi, \boldsymbol{\eta})|^2d\xi d\boldsymbol{\eta}\Big)^{\frac12}<+\infty.
\]

We define frequency and modulation projections $P_N$, $Q_L$ as
\[
(\F_{x,\mathbf{y}}P_{N}u ):=  \psi_{N}(|\cdot| )(\F_{x,\mathbf{y}}{u}),\qquad
\widehat{Q_{L} u}(\tau ,\xi,\boldsymbol{\eta} ):=  \psi_{L}(\tau - \xi(\xi^2+|\boldsymbol{\eta}|^2))\widehat{u}(\tau ,\xi,\boldsymbol{\eta} ).
\]

Let $ B_r(p) \subset \R^d$ denote the open ball with radius $r>0$ and center $p \in \R^d$, and define the spatial Fourier multiplier $P_{B_r(p)} f= \F_{x,\mathbf{y}}^{-1} \mathbf{1}_{B_r(p)}\F_{x,\mathbf{y}}f$, where $\mathbf{1}_{B_r(p)}$ denotes the characteristic function of $B_r(p)$.

We now define $X^{s,b}(\R^{d+1})$ spaces. 
Let $s, b \in \R$.
\[ X^{s,\,b} (\R^{d+1}) :=\Big\{ f \in \mathcal{S}'(\R^{d+1}) \ | \ \|f  \|_{X^{s,\, b}} := \Big( \sum_{N,\, L} N^{2s} L^{2b} \|  P_N Q_L f \|^2_{L_{t,x,\mathbf{y}}^{2}(\R^{d+1})} \Big)^{1/2}<+\infty\Big\}.
\]
For convenience, we define the set in frequency as
\begin{equation*}
G_{N, L} := \{ (\tau, \xi,\boldsymbol{\eta}) \in \R^{d+1} \, | \, \psi_{L}(\tau - \xi(\xi^2+|\boldsymbol{\eta}|^2)) \psi_{N}(|(\xi,\boldsymbol{\eta})| ) \not= 0 .\}
\end{equation*}

A simple calculation shows
$\overline{{X}^{s,\,b}} = X^{s,\,b}$ and 
$(X^{s,\,b})^* =X^{-s,\,-b}$,
for $s$, $b \in \R$.

Define the propagator for the linear Zakharov-Kuznetsov equation $U(t) = e^{-t \partial_{x} \Delta}$ and the one-dimensional Fourier multiplier $|\partial_{x}|^s = \F_{x}^{-1} |\xi|^s \F_{x}$.

\subsection{Strichartz type estimates and transversal estimates}\label{sub:str}

We start with a Strichartz or (dual) restriction type
estimate, where curvature properties of the characteristic set of the
differential operator are used.
\begin{prop}\label{l4-strichartz}
 Let $d\geq 3$ and $0 \leq s \leq 1/4$. 
 Then, for all $p \in \R^d$ and $r>0$ we have
\begin{equation*}
\||\partial_{x}|^s U(t) P_{B_r(p)} f \|_{L_{t,x,\mathbf{y}}^4} 
\lesssim r^{\frac{d-3}{4}+s} \| P_{B_r(p)} f \|_{L_{x,\mathbf{y}}^2}.
\end{equation*}
\end{prop}

\begin{proof}
It suffices to show the endpoint cases $s=0$ and $s=1/4$. 
We follow the proof of the Strichartz estimates of the KP-II type equations on cylinders, see Theorem 2 in \cite{GPS09} and \cite{Gru14}. 
The Littlewood-Paley theorem implies that it suffices to show the claim under the condition 
\begin{equation}
\supp \F_{x,\mathbf{y}} f \subset \{(\xi,\boldsymbol{\eta}) \, |\, 2^k \leq |\xi|\leq  2^{k+1} \} \cap B_r(p),
\label{est01-l4-strichartz}
\end{equation}
where $k$ is an arbitrary integer. 
Let $\psi : \R^{d-1} \to \C$. 
We recall the classical (non-endpoint) Strichartz estimate of the Schr\"{o}dinger
equations on $\R^{d-1}$, i.e.
\begin{equation*}
\| e^{it \Delta_{\mathbf{y}}} \psi \|_{L_t^p L_{\mathbf{y}}^q} \lesssim \| \psi\|_{L_{\mathbf{y}}^2},
\end{equation*}
where 
$(p,q)$ is an admissible pair satisfying $2 < p \leq \infty$ and $2/p = (d-1)(1/2 - 1/q)$. 
Let $(4, q_0)$ be admissible.
Since $f$ satisfies \eqref{est01-l4-strichartz}, if we fix $\xi\in \R$, 
by using the Sobolev inequality and the above Strichartz estimate, we easily get
\begin{align}
\|e^{it \xi \Delta_{\mathbf{y}}} \F_{x}f(\xi, \cdot) \|_{L_t^4 L_{\mathbf{y}}^4} & \leq 
r^{\frac{d-3}{4}} |\xi|^{-\frac{1}{4}} 
\|e^{it \Delta_{\mathbf{y}}} \F_{x}f(\xi, \cdot)\|_{L_t^4 L_{\mathbf{y}}^{q_0}}\notag \\
& \leq r^{\frac{d-3}{4}} |\xi|^{-\frac{1}{4}} \| \F_{x}f(\xi, \cdot)\|_{ L_{\mathbf{y}}^2}.\label{est02-l4-strichartz}
\end{align}
Therefore, it follows from Plancherel's theorem in $x$, H\"{o}lder's inequality, we get
\begin{align*}
&\| U(t)  P_{B_r(p)}  f \|_{L_{t,x,\mathbf{y}}^4}^2  = 
\|(U(t) P_{B_r(p)} f)(U(t) P_{B_r(p)} f)\|_{L_t^2 L_x^2}\\
\leq{}& \left\| 
\int_{\R} \left\|  
\bigl( e^{i t (\xi-\xi') \Delta_{\mathbf{y}}} \F_{x} f(\xi-\xi', \cdot) \bigr) 
\bigl( e^{i t \xi' \Delta_{\mathbf{y}}} \F_{x}  f(\xi', \cdot)\bigr) 
\right\|_{L_{t \mathbf{y}}^2} d \xi' 
 \right\|_{L_{\xi}^2}\\
\leq{}& \left\| 
\int_{\R} \|  e^{i t (\xi-\xi') \Delta_{\mathbf{y}}} \F_{x} f(\xi-\xi', \cdot)\|_{L_{t \mathbf{y}}^4}
\|  e^{i t \xi' \Delta_{\mathbf{y}}} \F_{x}  f(\xi', \cdot)\|_{L_{t \mathbf{y}}^4}
d \xi' 
                            \right\|_{L_{\xi}^2}
\end{align*}
Now, we use \eqref{est02-l4-strichartz} and continue with
\begin{align*}
\| U(t)  P_{B_r(p)}  f \|_{L_{t,x,\mathbf{y}}^4}^2  & \lesssim r^{\frac{d-3}{2}}  \left\| 
\int_{\R} (|\xi - \xi'| |\xi'|)^{-\frac{1}{4}} \|  \F_{x} f(\xi-\xi', \cdot)\|_{L_{\mathbf{y}}^2}
\| \F_{x}  f(\xi', \cdot)\|_{L_{\mathbf{y}}^2}
d \xi' 
 \right\|_{L_{\xi}^2}\\
& \lesssim r^{\frac{d-3}{2}} \left\| 
\int_{\R}  \|  \F_{x} f(\xi-\xi', \cdot)\|_{L_{\mathbf{y}}^2}
\| \F_{x}  f(\xi', \cdot)\|_{L_{\mathbf{y}}^2}
d \xi' 
 \right\|_{L_{\xi}^\infty}\\
& \leq r^{\frac{d-3}{2}} \|f\|_{L_{x}^2}^2.
\end{align*}
This completes the proof for $s=0$. Here we used $2^k \leq |\xi'| \leq 2^{k+1}$ and $2^k \leq  |\xi-\xi'| \leq  2^{k+1}$. Similarly, we have
\begin{align*}
\| |\partial_{x}|^{1/4} U(t)  P_{B_r(p)}  f \|_{L_{t,x,\mathbf{y}}^4}^2  & = 
\|(|\partial_{x}|^{1/4} U(t) P_{B_r(p)} f)( |\partial_{x}|^{1/4}U(t) P_{B_r(p)} f)
\|_{L_t^2 L_x^2}\\
& \lesssim r^{\frac{d-3}{2}}  \left\| 
\int_{\R} \|  \F_{x} f(\xi-\xi', \cdot)\|_{L_{\mathbf{y}}^2}
\| \F_{x}  f(\xi', \cdot)\|_{L_{\mathbf{y}}^2}
d \xi' 
 \right\|_{L_{\xi}^2}\\
& \lesssim r^{\frac{d-2}{2}} \left\| 
\int_{\R}  \|  \F_{x} f(\xi-\xi', \cdot)\|_{L_{\mathbf{y}}^2}
\| \F_{x}  f(\xi', \cdot)\|_{L_{\mathbf{y}}^2}
d \xi' 
 \right\|_{L_{\xi}^\infty}\\
& \leq r^{\frac{d-2}{2}} \|f\|_{L_{x,\mathbf{y}}^2}^2,
\end{align*}
where we used  \eqref{est01-l4-strichartz} again.
\end{proof}
Next, we recall the standard bilinear estimate which exploits
transversality, see
e.g. \cite[Lemma 2.6]{Candy2018}
for a proof with a general phase function and for references. We also provide a proof in
the appendix.
\begin{prop}\label{bilinear-transversal}
Let $d\geq 2$, $N_2 \leq N_1$, $\varphi (\xi,\boldsymbol{\eta}) = \xi(|\xi|^2+|\boldsymbol{\eta}|^2)$. 
Suppose that
\begin{equation*}
\supp \widehat{u}_{N_1,L_1} \subset G_{N_1,L_1} \cap (\R \times B_r(p)) , \quad 
\supp \widehat{v}_{N_2,L_2} \subset G_{N_2,L_2},
\end{equation*}
and there exists $K$ which satisfies $K \gtrsim r N_1$ and
\begin{equation*}
|\nabla \varphi(\xi_1,\boldsymbol{\eta}_1)-\nabla \varphi(\xi_2,\boldsymbol{\eta}_2)| \gtrsim K,
\end{equation*} 
for all $(\xi_1,\boldsymbol{\eta}_1)$, $(\xi_2,\boldsymbol{\eta}_2)$ in the spatial Fourier support of
$u_{N_1,L_1}$ and $v_{N_2,L_2}$, respectively. 
Then, we have
\begin{equation}
\| u_{N_1,L_1} \, v_{N_2, L_2} \|_{L_{t,x,\mathbf{y}}^2} \lesssim r^{\frac{d-1}{2}} K^{-\frac{1}{2}} 
(L_1 L_2)^{\frac{1}{2}} \| u_{N_1, L_1} \|_{L_{t,x,\mathbf{y}}^2} \|v_{N_2, L_2}  \|_{L_{t,x,\mathbf{y}}^2}.\label{est01-bilinear-transversal}
\end{equation}
In particular, if $N_2 \leq 2^{-3}N_1$ and 
\begin{equation*}
\supp \widehat{u}_{N_1,L_1} \subset G_{N_1,L_1}, \quad 
\supp \widehat{v}_{N_2,L_2} \subset G_{N_2,L_2},
\end{equation*}
we have
\begin{equation}
\| u_{N_1,L_1} \, v_{N_2, L_2} \|_{L_{t,x,\mathbf{y}}^2} \lesssim N_1^{-1} N_2^{\frac{d-1}{2}} 
(L_1 L_2)^{\frac{1}{2}} \| u_{N_1, L_1} \|_{L_{t,x,\mathbf{y}}^2} \|v_{N_2, L_2}  \|_{L_{t,x,\mathbf{y}}^2}.\label{est02-bilinear-transversal}
\end{equation}
\end{prop}

A trilinear estimate based on transversality is the following
generalization of the classical Loomis-Whitney inequality, which is Corollary 1.5 in \cite{BHT10}, see also \cite{BCW05,KS15}.

\begin{prop} \label{prop2.3}
Assume that for $i\in \{1,2,3\}$ the surface $S_i \subset \R^3$ 
is an open and bounded subset of $S_i^*$ which 
satisfies the following three conditions:

\textnormal{(i)} For a convex $U_i \subset \R^3$ such that
$\mathrm{dist}(S_i, U_i^c) \geq \mathrm{diam}(S_i)$ we have
\begin{equation*}
S_i^* = \{ {\lambda_i} \in U_i \ | \ \Phi_i({\lambda_i}) = 0 , \nabla \Phi_i \not= 0, \Phi_i \in C^{1,1} (U_i) \}.
\end{equation*}

\textnormal{(ii)} The unit normal vector field $\mathfrak{n}_i$ on $S_i^*$ satisfies the H\"{o}lder condition
\begin{equation*}
\sup_{\lambda, \lambda' \in S_i^*} \frac{|\mathfrak{n}_i(\lambda) - 
\mathfrak{n}_i(\lambda')|}{|\lambda - \lambda'|}
+ \frac{|\mathfrak{n}_i(\lambda) ({\lambda} - {\lambda}')|}{|{\lambda} - {\lambda}'|^2} \lesssim 1.
\end{equation*}

\textnormal{(iii)} There exists $\delta >0$ such that
$\mathrm{diam}({S_i}) \lesssim \delta$ and the matrix ${N}({\lambda_1}, {\lambda_2}, {\lambda_3}) = ({\mathfrak{n}_1} 
({\lambda_1}), {\mathfrak{n}_2}({\lambda_2}), {\mathfrak{n}_3}({\lambda_3}))$ 
satisfies the transversality condition
\begin{equation*}
\delta \leq |\textnormal{det} {N}({\lambda_1}, {\lambda_2}, {\lambda_3}) |
\leq 1, \quad
\text{ for all }({\lambda_1}, {\lambda_2}, {\lambda_3}) \in {S_1^*}
\cross {S_2^*} \cross {S_3^*}.
\end{equation*}

Then, for functions $f \in L^2 (S_1)$ and $g \in L^2 (S_2)$, the restriction of the convolution $f*g$ to 
$S_3$ is a well-defined $L^2(S_3)$-function which satisfies 
\begin{equation*}
\| f *g \|_{L^2(S_3)} \lesssim {\delta}^{-\frac12} \| f \|_{L^2(S_1)} \| g\|_{L^2(S_2)}.
\end{equation*}
\end{prop}

\section{The key bilinear estimate}\label{sec:key-bil}
The main contribution of this paper is the following:
\begin{thm}\label{mainthm}
For any $s > (d-4)/2$, there exist $b \in (\frac12, 1)$ and $b'\in (b-1,0)$, such that
\begin{equation}
\|\partial_{x} (uv) \|_{X^{s,\,b'}} \  \lesssim \|u \|_{X^{s,\,b}} \|v \|_{X^{s,\,b}}.
\label{goal01-mthm}
\end{equation}
\end{thm}
The remainder of the paper will be devoted to its proof.
By a duality argument and dyadic decompositions, we observe that 
\begin{equation}\label{eq:goal01-mthm-equiv}
\eqref{goal01-mthm} 
\iff \  \left| \int{ w \partial_{x} (uv)} dtdxd \mathbf{y} \right| \lesssim \|u \|_{X^{s,\,b}} \|v \|_{X^{s,\,b}} 
         \|w \|_{X^{-s,\,-b'}}.
       \end{equation}

We will use the shorthand notations
\[
 w_{N_0, L_0} := Q_{L_0} P_{N_0}w, \ u_{N_1, L_1} := Q_{L_1} P_{N_1} u, \ 
v_{N_2, L_2} := Q_{L_2} P_{N_2}v. 
\]

Obviously, \eqref{eq:goal01-mthm-equiv} follows from
\begin{equation}\label{goal02-mthm}
\sum_{{{\substack{N_j, L_j\\ \tiny{(j=0,1,2)}}}}}
 \left|\int{ \left(\partial_{x} w_{N_0, L_0} \right) 
 u_{N_1, L_1}  v_{N_2, L_2}
} 
dt dx d \mathbf{y} \right|
\lesssim{}\|u \|_{X^{s,\,b}} \|v \|_{X^{s,\,b}} 
\|w \|_{X^{-s,\,-b'}}.
\end{equation}

For brevity, we write
\[L_{012}^{\max} := \max (L_0, L_1, L_2), \  
N_{012}^{\max} := \max (N_0, N_1, N_2), \ N_{012}^{\min} := \min (N_0, N_1, N_2).\]

\subsection{Reductions}\label{subsec:red}
Here, we prove \eqref{goal02-mthm} in the following relatively simple
cases:
\begin{enumerate}
\item[(1)] $L_{012}^{\max} \gtrsim (N_{012}^{\max})^3$, \quad
\item[(2)] $N_{012}^{\min} \sim 1$ and $L_{012}^{\max} \ll
  (N_{012}^{\max})^3$
\end{enumerate}
We will use Propositions \ref{l4-strichartz} and \ref{bilinear-transversal}, respectively.

%%%%%%%%%%%%%%%%%%%%%%%%%%%%%%%%%%%%%
%%%%%%%%%%%%%%%%%%%%%%%%%%%%%%%%%%%%%
%%%%%%%%%%%%%%%%%%%%%%%%%%%%%%%%%%%%%
%%%%%%%%%%%%%%%%%%%%%%%%%%%%%%%%%%%%%
We first assume $L_{012}^{\max} \gtrsim (N_{012}^{\max})^3$ and show 
\begin{equation}
\begin{split}
\left|\int{ w_{N_0, L_0} \, 
 u_{N_1, L_1} \, v_{N_2, L_2}
} 
dt dx d \mathbf{y} \right| &\\
\lesssim 
(N_{012}^{\min})^{\frac{d-3}{2}}  (N_{012}^{\max}  )^{-\frac{3}{2} + 3 \e} & 
L_0^{\frac{1}{2}- \e}  (L_1 L_2)^{\frac{1}{2}} 
\|u_{N_1, L_1} \|_{L^2} \|v_{N_2, L_2} \|_{L^2}
\|w_{N_0, L_0} \|_{L^2}
\end{split}\label{est01-mthm}
\end{equation}
for some small $\e>0$.
Clearly, this inequality gives \eqref{goal02-mthm}. 
Here we only consider the case $L_0 \gtrsim (N_{012}^{\max})^3$. 
The other two cases $L_1 \gtrsim (N_{012}^{\max})^3$ and $L_2 \gtrsim (N_{012}^{\max})^3$ can be treated similarly. 
By the almost orthogonality, we may replace $ u_{N_1, L_1}$ and $v_{N_2, L_2}$ by 
$P_{B} u_{N_1, L_1}$ and 
$P_{B'}v_{N_2, L_2}$ where $P_{B}$ and $P_{B'}$ denote the spatial frequency localization operators 
for some fixed balls $B$ and $B'$ with radius $N_{012}^{\min}$, respectively. 
It follows from the H\"{o}lder's inequality and Proposition \ref{l4-strichartz} that
\begin{align*}
& \left|\int{ w_{N_0, L_0} \, 
 (P_{B}u_{N_1, L_1}) \, 
(P_{B'} v_{N_2, L_2})
} 
dt dx d \mathbf{y} \right|\\
& \lesssim \|w_{N_0, L_0}\|_{L^2} \|P_{B} u_{N_1, L_1}\|_{L^4} 
\| P_{B'}v_{N_2, L_2}\|_{L^4} \\
& \lesssim (N_{012}^{\min})^{\frac{d-3}{2}}  (N_{012}^{\max} )^{-\frac{3}{2} + 3 \e} 
L_0^{\frac{1}{2}- \e} (L_1 L_2)^{\frac{1}{2}}
\|u_{N_1, L_1} \|_{L^2} \|v_{N_2, L_2} \|_{L^2}
\|w_{N_0, L_0} \|_{L^2},
\end{align*}
which completes the proof of \eqref{est01-mthm}. 

Next we deal with the case $N_{012}^{\min} \sim 1$ and $L_{012}^{\max} \ll (N_{012}^{\max})^3$. 
If $1 \sim N_0 \sim N_1 \sim N_2$, by using the $L^4$ Strichartz estimate, we get
\begin{align*}
 \left|\int{ w_{N_0, L_0} \, 
u_{N_1, L_1} \, 
v_{N_2, L_2}
} 
dt dx d \mathbf{y} \right| 
& \lesssim \|w_{N_0, L_0}\|_{L^2} \| u_{N_1, L_1}\|_{L^4} 
\|v_{N_2, L_2}\|_{L^4} \\
& \lesssim 
 (L_1 L_2)^{\frac{1}{2}}
\|u_{N_1, L_1} \|_{L^2} \|v_{N_2, L_2} \|_{L^2}
\|w_{N_0, L_0} \|_{L^2},
\end{align*}
which implies \eqref{goal02-mthm}. 
Thus, by symmetry, we only need to consider $1 \sim N_0 \ll N_1 \sim N_2$ and 
$1 \sim N_2 \ll N_0 \sim N_1$. 
The both cases are treated by Proposition \ref{bilinear-transversal}. 
First we assume $1 \sim N_0 \ll N_1 \sim N_2$ and show the following.
\begin{equation}
\begin{split}
\left|\int{ w_{N_0, L_0} \, 
 u_{N_1, L_1} \, v_{N_2, L_2}
} 
dt dx d \mathbf{y} \right| &\\
\lesssim 
N_1^{-1}    (L_0 L_1)^{\frac{1}{2}} & 
\|u_{N_1, L_1} \|_{L^2} \|v_{N_2, L_2} \|_{L^2}
\|w_{N_0, L_0} \|_{L^2},
\end{split}\label{est02-mthm-1}
\end{equation}
which immediately yields \eqref{goal02-mthm} since $-2s < 4-d \leq 1$ and $L_0 \ll N_1^3$. 
We deduce from $N_0 \sim 1$ and the almost orthogonality, we can replace $ u_{N_1, L_1}$ by 
$P_B u_{N_1, L_1}$ with a fixed ball of spatial frequency $B$ whose radius is $1$. 
Thus, by the H\"{o}lder's inequality and Proposition \ref{bilinear-transversal}, we observe
\begin{align*}
 \left|\int{ w_{N_0, L_0} \, 
 (P_B u_{N_1, L_1}) \, v_{N_2, L_2}
} 
dt dx d \mathbf{y} \right|
& \leq \| w_{N_0, L_0} 
 (P_B u_{N_1, L_1}) \|_{L^2} \|v_{N_2, L_2}\|_{L^2}\\
& \lesssim N_1^{-1}    (L_0 L_1)^{\frac{1}{2}} 
\|u_{N_1, L_1} \|_{L^2} \|v_{N_2, L_2} \|_{L^2}
\|w_{N_0, L_0} \|_{L^2},
\end{align*}
which completes the proof of \eqref{est02-mthm-1}. 
Similarly, if $1 \sim N_2 \ll N_0 \sim N_1$, by replacing $ u_{N_1, L_1}$ with 
$P_B u_{N_1, L_1}$, it follows from the H\"{o}lder's inequality and 
Proposition \ref{bilinear-transversal} that
\begin{align*}
 \left|\int{ w_{N_0, L_0} \, 
 (P_B u_{N_1, L_1}) \, v_{N_2, L_2}
} 
dt dx d \mathbf{y} \right|
& \leq \|  (P_B u_{N_1, L_1}) v_{N_2, L_2}\|_{L^2} \|w_{N_0, L_0}\|_{L^2}\\
& \lesssim N_1^{-1}  (L_1 L_2)^{\frac{1}{2}} 
\|u_{N_1, L_1} \|_{L^2} \|v_{N_2, L_2} \|_{L^2}
\|w_{N_0, L_0} \|_{L^2},
\end{align*}
which verifies \eqref{goal02-mthm}.

As a consequence, we can assume $L_{012}^{\max} \ll (N_{012}^{\max})^3$ and 
$1 \ll N_{012}^{\min}$ in the sequel.

\section{Proof of the key bilinear estimate: Case 1}\label{sec:proof-bil-1}

The goal of this section is to establish \eqref{goal02-mthm} under the following assumptions.
\begin{ass}\label{assumption1}$ \, $\\
(1) $L_{012}^{\max} \ll (N_{012}^{\max})^3$, $ \ $ 
(2) $1 \ll N_0 \lesssim N_1 \sim N_2$,  $ \ $ 
(3) $\textnormal{max}(|\xi_1|, |\xi_2|) \geq 2^{-5} N_1$.
\end{ass}
\begin{prop}\label{prop3.2}
Assume \textit{Assumption} \textnormal{\ref{assumption1}}. Then we get
\begin{equation}
\begin{split}
& \left|\int_{*}{  \ha{w}_{{N_0, L_0}}(\tau, \xi, \boldsymbol{\eta}) 
\ha{u}_{N_1, L_1}(\tau_1, \xi_1, \boldsymbol{\eta_1})  \ha{v}_{N_2, L_2}(\tau_2, \xi_2, \boldsymbol{\eta_2}) 
}
d\sigma_1 d\sigma_2 \right| \\
& \qquad \qquad  \lesssim N_0^{\frac{d-4}{2} +2 \e}  N_{1}^{-1-\e}   (L_0 L_1 L_2)^{\frac{1}{2}} \|\ha{u}_{N_1, L_1} \|_{L^2} \| \ha{v}_{N_2, L_2} \|_{L^2} 
\|\ha{w}_{{N_0, L_0}} \|_{L^2},\label{goal01-prop3.2}
\end{split}
\end{equation}
where $d \sigma_j = d\tau_j d \xi_j d \boldsymbol{\eta}_j $ and $*$ denotes $(\tau, \xi, \boldsymbol{\eta}) = (\tau_1 + \tau_2, \xi_1+ \xi_2, \boldsymbol{\eta_1} + \boldsymbol{\eta_2}).$
\end{prop}
\begin{rem}\label{remark3.1}
(i) Since $L_{012}^{\max} \ll (N_{012}^{\max})^3$, it is easily observed that \eqref{goal01-prop3.2} yields \eqref{goal02-mthm}.\\
(ii) By replacing the role of $\ha{w}_{{N_0, L_0}}$ with that of $\ha{v}_{N_2, L_2}$, 
we can show 
\begin{equation}
\begin{split}
& \left|\int_{*}{  \ha{v}_{N_2, L_2}(\tau, \xi, \boldsymbol{\eta}) 
\ha{u}_{N_1, L_1}(\tau_1, \xi_1, \boldsymbol{\eta_1})  \ha{w}_{{N_0, L_0}}(\tau_2, \xi_2, \boldsymbol{\eta_2}) 
}
d\sigma_1 d\sigma_2 \right| \\
& \qquad \qquad  \lesssim N_2^{\frac{d-4}{2} +2 \e}  N_{1}^{-1-\e}   (L_0 L_1 L_2)^{\frac{1}{2}} \|\ha{u}_{N_1, L_1} \|_{L^2} \| \ha{v}_{N_2, L_2} \|_{L^2} 
\|\ha{w}_{{N_0, L_0}} \|_{L^2},
\end{split}\label{goal-n2low}
\end{equation}
under the assumptions\\
(1) $L_{012}^{\max} \ll (N_{012}^{\max})^3$,  $ \ $ 
(2') $1 \ll N_2 \lesssim N_0 \sim N_1$,  $ \ $ 
(3) $\textnormal{max}(|\xi_1|, |\xi_2|) \geq 2^{-5} N_1$.\\
Clearly, \eqref{goal-n2low} gives \eqref{goal02-mthm} under the same assumptions as above.
\end{rem}
We divide the proof of Proposition \ref{prop3.2} into the three cases.\\
(Ia) $\textnormal{max}(|\boldsymbol{\eta_1}|, |\boldsymbol{\eta_2}|) \ll N_1$,\\
(Ib) $\textnormal{min}(|\boldsymbol{\eta_1}|, |\boldsymbol{\eta_2}|) \ll N_1$,  
$\textnormal{max}(|\boldsymbol{\eta_1}|, |\boldsymbol{\eta_2}|) \sim N_1$,\\
(Ic) $|\boldsymbol{\eta_1}| \sim |\boldsymbol{\eta_2}| \sim N_1$.

%%%%%%%%%%%%%%%%%%%%%%%%%%%%%%%%%%%%%
%%%%%%%%%%%%%%%%%%%%%%%%%%%%%%%%%%%%%
%%%%%%%%%%%%%%%%%%%%%%%%%%%%%%%%%%%%%
%%%%%%%%%%%%%%%%%%%%%%%%%%%%%%%%%%%%%
%%%%%%%%%%%%%%%%%%%%%%%%%%%%%%%%%%%%%
First, we consider the case (Ia): Note that the assumptions 
$1 \ll N_0 \lesssim N_1 \sim N_2$ and 
$\textnormal{max}(|\boldsymbol{\eta_1}|, |\boldsymbol{\eta_2}|) \ll N_1$ imply $|\xi_1| \sim |\xi_2|\sim N_1$.

Following \cite{BH11}, for $A\in \N$ we choose a maximally separated set $\{\omega_A^j \}_{j \in \Omega_A} $ of spherical caps of ${\mathbb S}^{d-1}$ of aperture $A$, i.e. the angle $\angle{(\theta_1,\theta_2)}$ between 
any two vectors in $\theta_1$, $\theta_2 \in \omega_A^j$ satisfies
$
\left| \angle{(\theta_1,\theta_2)} \right| \leq A^{-1}.
$
and the characteristic functions $\{ {\mathbf 1}_{\omega_A^j} \}$ satisfy
$
1 \leq \sum_{j \in \Omega_A} {\mathbf 1}_{\omega_A^j}(\theta) \leq 2^d$, for all $ \theta \in {\mathbb S}^{d-1}$.

Further, we define the function
\begin{equation*}
\alpha (j_1,j_2) = \inf \left\{ \left| \angle{( \pm \theta_1, \theta_2)} \right| : \ \theta_1 \in \omega_A^{j_1}, \ \theta_2 \in \omega_A^{j_2} \right\}
\end{equation*}
which measures the minimal angle between any two straight lines through the spherical caps $\omega_A^{j_1}$ and 
$\omega_A^{j_2}$, respectively. It is easily observed that for any fixed $j_1 \in \Omega_A$ there exist only a finite number of $j_2 \in \Omega_A$ which satisfies $\alpha (j_1,j_2) \sim A^{-1}$.
Based on the above construction, for each $j \in \Omega_A$ we define 
\begin{equation*}
{\mathcal{S}}_{j}^A = \Big\{ (\tau, \xi, \boldsymbol{\eta}) \in \R \cross (\R^d \setminus \{0\}) \, : \, 
 \frac{(\xi,\boldsymbol{\eta})}{|(\xi,\boldsymbol{\eta})|} \in \omega_A^j \Big\}
\end{equation*}
and the corresponding localization operator
\begin{equation*}
\F (R_j^A u) (\tau, \xi,\boldsymbol{\eta}) 
= \mathbf{1}_{\omega_j^A} \Big(\frac{(\xi,\boldsymbol{\eta})}{|(\xi,\boldsymbol{\eta})|} \Big) \F u (\tau , \xi,\boldsymbol{\eta}).
\end{equation*}
In addition, we define 
\begin{align*}
 C_{\textnormal{(Ia)}} = \{ (\xi,\boldsymbol{\eta} ) \in \mathbb{S}^{d-1} \, | \, |\boldsymbol{\eta}| \ll 1 . \},\qquad
& \Omega_{A, \textnormal{(Ia)}} = \{ j \in \Omega_A \, | \, \omega_A^j \cap C_{\textnormal{(Ia)}} \not= 
\emptyset.\}
\end{align*}

\begin{prop}\label{prop3.3}
Assume \textit{Assumption} \textnormal{\ref{assumption1}}. Let $A \gg 1$ be dyadic, 
$j_1$, $j_2 \in \Omega_{A, \textnormal{(Ia)}}$ and $\alpha(j_1, j_2) \lesssim A^{-1}$. 
Then we have
\begin{align}
& \left\| {\mathbf{1}}_{G_{N_1, L_1} \cap \mathcal{S}_{j_1}^A} \int \ha{v}_{N_2, L_2}|_{\mathcal{S}_{j_2}^A} 
(\tau_2,\xi_2,\boldsymbol{\eta_2}) 
\ha{w}_{N_0, L_0} (\tau_1+ \tau_2, \xi_1+\xi_2, \boldsymbol{\eta_1}+ \boldsymbol{\eta_2}) d\sigma_2 
\right\|_{L^2}
\notag \\ 
& \qquad \qquad \qquad \qquad \qquad 
\lesssim A^{-\frac{d-1}{2}} N_1^{\frac{d-3}{2}} (L_0 L_2)^{\frac{1}{2}} 
\|\ha{v}_{N_2, L_2}|_{\mathcal{S}_{j_2}^A} \|_{L^2} 
\|\ha{w}_{N_0, L_0}\|_{L^2}, \label{bilinearStrichartz-2}\\
& \left\| {\mathbf{1}}_{G_{N_2,L_2} \cap \mathcal{S}_{j_2}^A} \int 
\ha{w}_{N_0, L_0} (\tau_1+ \tau_2, \xi_1+\xi_2, \boldsymbol{\eta_1}+ \boldsymbol{\eta_2} )   \ha{u}_{N_1, L_1}|_{\mathcal{S}_{j_1}^A}(\tau_1, \xi_1, \boldsymbol{\eta_1} )  d \sigma_1 
\right\|_{L^2} \notag \\ 
& \qquad \qquad \qquad \qquad \qquad 
\lesssim A^{-\frac{d-1}{2}} N_1^{\frac{d-3}{2}} (L_0 L_1)^{\frac{1}{2}} \|\ha{w}_{N_0, L_0} \|_{L^2}
\|\ha{u}_{N_1, L_1}|_{\mathcal{S}_{j_1}^A} \|_{L^2} .\label{bilinearStrichartz-3}
\end{align}
In addition, if $| \xi | \gg A^{-1} N_1$, we get
\begin{align}
& \left\| {\mathbf{1}}_{G_{N_0, L_0}} \int  \ha{u}_{N_1, L_1}|_{\mathcal{S}_{j_1}^A}(\tau_1, \xi_1, \boldsymbol{\eta_1} ) 
\ha{v}_{N_2, L_2}|_{\mathcal{S}_{j_2}^A} (\tau- \tau_1, \xi-\xi_1, \boldsymbol{\eta}- \boldsymbol{\eta_1} ) d\sigma_1 \right\|_{L^2} \notag \\ 
& \qquad \qquad \qquad \qquad \qquad 
\lesssim A^{-\frac{d-2}{2}} N_1^{\frac{d-3}{2}} (L_1 L_2)^{\frac{1}{2}} \| \ha{u}_{N_1, L_1}|_{\mathcal{S}_{j_1}^A}\|_{L^2} 
\|\ha{v}_{N_2, L_2}|_{\mathcal{S}_{j_2}^A} \|_{L^2}.\label{bilinearStrichartz-1}
\end{align}
\end{prop}
%%%%%%%%%%%%%%%%%%%%%%%%%%%%%%%%%%%%%
%%%%%%%%%%%%%%%%%%%%%%%%%%%%%%%%%%%%%
%%%%%%%%%%%%%%%%%%%%%%%%%%%%%%%%%%%%%
%%%%%%%%%%%%%%%%%%%%%%%%%%%%%%%%%%%%%
%%%%%%%%%%%%%%%%%%%%%%%%%%%%%%%%%%%%%
\begin{proof}
First, we show \eqref{bilinearStrichartz-1}.
We observe that $L_{012}^{\max} \ll N_1^3$ yields $|\xi| \ll N_1$. Indeed, we calculate that
\begin{align*}
3 L_{012}^{\max} & \geq \bigl| \bigl( \tau-\xi(\xi^2+|\boldsymbol{\eta}|^2 ) \bigr) - 
\bigl( \tau_1-\xi_1(\xi_1^2 + |\boldsymbol{\eta_1}|^2 )\bigr) \\
& \qquad \qquad \qquad \qquad 
- \bigl( (\tau-\tau_1) - (\xi-\xi_1)((\xi-\xi_1)^2 + |\boldsymbol{\eta}-\boldsymbol{\eta_1}|^2
 )\bigr) \bigr|\\
& =  \bigl|3 \xi \xi_1(\xi-\xi_1) + \xi |\boldsymbol{\eta}|^2 - \xi_1  |\boldsymbol{\eta_1}|^2 - 
(\xi- \xi_1) |\boldsymbol{\eta} - \boldsymbol{\eta_1}|^2 \bigr|\\
& =: | \Phi_{\xi,\boldsymbol{\eta}}(\xi_1,\boldsymbol{\eta_1})|.
\end{align*}
Since $j_1$, $j_2 \in \Omega_{A, \textnormal{(Ia)}}$ and $L_{012}^{\max} \ll N_1^3$, 
 $|\xi_1| \sim |\xi-\xi_1| \sim N_1$ and 
$|\boldsymbol{\eta_1}| \ll N_1$, $|\boldsymbol{\eta} - \boldsymbol{\eta_1}| \ll N_1$, the above inequality implies $|\xi| \ll N_1$. 
By following the standard Cauchy-Schwarz argument, we get
\begin{align*}
& \left\| {\mathbf{1}}_{G_{N_0, L_0}} \int  \ha{u}_{N_1, L_1}|_{\mathcal{S}_{j_1}^A}(\tau_1, \xi_1, \boldsymbol{\eta_1} ) 
\ha{v}_{N_2, L_2}|_{\mathcal{S}_{j_2}^A} (\tau- \tau_1, \xi-\xi_1, \boldsymbol{\eta}- \boldsymbol{\eta_1} ) d\sigma_1 \right\|_{L^2} \notag \\ 
\leq & \sup_{(\tau, \xi, \boldsymbol{\eta} ) \in G_{N_0, L_0}} |E(\tau, \xi, \boldsymbol{\eta} )|^{1/2} 
 \left\| \left| 
 \ha{u}_{N_1, L_1}|_{\mathcal{S}_{j_1}^A}\right|^2 * 
\left| \ha{v}_{N_2, L_2}|_{\mathcal{S}_{j_2}^A}  \right|^2
\right\|_{L^1}^{1/2}\\
\leq & \sup_{(\tau, \xi, \boldsymbol{\eta} ) \in G_{N_0, L_0}} |E(\tau, \xi, \boldsymbol{\eta} )|^{1/2} 
\|\ha{u}_{N_1, L_1}|_{\mathcal{S}_{j_1}^A} \|_{L^2} 
\| \ha{v}_{N_2, L_2}|_{\mathcal{S}_{j_2}^A} \|_{L^2},
\end{align*}
where $E(\tau, \xi, \boldsymbol{\eta} ) \subset \R^{d+1}$ is defined by
\begin{equation*}
E(\tau, \xi, \boldsymbol{\eta} ) = \{ (\tau_1, \xi_1, \boldsymbol{\eta_1} ) \in G_{N_1, L_1} \cap {\mathcal{S}_{j_1}^A} 
\, | \, (\tau-\tau_1, \xi- \xi_1, \boldsymbol{\eta}-\boldsymbol{\eta_1} ) \in G_{N_2,L_2} \cap  {\mathcal{S}_{j_2}^A}  \}.
\end{equation*}
Thus, it suffices to show
\begin{equation}
\sup_{(\tau, \xi, \boldsymbol{\eta} ) \in G_{N_0, L_0}} |E(\tau, \xi, \boldsymbol{\eta} )| 
\lesssim A^{-(d-2)} N_1^{d-3}L_1 L_2.\label{est01-prop3.3}
\end{equation}
Clearly, for fixed $(\xi_1, \boldsymbol{\eta_1} )$, it holds
\begin{equation}
\sup_{(\tau, \xi, \boldsymbol{\eta} ) \in G_{N_0, L_0}}
| \{ \tau_1 \, | \, (\tau_1, \xi_1, \boldsymbol{\eta_1} ) \in E(\tau, \xi, \boldsymbol{\eta} ) \}| 
\lesssim \min(L_1, L_2).\label{est02-prop3.3}
\end{equation}
A simple calculation yields
\begin{align}
& |\partial_{\xi_1} \Phi_{\xi,\boldsymbol{\eta} }(\xi_1,\boldsymbol{\eta_1} )|  = 
|3 \xi (\xi- 2 \xi_1) + \boldsymbol{\eta} \cdot (\boldsymbol{\eta} - 2 \boldsymbol{\eta_1})|, 
\label{est02a-prop3.3}\\
& |\nabla_{\boldsymbol{\eta_1}} \Phi_{\xi,\boldsymbol{\eta} }(\xi_1,\boldsymbol{\eta_1} )|  = 
2|(\xi-\xi_1) \boldsymbol{\eta} - \xi \boldsymbol{\eta_1}| \label{est02b-prop3.3}.
\end{align}
Let $r_1=|(\xi_1,\boldsymbol{\eta_1} )|$, $r_2=|(\xi-\xi_1,\boldsymbol{\eta}-\boldsymbol{\eta_1} )|$ and $(\theta_1, \boldsymbol{\theta_1'})= (\xi_1/r_1, \boldsymbol{\eta_1}/r_1)$, 
$(\theta_2, \boldsymbol{\theta_2'}) =
 ((\xi-\xi_1)/r_2, (\boldsymbol{\eta}-\boldsymbol{\eta_1})/r_2 )\in \mathbb{S}^{d-1}$. 
Since 
$(\theta_1, \boldsymbol{\theta_1'}) \times (\theta_2, \boldsymbol{\theta_2'}) 
\in \omega_A^{j_1} \times \omega_A^{j_2}$ with $\alpha(j_1,j_2) \lesssim A^{-1}$ and $|\xi| \ll N_1$, 
we have $|(\theta_1, \boldsymbol{\theta_1'}) + (\theta_2, \boldsymbol{\theta_2'}) | \lesssim A^{-1}$. 
Furthermore, the assumption $(\theta_j, \boldsymbol{\theta_j'}) \in C_{\textnormal{(Ia)}}$ implies 
$|\theta_1| \sim |\theta_2| \sim 1$ and 
$ \max (|\boldsymbol{\theta_1'}|, |\boldsymbol{\theta_2'}|) \ll 1$. 
Therefore, we deduce from the assumption $A^{-1} N_1 \ll |\xi|$ that
\begin{align*}
|\boldsymbol{\eta}|& = |r_1\boldsymbol{\theta_1'}+ r_2 \boldsymbol{\theta_2'}| \leq  |r_1-r_2| |\boldsymbol{\theta_1'}| + r_2| \boldsymbol{\theta_1'}+ \boldsymbol{\theta_2'}|\\
& \leq |r_1 \theta_1 - r_2 \theta_1| + r_2 | \boldsymbol{\theta_1'}+ \boldsymbol{\theta_2'}|
 \leq |r_1 \theta_1 + r_2 \theta_2| + 2r_2 |  (\theta_1,\boldsymbol{\theta_1'})+ 
(\theta_2, \boldsymbol{\theta_2'}) | \leq 2 |\xi|.
\end{align*}
Hence, if $ (\tau_1, \xi_1, \boldsymbol{\eta_1} ) \in E(\tau, \xi, \boldsymbol{\eta} )$, 
the above inequality and \eqref{est02a-prop3.3}, \eqref{est02b-prop3.3} yield
\begin{align*}
 |\partial_{r_1} \Phi_{\xi,\boldsymbol{\eta} }(\xi_1,\boldsymbol{\eta_1} )| & =
|(\theta_1\partial_{\xi_1} + \boldsymbol{\theta_1'} \cdot  \nabla_{\boldsymbol{\eta_1}}) \Phi_{\xi,\boldsymbol{\eta} }(\xi_1,\boldsymbol{\eta_1} ) | \\
& \geq |\theta_1| |\partial_{\xi_1} \Phi_{\xi,\boldsymbol{\eta} }(\xi_1,\boldsymbol{\eta_1} )| - 
|\boldsymbol{\theta_1'}| |\nabla_{\boldsymbol{\eta_1}} \Phi_{\xi,\boldsymbol{\eta} }(\xi_1,\boldsymbol{\eta_1} ) | \\
& \gtrsim  |\xi| N_1 \geq A^{-1} N_1^2,
\end{align*}
which implies that $r_1$ is confined to a set of measure 
$\sim A \max(L_1,L_2)/N_1^2$ for fixed $\theta_1$, $\boldsymbol{\theta_1'}$ since, 
as we saw above, it holds that
\begin{align*}
\max (L_1, L_2) & \gtrsim \bigl| 
\bigl( \tau_1-\xi_1(\xi_1^2 + |\boldsymbol{\eta_1}|^2 )\bigr) 
+ \bigl( (\tau-\tau_1) - (\xi-\xi_1)((\xi-\xi_1)^2 + |\boldsymbol{\eta}-\boldsymbol{\eta_1}|^2
 )\bigr) \bigr|\\
& = |\bigl( \tau-\xi(\xi^2+|\boldsymbol{\eta}|^2 ) \bigr) + \Phi_{\xi,\boldsymbol{\eta} }(\xi_1,\boldsymbol{\eta_1} )|.
\end{align*}
Therefore, we get
\begin{align*}
&  | \{ (\xi_1, \boldsymbol{\eta_1} ) \ | \ (\tau_1, \xi_1, \boldsymbol{\eta_1} ) \in E(\tau, \xi,\boldsymbol{\eta} ) \} |\\
= &  \int_{\boldsymbol{\theta_1'}}\int_{\theta_1} \int_{r_1} {{\mathbf{1}}}_{E(\tau,\xi,\boldsymbol{\eta} )}
 (r_1, \theta_1, \boldsymbol{\theta_1'}) r_1^{d-1}  d r_1 d \theta_1d \boldsymbol{\theta_1'} \\
\lesssim & A^{-(d-2)} N_1^{d-3}\max (L_1, L_2).
\end{align*}
This and \eqref{est02-prop3.3} give \eqref{est01-prop3.3}.

Next, we consider \eqref{bilinearStrichartz-2}. By the same argument it suffices to prove
\begin{equation}
\sup_{(\tau_1, \xi_1, \boldsymbol{\eta}_1 ) \in G_{N_1, L_1}} |E_1(\tau_1, \xi_1, \boldsymbol{\eta}_1 )| 
\lesssim A^{-(d-1)} N_1^{d-3}L_0 L_2,\label{est11-prop3.3}
\end{equation}
where
\[
E_1(\tau_1, \xi_1, \boldsymbol{\eta}_1 )=\big\{ (\tau_2,\xi_2,\boldsymbol{\eta}_2)\in G_{N_2,L_2}\cap \mathcal{S}_{j_2}^A \, | \,  (\tau_1+\tau_2,\xi_1+\xi_2,\boldsymbol{\eta}_1+\boldsymbol{\eta}_2)\in G_{N_0,L_0}\big\}.
\]
As above, we have
\begin{align*}
3 L_{012}^{\max} & \geq  \bigl|3 (\xi_1+\xi_2) \xi_1\xi_2 + (\xi_1+\xi_2) |\boldsymbol{\eta}_1+\boldsymbol{\eta}_2|^2 - \xi_1  |\boldsymbol{\eta_1}|^2 - 
\xi_2 |\boldsymbol{\eta_2}|^2 \bigr|\\
                 & =: | \Phi^{(1)}_{\xi_1,\boldsymbol{\eta_1}}(\xi_2,\boldsymbol{\eta_2})|.
     \end{align*}
     We compute
     $
|\partial_{\xi_2}\Phi^{(1)}_{\xi_1,\boldsymbol{\eta_1}}(\xi_2,\boldsymbol{\eta_2})|\gtrsim N_1^2$,  $|\nabla_{\boldsymbol{\eta_1}}\Phi^{(1)}_{\xi_1,\boldsymbol{\eta_1}}(\xi_2,\boldsymbol{\eta_2})|\ll N_1^2,
     $
and with the same notation for polar coordinates as above , we obtain
     \[
|\partial_{r_2} \Phi^{(1)}_{\xi_1,\boldsymbol{\eta_1}}(\xi_2,\boldsymbol{\eta_2})|\gtrsim N_1^2.
     \]
 If $(\tau_2,\xi_2,\boldsymbol{\eta}_2) \in E_1(\tau_1, \xi_1, \boldsymbol{\eta}_1 ) $, then for a fixed angular part $(\theta_2,\boldsymbol{\theta_2'})$ of $(\xi_2,\boldsymbol{\eta}_2)$, the radial direction $r_2$ is confined to an interval of length $\lesssim \max(L_0,L_2)/N_1^2$. By the analogue of \eqref{est02-prop3.3} we conclude \eqref{est11-prop3.3} and the proof of \eqref{bilinearStrichartz-2} is complete.
     
Finally, \eqref{bilinearStrichartz-3} follows by symmetry.
\end{proof}
%%%%%%%%%%%%%%%%%%%%%%%%%%%%%%%%%%%%%
%%%%%%%%%%%%%%%%%%%%%%%%%%%%%%%%%%%%%
%%%%%%%%%%%%%%%%%%%%%%%%%%%%%%%%%%%%%
%%%%%%%%%%%%%%%%%%%%%%%%%%%%%%%%%%%%%
%%%%%%%%%%%%%%%%%%%%%%%%%%%%%%%%%%%%%
\begin{prop}\label{prop3.4}
Assume \textit{Assumption} \textnormal{\ref{assumption1}}. 
Let $A \gg 1$ be dyadic, 
$j_1$, $j_2 \in \Omega_{A, \textnormal{(Ia)}}$, $\alpha(j_1, j_2) \sim A^{-1}$ and 
$|\xi_1+\xi_2| \lesssim A^{-1} N_1$. 
Then we get
\begin{equation}
\begin{split}
& \left|\int_{*}{  \ha{w}_{{N_0, L_0}}(\tau, \xi, \boldsymbol{\eta} ) 
\ha{u}_{N_1, L_1}|_{\mathcal{S}_{j_1}^A}(\tau_1, \xi_1, \boldsymbol{\eta_1} )  
\ha{v}_{N_2, L_2}|_{\mathcal{S}_{j_2}^A}(\tau_2, \xi_2, \boldsymbol{\eta_2} ) 
}
d\sigma_1 d\sigma_2 \right| \\
& \qquad \qquad  \lesssim  A^{-\frac{d-3}{2}} 
N_1^{\frac{d-6}{2}} (L_0 L_1 L_2)^{\frac{1}{2}} \|\ha{u}_{N_1, L_1} \|_{L^2} \| \ha{v}_{N_2, L_2} \|_{L^2} 
\|\ha{w}_{{N_0, L_0}} \|_{L^2}.\label{est01-prop3.4}
\end{split}
\end{equation}
\end{prop}
\begin{proof}
In the proof of Proposition \ref{prop3.3} we proved
$|\boldsymbol{\eta}| \lesssim |\xi|+A^{-1} N_1 \lesssim A^{-1} N_1$. 
Thus, by almost orthogonality, we may assume that
$\boldsymbol{\eta_j}$ is confined to a ball whose radius is comparable
to $A^{-1} N_1$. We write $\boldsymbol{\eta_j}=(\eta_j, \boldsymbol{\eta_j'})$.
Further, without loss of generality, we can assume $\max(|\boldsymbol{\eta_1'}|,|\boldsymbol{\eta_2'}|) \lesssim A^{-1} N_1$. 
Indeed, we can apply a rotation in the $\boldsymbol{\eta}$-subspace, since the phase function is invariant under such rotations.

Since $\alpha(j_1, j_2) \sim A^{-1}$, we have 
$\max(|\xi_1 \eta_2-\xi_2 \eta_1|, |\xi_1 \boldsymbol{\eta_2'} - \xi_2 \boldsymbol{\eta_1'}|) \sim A^{-1} N_1^2$. 
We first consider the case $|\xi_1 \eta_2-\xi_2 \eta_1| \sim A^{-1} N_1^2$. 
For fixed $\boldsymbol{\eta_1'}$, $\boldsymbol{\eta_2'}$, we will show that
\begin{equation}
\begin{split}
& \left|\int_{\hat{*}}{  \ha{w}_{{N_0, L_0}}(\tau, \xi, \boldsymbol{\eta} ) 
\ha{u}_{N_1, L_1}|_{\mathcal{S}_{j_1}^A}(\tau_1, \xi_1, \boldsymbol{\eta_1} )  
\ha{v}_{N_2, L_2}|_{\mathcal{S}_{j_2}^A} (\tau_2, \xi_2, \boldsymbol{\eta_2} ) 
}
d\hat{\sigma}_1 d\hat{\sigma}_2 \right| \\
& \qquad \lesssim A^{\frac{1}{2}} N_{1}^{-2}   (L_0 L_1 L_2)^{\frac{1}{2}} 
\|\ha{u}_{N_1, L_1} (\boldsymbol{\eta_1'}) \|_{L^2_{\tau \xi \eta}} 
\| \ha{v}_{N_2, L_2} (\boldsymbol{\eta_2'})\|_{L^2_{\tau \xi \eta}} 
\|\ha{w}_{{N_0, L_0}}(\boldsymbol{\eta'}) \|_{L^2_{\tau \xi \eta}} ,\label{est02-prop3.4}
\end{split}
\end{equation}
where $d \hat{\sigma}_j = d\tau_j d \xi_j d \eta_j$ and 
$\hat{*}$ denotes $(\tau,\xi, \eta) = (\tau_1 + \tau_2,\xi_1+\xi_2, \eta_1 + \eta_2).$ 
\eqref{est02-prop3.4} implies \eqref{est01-prop3.4} because 
\begin{align*}
& \left|\int_{*}{  \ha{w}_{{N_0, L_0}}(\tau, \xi, \boldsymbol{\eta} ) 
\ha{u}_{N_1, L_1}|_{\mathcal{S}_{j_1}^A}(\tau_1, \xi_1, \boldsymbol{\eta_1} )  
\ha{v}_{N_2, L_2}|_{\mathcal{S}_{j_2}^A} (\tau_2, \xi_2, \boldsymbol{\eta_2} ) 
}
d\sigma_1 d\sigma_2 \right| \\
\lesssim & \int \left|\int_{\hat{*}}{  \ha{w}_{{N_0, L_0}}(\tau, \xi, \boldsymbol{\eta} ) 
\ha{u}_{N_1, L_1}|_{\mathcal{S}_{j_1}^A}(\tau_1, \xi_1, \boldsymbol{\eta_1} )  
\ha{v}_{N_2, L_2}|_{\mathcal{S}_{j_2}^A} (\tau_2, \xi_2, \boldsymbol{\eta_2} ) 
}
d\hat{\sigma}_1 d\hat{\sigma}_2 \right| d \boldsymbol{\eta_1'} d \boldsymbol{\eta_2'}\\
\lesssim & \, A^{\frac{1}{2}} N_{1}^{-2}   (L_0 L_1 L_2)^{\frac{1}{2}} \\
& \qquad \times
\int \|\ha{u}_{N_1, L_1}|_{\mathcal{S}_{j_1}^A} (\boldsymbol{\eta_1'}) \|_{L^2_{\tau \xi \eta}} 
\| \ha{v}_{N_2, L_2}|_{\mathcal{S}_{j_2}^A} (\boldsymbol{\eta_2'})\|_{L^2_{\tau \xi \eta}} 
\|\ha{w}_{{N_0, L_0}}(\boldsymbol{\eta_1'}+\boldsymbol{\eta_2'}) \|_{L^2_{\tau \xi \eta}}d \boldsymbol{\eta_1'} d \boldsymbol{\eta_2'}\\
\lesssim & \, 
A^{-\frac{d-3}{2}} 
N_1^{\frac{d-6}{2}}  (L_0 L_1 L_2)^{\frac{1}{2}} \|\ha{u}_{N_1, L_1} \|_{L^2} \| \ha{v}_{N_2, L_2} \|_{L^2} 
\|\ha{w}_{{N_0, L_0}} \|_{L^2},
\end{align*}
where we used the
support conditions in the last step.

Now, we prove \eqref{est02-prop3.4}. We use the shorthand notation
\begin{align*}
& f_{\boldsymbol{\eta_1'}}(\tau_1,\xi_1, \eta_1) = 
\ha{u}_{N_1, L_1}|_{\mathcal{S}_{j_1}^A}  (\tau_1, \xi_1,
                 \boldsymbol{\eta_1} ),  \quad g_{\boldsymbol{\eta_2'}}(\tau_2, \xi_2, \eta_2) = 
\ha{v}_{N_2, L_2}|_{\mathcal{S}_{j_2}^A}  (\tau_2, \xi_2, \boldsymbol{\eta_2} ), \\
& h_{\boldsymbol{\eta'}} (\tau, \xi, \eta) = \ha{w}_{{N_0, L_0}}(\tau, \xi, \boldsymbol{\eta} ),
\end{align*}
and show
\begin{equation}
\begin{split}
& \left|\int_{\hat{*}}{   h_{\boldsymbol{\eta'}} (\tau, \xi, \eta)
f_{\boldsymbol{\eta_1'}}(\tau_1,\xi_1, \eta_1) 
g_{\boldsymbol{\eta_2'}}(\tau_2, \xi_2, \eta_2)  
}
d\hat{\sigma}_1 d\hat{\sigma}_2 \right| \\
& \qquad \qquad \qquad \lesssim A^{\frac{1}{2}} N_{1}^{-2}   (L_0 L_1 L_2)^{\frac{1}{2}} 
\|f_{\boldsymbol{\eta_1'}} \|_{L^2_{\tau \xi \eta}} 
\| g_{\boldsymbol{\eta_2'}}\|_{L^2_{\tau \xi \eta}} 
\| h_{\boldsymbol{\eta'}} \|_{L^2_{\tau \xi \eta}}.\label{est03-prop3.4}
\end{split}
\end{equation}
Applying the transformation $\tau_1 = \xi_1(\xi_1^2 + \eta_1^2 +|\boldsymbol{\eta_1'}|^2) + c_1$ and 
$\tau_2 = \xi_2(\xi_2^2 + \eta_2^2 +|{\boldsymbol{\eta_2'}}|^2) + c_2$ and Fubini's theorem, we find that it suffices to prove
\begin{align}
& \left| \int h_{\boldsymbol{\eta'}} (\phi_{\boldsymbol{\eta_1'},c_1} (\xi_1, \eta_1) + \phi_{\boldsymbol{\eta_2'},c_2} (\xi_2, \eta_2))  
f_{\boldsymbol{\eta_1'}} (\phi_{\boldsymbol{\eta_1'},c_1} (\xi_1, \eta_1) ) 
g_{\boldsymbol{\eta_2'}} (\phi_{\boldsymbol{\eta_2'},c_2}(\xi_2, \eta_2)) 
d \xi_1d \eta_1 d\xi_2 d\eta_2 \right| \notag\\
& \qquad \qquad \qquad \qquad \qquad  \lesssim  A^{\frac{1}{2}}N_1^{-2} 
\| f_{\boldsymbol{\eta_1'}} \circ \phi_{\boldsymbol{\eta_1'},c_1}\|_{L_{\xi \eta}^2} \|g_{\boldsymbol{\eta_2'}} 
\circ \phi_{\boldsymbol{\eta_2'},c_2} \|_{L_{\xi \eta}^2} 
\|h_{\boldsymbol{\eta'}} \|_{L_{\tau \xi \eta }^2}, \label{est04-prop3.4}
\end{align}
where $h_{\boldsymbol{\eta'}} (\tau, \xi, \eta)$ 
is supported in $c_0 \leq \tau - \xi(\xi^2 + \eta^2 + |{\boldsymbol{\eta'}}|^2) \leq c_0 +1$ and 
\begin{equation*}
\phi_{\boldsymbol{\eta_j'}, c_j} (\xi_j,\eta_j) = (\xi_j(\xi_j^2 + \eta_j^2 + |\boldsymbol{{\eta_j'}}|^2) + c_j, \, \xi_j, \, \eta_j) \quad \textnormal{for} \ j=1,2.
\end{equation*}
We use the scaling $(\tau, \, \xi, \, \eta) \to (N_1^3 \tau , \, N_1 \xi, \, N_1 \eta)$ to define
\begin{align*}
 &\tilde{f}_{\boldsymbol{\eta_1'}} (\tau_1, \xi_1 , \eta_1)  = f_{\boldsymbol{\eta_1'}} (N_1^3 \tau_1 , N_1 \xi_1, N_1 \eta_1), \quad
 \tilde{g}_{\boldsymbol{\eta_2'}} (\tau_2, \xi_2, \eta_2)  = g_{\boldsymbol{\eta_2'}} (N_1^3 \tau_2, N_1 \xi_2, N_1 \eta_2), \\
& \tilde{h}_{\boldsymbol{\eta'}} (\tau, \xi, \eta)  = h_{\boldsymbol{\eta'}} (N_1^3 \tau, N_1 \xi, N_1 \eta). 
\end{align*}
Let $\boldsymbol{\tilde{\eta_j}} = N_1^{-1} \boldsymbol{\eta_j'}$, 
$\tilde{c_j} = N_1^{-3} c_j$. The inequality (\ref{est04-prop3.4}) reduces to
\begin{equation*}
\begin{split}
& \left| \int \tilde{h}_{\boldsymbol{\eta'}} (\phi_{\boldsymbol{\tilde{\eta}_1}, \tilde{c}_1} (\xi_1, \eta_1) + 
\phi_{\boldsymbol{\tilde{\eta}_2}, \tilde{c}_2} (\xi_2, \eta_2))  \tilde{f}_{\boldsymbol{\eta_1'}}
 (\phi_{\boldsymbol{\tilde{\eta}_1}, \tilde{c}_1} (\xi_1, \eta_1) ) \tilde{g}_{\boldsymbol{\eta_2'}} (\phi_{\boldsymbol{\boldsymbol{\tilde{\eta}_2}}, \tilde{c}_2}
(\xi_2, \eta_2)) d \xi_1d \eta_1 d\xi_2 d\eta_2 \right| \\
&\qquad \qquad \qquad \qquad \qquad \qquad  \lesssim  A^{\frac{1}{2}}N_1^{-\frac{3}{2}} 
\| \tilde{f}_{\boldsymbol{\eta_1'}}
 \circ \phi_{\boldsymbol{\tilde{\eta}_1}, \tilde{c}_1} \|_{L_{\xi \eta}^2} 
\| \tilde{g}_{\boldsymbol{\eta_2'}}  \circ \phi_{\boldsymbol{\tilde{\eta}_2}, \tilde{c}_2} \|_{L_{\xi \eta}^2} 
\|\tilde{h}_{\boldsymbol{\eta'}} \|_{L_{\tau \xi \eta}^2},
\end{split}
\end{equation*}
Note that $|\boldsymbol{\tilde{\eta_j}}| \lesssim A^{-1}$ and we easily see $|\xi_j| \sim 1$, $|\eta_j| \ll 1$ 
and $|\xi_1 \eta_2-\xi_2 \eta_1| \sim A^{-1}$ if 
$(\xi_1,\eta_1) \in \supp (\tilde{f}_{\boldsymbol{\eta_1'}}\circ \phi_{\boldsymbol{\tilde{\eta}_1}, \tilde{c}_1})$, $(\xi_2,\eta_2) \in \supp (\tilde{g}_{\boldsymbol{\eta_2'}}\circ \phi_{\boldsymbol{\tilde{\eta}_2}, \tilde{c}_2})$. 
Therefore, letting $\boldsymbol{\tilde{\eta}} = N_1^{-1}\boldsymbol{\eta'}$, we can assume that $\tilde{h}_{\boldsymbol{\eta'}}$ is supported in $S_3 (N_1^{-3})$ where 
\begin{equation*}
S_3 (N_1^{-3}) = \left\{ (\tau, \xi, \eta) 
\ \Bigl| \  |(\xi, \eta) | \lesssim A^{-1}, \ \frac{c_0}{N_1^{3}}  \leq \tau - \xi(\xi^2 + \eta^2 + 
|\boldsymbol{\tilde{\eta}}|^2) \leq  \frac{c_0+1}{N_1^{3}} \right\}.
\end{equation*}
By density and duality, it suffices to show that for continuous 
$\tilde{f}_{\boldsymbol{\eta_1'}}$ and $\tilde{g}_{\boldsymbol{\eta_2'}}$ it holds that
\begin{equation}
\| \tilde{f}_{\boldsymbol{\eta_1'}} |_{S_1} * \tilde{g}_{\boldsymbol{\eta_2'}} |_{S_2} \|_{L^2(S_3 (N_1^{-3}))} 
\lesssim A^{\frac{1}{2}} N_1^{-\frac{3}{2}} 
\| \tilde{f}_{\boldsymbol{\eta_1'}} \|_{L^2(S_1)} \| \tilde{g}_{\boldsymbol{\eta_2'}} \|_{L^2(S_2)}\label{est05-prop3.4}
\end{equation}
where $S_1$, $S_2$ denote the following surfaces 
\begin{align*}
S_1 =&  \{ \phi_{\boldsymbol{\tilde{\eta}_1},\tilde{c}_1} (\xi_1, \eta_1) \in \R^3 \ | \ 
(\xi_1, \eta_1) \in \supp (\tilde{f}_{\boldsymbol{\eta_1'}}\circ \phi_{\boldsymbol{\tilde{\eta}_1}, \tilde{c}_1})\}, \\
S_2 =&  \{ \phi_{\boldsymbol{\tilde{\eta}_2}, \tilde{c}_2}(\xi_2,\eta_2) \in \R^3 \ | \ (\xi_2, \eta_2) \in \supp (\tilde{g}_{\boldsymbol{\eta_2'}}\circ \phi_{\boldsymbol{\tilde{\eta}_2}, \tilde{c}_2}) \}.
\end{align*}
(\ref{est05-prop3.4}) is immediately obtained by the following.
\begin{equation}
\| \tilde{f}_{\boldsymbol{\eta_1'}} |_{S_1} * \tilde{g}_{\boldsymbol{\eta_2'}} |_{S_2} \|_{L^2(S_3)} \lesssim A^{\frac{1}{2}}  
\| \tilde{f}_{\boldsymbol{\eta_1'}} \|_{L^2(S_1)} \| \tilde{g}_{\boldsymbol{\eta_2'}} \|_{L^2(S_2)}\label{est06-prop3.4}
\end{equation}
where 
\begin{equation*}
S_3 = \left\{ (\psi_{\boldsymbol{\tilde{\eta}}} (\xi,\eta), \xi, \eta) \in \R^3  \ | \  |(\xi, \eta)| \lesssim A^{-1},
\  \psi_{\boldsymbol{\tilde{\eta}}} (\xi,\eta) = \xi(\xi^2 + \eta^2 + |\boldsymbol{\tilde{\eta}}|^2) + \frac{c_0'}{N_1^{3}} \right\},
\end{equation*} 
for any fixed $c_0' \in [c_0, \, c_0 +1]$. Since $\textnormal{diam} (S_3) \lesssim A^{-1}$, by the almost 
orthogonality and harmless decompositions, we may assume
\begin{equation}
\textnormal{diam} (S_i) \ll A^{-1} \qquad \textnormal{for} \ \, i=1,2,3.\label{diam-prop3.4}
\end{equation}
For any $\lambda_i \in S_i$, there exist $(\xi_1, \eta_1)$, $(\xi_2, \eta_2)$, $(\xi, \eta)$ such that
\begin{equation*}
\lambda_1=\phi_{\boldsymbol{\tilde{\eta}_1},\tilde{c}_1}  (\xi_1, \eta_1), \quad 
\lambda_2 = \phi_{\boldsymbol{\tilde{\eta}_2}, \tilde{c}_2} (\xi_2,\eta_2), \quad
\lambda_3 = (\psi_{\boldsymbol{\tilde{\eta}}} (\xi,\eta), \xi,\eta),
\end{equation*}
and the unit normals ${\mathfrak{n}}_i$ on $\lambda_i$ are written as
\[
{\mathfrak{n}}_i(\lambda_i) = 
\frac{1}{\sqrt{1+ (3 \xi_i^2 + \eta_i^2+|\boldsymbol{\tilde{\eta}_i}|^2)^2 + 4\xi_i^2 \eta_i^2}} 
\left(-1, \ 3 \xi_i^2 + \eta_i^2+|\boldsymbol{\tilde{\eta}_i}|^2, \ 2 \xi_i \eta_i \right)
\]
for $i=1,2$, and the same for $\mathfrak{n}_3(\lambda_3)$.
Clearly, the surfaces $S_1$, $S_2$, $S_3$ satisfy the following 
H\"{o}lder condition.
\begin{equation}
\sup_{\lambda_i, \widehat{\lambda}_i \in S_i} \frac{|\mathfrak{n}_i(\lambda_i) - 
\mathfrak{n}_i(\widehat{\lambda}_i)|}{|\lambda_i - \widehat{\lambda}_i|}
+ \frac{|\mathfrak{n}_i(\lambda_i) (\lambda_i -
  \widehat{\lambda}_i)|}{|\lambda_i - \widehat{\lambda}_i|^2} \lesssim
1.\label{normals00-prop3.4}
\end{equation}
We may assume that there exist $(\widehat{\xi}_1, \widehat{\eta}_1)$, $(\widehat{\xi}_2, \widehat{\eta}_2)$, $(\widehat{\xi}, \widehat{\eta})$ such that
\begin{align*}
& (\widehat{\xi}_1, \widehat{\eta}_1) + (\widehat{\xi}_2, \widehat{\eta}_2) = (\widehat{\xi}, \widehat{\eta}), \\  
\phi_{\boldsymbol{\tilde{\eta}_1},\tilde{c}_1} (\widehat{\xi}_1, \widehat{\eta}_1) & \in S_1, \ \ 
\phi_{\boldsymbol{\tilde{\eta}_2}, \tilde{c}_2} (\widehat{\xi}_2, \widehat{\eta}_2)\in S_2, \ \ 
(\psi_{\boldsymbol{\tilde{\eta}}} (\widehat{\xi},\widehat{\eta}), \widehat{\xi}, \widehat{\eta}) \in S_3,
\end{align*}
otherwise the left-hand side of \eqref{est06-prop3.4} vanishes. 
Let $\widehat{\lambda}_1 = \phi_{\boldsymbol{\tilde{\eta}_1},\tilde{c}_1} (\widehat{\xi}_1, \widehat{\eta}_1)$, 
$\widehat{\lambda}_2 = \phi_{\boldsymbol{\tilde{\eta}_2}, \tilde{c}_2} (\widehat{\xi}_2, \widehat{\eta}_2)$, 
$\widehat{\lambda}_3 = (\psi_{\boldsymbol{\tilde{\eta}}} (\widehat{\xi},\widehat{\eta}), \widehat{\xi},\widehat{\eta})$. 
For any $i=1,2,3$ and $\lambda_i$, $\widehat{\lambda}_i \in S_i$ 
 \eqref{diam-prop3.4} implies that
\begin{equation}
|{\mathfrak{n}}_i(\lambda_i) - {\mathfrak{n}}_i(\widehat{\lambda}_i)| \ll A^{-1}.\label{normal01-prop3.4}
\end{equation}
From \eqref{diam-prop3.4} and \eqref{normals00-prop3.4}, 
once the following transversality condition \begin{equation}
A^{-1} \lesssim |\textnormal{det} N(\lambda_1, \lambda_2, \lambda_3)| \quad 
\textnormal{for any} \ \lambda_i \in S_i.\label{trans-prop3.4}
\end{equation}
is verified, we obtain the desired estimate \eqref{est06-prop3.4} by
applying the nonlinear Loomis-Whitney inequality from Proposition \ref{prop2.3}.
Using $|\widehat{\xi_j}| \sim 1$, $|\widehat{\eta_j}| \ll 1$, $|\boldsymbol{\tilde{\eta_j}}| \lesssim A^{-1}$ and 
$|\widehat{\xi}_1 \widehat{\eta}_2 - \widehat{\xi}_2 \widehat{\eta}_1 | \sim A^{-1}$,
 we compute
\begin{align*}
|\textnormal{det} N(\widehat{\lambda}_1, \widehat{\lambda}_2, \widehat{\lambda}_3)|  \gtrsim{} & 
 \left|\textnormal{det}
\begin{pmatrix}
-1 & -1 & - 1 \\
3 {\widehat{\xi}_1}^2 + {\widehat{\eta}_1}^2+|\boldsymbol{\tilde{\eta}_1}|^2  & 3 {\widehat{\xi}_2}^2 + {\widehat{\eta}_2}^2+|\boldsymbol{\tilde{\eta}_2}|^2  
& 3 {\widehat{\xi}}^2 + {\widehat{\eta}}^2+|\boldsymbol{\tilde{\eta}}|^2 \\
2 \widehat{\xi}_1 \widehat{\eta}_1   & 2 \widehat{\xi}_2 \widehat{\eta}_2  & 2 \widehat{\xi} \widehat{\eta}
\end{pmatrix} \right| \notag \\
\gtrsim{} &  \bigl| (\widehat{\xi}_1 \widehat{\eta}_2 - \widehat{\xi}_2 \widehat{\eta}_1)\bigl( 3 ({\widehat{\xi}_1}^2 + \widehat{\xi}_1 \widehat{\xi}_2 + {\widehat{\xi}_2}^2 ) - 
({\widehat{\eta}_1}^2 + \widehat{\eta}_1 \widehat{\eta}_2 + {\widehat{\eta}_2}^2) \bigr)\\
&{}+ (\widehat{\xi}_1 \widehat{\eta}_2 + \widehat{\xi}_2(\widehat{\eta}_1 +\widehat{\eta}_2))|\boldsymbol{\tilde{\eta}_1}|^2 - 2 (\widehat{\xi}_1 \widehat{\eta}_1 - \widehat{\xi}_2 \widehat{\eta}_2) 
\boldsymbol{\tilde{\eta}_1} \cdot \boldsymbol{\tilde{\eta}_2} - (\widehat{\xi}_1(\widehat{\eta}_1 + \widehat{\eta}_2) + \widehat{\xi}_2 \widehat{\eta}_1) |\boldsymbol{\tilde{\eta}_2}|^2 \bigr|\\
\gtrsim{} & A^{-1},
\end{align*}
which implies \eqref{trans-prop3.4} due to
\eqref{normal01-prop3.4}. In the above computation, we used
multi-linearity in the columns to separate the contributions of
$\boldsymbol{\tilde{\eta}_1}$, $\boldsymbol{\tilde{\eta}_2}$ and
$\boldsymbol{\tilde{\eta}}$ from the main one corresponding to the first
line above.

Next, we treat the case $|\xi_1 \boldsymbol{\eta_2'} - \xi_2 \boldsymbol{\eta_1'}| \sim A^{-1} N_1^2$. 
Without loss of generality, we assume $|\xi_1 \eta_2' - \xi_2 \eta_1'| \sim A^{-1} N_1^2$ where 
$\eta_1'$ and $\eta_2'$ are the first components of $\boldsymbol{\eta_1'}$ and 
$\boldsymbol{\eta_2'}$, respectively. 
By replacing the role of $(\eta_1,\eta_2)$ with that of $(\eta_1',\eta_2')$ in the
proof in the previous case, it suffices to show
\begin{equation}
\begin{split}
& \left|\int_{\bar{*}}{  \ha{w}_{{N_0, L_0}}(\tau, \xi, \boldsymbol{\eta} ) 
\ha{u}_{N_1, L_1}|_{\mathcal{S}_{j_1}^A}(\tau_1, \xi_1, \boldsymbol{\eta_1} )  
\ha{v}_{N_2, L_2}|_{\mathcal{S}_{j_2}^A} (\tau_2, \xi_2, \boldsymbol{\eta_2} ) 
}
d\bar{\sigma}_1 d\bar{\sigma}_2 \right| \\
& \quad \lesssim A^{\frac{1}{2}} N_{1}^{-2}   (L_0 L_1 L_2)^{\frac{1}{2}} 
\|\ha{u}_{N_1, L_1} (\boldsymbol{\bar{\eta}_1}) \|_{L^2_{\tau \xi \eta'}} 
\| \ha{v}_{N_2, L_2} (\boldsymbol{\bar{\eta}_2})\|_{L^2_{\tau \xi \eta'}} 
\|\ha{w}_{{N_0, L_0}}(\boldsymbol{\bar{\eta}}) \|_{L^2_{\tau \xi \eta'}} ,\label{est07-prop3.4}
\end{split}
\end{equation}
where $\boldsymbol{\bar{\eta}_j} \in \R^{d-2}$ denotes $\boldsymbol{\eta_j}$ excluding $\eta_j'$, $d \bar{\sigma}_j = d\tau_j d \xi_j d \eta_j'$ and 
$\bar{*}$ denotes $(\tau,\xi, \eta') = (\tau_1 + \tau_2,\xi_1+\xi_2, \eta_1' + \eta_2').$ 
Similarly to the previous case, \eqref{est07-prop3.4} is established by the nonlinear Loomis-Whitney inequality. 
To avoid redundancy, we here only consider the transversality
condition, which is given by
\begin{align*}
& \bigl| (\xi_1 \eta_2' - \xi_2 \eta_1') \bigl( 3 ({{\xi}_1}^2 + {\xi}_1 {\xi}_2 + {{\xi}_2}^2 ) - 
({{\eta}_1'}^2 + {\eta}_1' {\eta}_2' + {{\eta}_2'}^2) \bigr)\\
& \quad + ({\xi}_1 {\eta}_2' + {\xi}_2({\eta}_1' +{\eta}_2'))|\boldsymbol{\bar{\eta}_1}|^2 - 2 ({\xi}_1 {\eta}_1' - {\xi}_2 {\eta}_2') 
\boldsymbol{\bar{\eta}_1} \cdot \boldsymbol{\bar{\eta}_2} - ({\xi}_1({\eta}_1' + {\eta}_2') + {\xi}_2 {\eta}_1') |\boldsymbol{\bar{\eta}_2}|^2 \bigr|\\
& \gtrsim A^{-1} N_1^4,
\end{align*}
where we used $|\xi_1 \eta_2' - \xi_2 \eta_1'| \sim A^{-1} N_1^2$, 
$|\eta_j'| \lesssim A^{-1} N_1$ and 
$|\boldsymbol{\bar{\eta}_j}| \ll N_1$.
\end{proof}

\begin{prop}\label{prop3.5}
Assume \textit{Assumption} \textnormal{\ref{assumption1}}. 
Let $A \gg 1$ be dyadic, 
$j_1$, $j_2 \in \Omega_{A, \textnormal{(Ia)}}$, $\alpha(j_1, j_2) \sim A^{-1}$. 
Then we get
\begin{equation}
\begin{split}
& \left|\int_{*}{  \ha{w}_{{N_0, L_0}}(\tau, \xi, \boldsymbol{\eta} ) 
\ha{u}_{N_1, L_1}|_{\mathcal{S}_{j_1}^A}(\tau_1, \xi_1, \boldsymbol{\eta_1} )  
\ha{v}_{N_2, L_2}|_{\mathcal{S}_{j_2}^A}(\tau_2, \xi_2, \boldsymbol{\eta_2} ) 
}
d\sigma_1 d\sigma_2 \right| \\
& \qquad \qquad  \lesssim  A^{-\frac{d-3}{2}} 
N_1^{\frac{d-6}{2}} (L_0 L_1 L_2)^{\frac{1}{2}} \|\ha{u}_{N_1, L_1} \|_{L^2} \| \ha{v}_{N_2, L_2} \|_{L^2} 
\|\ha{w}_{{N_0, L_0}} \|_{L^2}.\label{goal-prop3.5}
\end{split}
\end{equation}
\end{prop}
%%%%%%%%%%%%%%%%%%%%%%%%%%%%%%%%%%%%%
%%%%%%%%%%%%%%%%%%%%%%%%%%%%%%%%%%%%%
%%%%%%%%%%%%%%%%%%%%%%%%%%%%%%%%%%%%%
%%%%%%%%%%%%%%%%%%%%%%%%%%%%%%%%%%%%%
%%%%%%%%%%%%%%%%%%%%%%%%%%%%%%%%%%%%%
\begin{proof}
Proposition \ref{prop3.4} gives \eqref{goal-prop3.5} if $|\xi_1+ \xi_2| \lesssim A^{-1} N_1$. 
Thus we assume $|\xi_1+ \xi_2| \gg A^{-1} N_1$. 
We show that $|\xi_1+\xi_2| \gg A^{-1} N_1$ provides
\begin{align}
|\Phi(\xi_1,\boldsymbol{\eta_1} , \xi_2,\boldsymbol{\eta_2})| := &
\bigl| 3 \xi_1 \xi_2 (\xi_1+ \xi_2) +  (\xi_1 + \xi_2) |\boldsymbol{\eta_1} + \boldsymbol{\eta_2}|^2 - \xi_1  |\boldsymbol{\eta_1}|^2 - 
\xi_2  |\boldsymbol{\eta_2}|^2 \bigr| \notag\\
\gtrsim & A^{-1} N_1^3.\label{est01-prop3.5}
\end{align}
Since $L_{012}^{\max} \gtrsim |\Phi(\xi_1,\boldsymbol{\eta_1} , \xi_2,\boldsymbol{\eta_2})|$, this and Proposition \ref{prop3.3} verify \eqref{goal-prop3.5}. 
Recall that the assumptions $j_1$, $j_2 \in \Omega_{A, \textnormal{(Ia)}}$ 
and $|\xi_1+\xi_2| \gg A^{-1} N_1$ imply 
$\max (|\boldsymbol{\eta_1}|, |\boldsymbol{\eta_2}|) \ll N_1$, 
$|\boldsymbol{\eta_1} + \boldsymbol{\eta_2}| \lesssim |\xi_1 +\xi_2|$. 
Therefore, we have
\begin{align*}
& | (\xi_1 + \xi_2) |\boldsymbol{\eta_1} + \boldsymbol{\eta_2}|^2 - \xi_1  |\boldsymbol{\eta_1}|^2 - 
\xi_2  |\boldsymbol{\eta_2}|^2 | \\ \leq & 
| \xi_1 + \xi_2| (|\boldsymbol{\eta_1} + \boldsymbol{\eta_2}|^2  +  |\boldsymbol{\eta_1}|^2)
 + |\xi_2 (|\boldsymbol{\eta_1}|^2 - |\boldsymbol{\eta_2}|^2) |\\
 \ll & N_1^2 |\xi_1 + \xi_2|,
\end{align*}
which immediately yields \eqref{est01-prop3.5}.
\end{proof}
%%%%%%%%%%%%%%%%%%%%%%%%%%%%%%%%%%%%%
%%%%%%%%%%%%%%%%%%%%%%%%%%%%%%%%%%%%%
%%%%%%%%%%%%%%%%%%%%%%%%%%%%%%%%%%%%%
%%%%%%%%%%%%%%%%%%%%%%%%%%%%%%%%%%%%%
%%%%%%%%%%%%%%%%%%%%%%%%%%%%%%%%%%%%%
\begin{proof}[\underline{Proof of \eqref{goal01-prop3.2} in the case \textnormal{(Ia)}}]
Assume that $(\xi_j,\boldsymbol{\eta}_j)/|(\xi_j,\boldsymbol{\eta}_j)| \in C_{\textnormal{(Ia)}}$. 
Define 
\begin{equation*}
I_{j_1,j_2}^A = \left|\int_{*}{  \ha{w}_{{N_0, L_0}}(\tau, \xi, \boldsymbol{\eta} ) 
\ha{u}_{N_1, L_1}|_{\mathcal{S}_{j_1}^A}(\tau_1, \xi_1, \boldsymbol{\eta_1} )  
\ha{v}_{N_2, L_2}|_{\mathcal{S}_{j_2}^A}(\tau_2, \xi_2, \boldsymbol{\eta_2} ) 
}
d\sigma_1 d\sigma_2 \right|.
\end{equation*}
We observe
\begin{align*}
& \left|\int_{*}{  \ha{w}_{{N_0, L_0}}(\tau, \xi, \boldsymbol{\eta} ) 
\ha{u}_{N_1, L_1}(\tau_1, \xi_1, \boldsymbol{\eta_1} )  \ha{v}_{N_2, L_2}(\tau_2, \xi_2, \boldsymbol{\eta_2} ) 
}
d\sigma_1 d\sigma_2 \right| \\
\lesssim & \sum_{1 \ll A \leq N_1^{3/2}} \sum_{\alpha(j_1, j_2) \sim A^{-1}}
I_{j_1,j_2}^A  + \sum_{\alpha(j_1, j_2) \lesssim N_1^{-3/2}} I_{j_1,j_2}^{N_1^{3/2}}.
\end{align*}
Note that $\LR{|(\xi_1+ \xi_2, \boldsymbol{\eta_1} + \boldsymbol{\eta_2})|} \sim N_0 \gtrsim A^{-1} N_1$ if 
$(\tau_1, \xi_1, \boldsymbol{\eta_1} ) \times 
(\tau_2, \xi_2, \boldsymbol{\eta_2} ) \in {\mathcal{S}}_{j_1}^A \times {\mathcal{S}}_{j_2}^A$ with 
$\alpha(j_1, j_2) \sim A^{-1}$. 
Thus, the former term is estimated by using Proposition \ref{prop3.5} as
\begin{align*}
& \sum_{1 \ll A \leq N_1^{3/2}} \sum_{\alpha(j_1, j_2) \sim A^{-1}}
I_{j_1,j_2}^A \\
\lesssim & 
\sum_{1 \ll A \leq N_1^{3/2}} \sum_{\alpha(j_1, j_2) \sim A^{-1}} N_0^{ \frac{d-3}{2}} N_1^{-\frac{3}{2}} (L_0 L_1 L_2)^{\frac{1}{2}} \|\ha{u}_{N_1, L_1}|_{\mathcal{S}_{j_1}^A} \|_{L^2} 
\| \ha{v}_{N_2, L_2}|_{\mathcal{S}_{j_2}^A} \|_{L^2} 
\|\ha{w}_{{N_0, L_0}} \|_{L^2}\\
\lesssim & 
\sum_{1 \ll A \leq N_1^{3/2}} N_0^{ \frac{d-3}{2}} N_1^{-\frac{3}{2}} (L_0 L_1 L_2)^{\frac{1}{2}} 
\|\ha{u}_{N_1, L_1} \|_{L^2} 
\| \ha{v}_{N_2, L_2}\|_{L^2} 
\|\ha{w}_{{N_0, L_0}} \|_{L^2}\\
\lesssim & \, (\log{N_1}) N_0^{ \frac{d-3}{2}} N_1^{-\frac{3}{2}} (L_0 L_1 L_2)^{\frac{1}{2}} 
\|\ha{u}_{N_1, L_1} \|_{L^2} 
\| \ha{v}_{N_2, L_2}\|_{L^2} 
\|\ha{w}_{{N_0, L_0}} \|_{L^2}.
\end{align*}
By using \eqref{bilinearStrichartz-2} in Proposition \ref{prop3.3} for the latter term, we completes the proof.
\end{proof}
%%%%%%%%%%%%%%%%%%%%%%%%%%%%%%%%%%%%%
%%%%%%%%%%%%%%%%%%%%%%%%%%%%%%%%%%%%%
%%%%%%%%%%%%%%%%%%%%%%%%%%%%%%%%%%%%%
%%%%%%%%%%%%%%%%%%%%%%%%%%%%%%%%%%%%%
%%%%%%%%%%%%%%%%%%%%%%%%%%%%%%%%%%%%%
Next, we treat the case (Ib) $\textnormal{min}(|\boldsymbol{\eta_1} |, |\boldsymbol{\eta_2} |) \ll N_1$,  $\textnormal{max}(|\boldsymbol{\eta_1} |, |\boldsymbol{\eta_2}|) \sim N_1$. 
Without loss of generality, we assume $|{\eta_1}| \sim N_1$ and 
$|\boldsymbol{\eta_2} | \ll N_1$. 
Note that $N_1 \sim N_2$ and $|\boldsymbol{\eta_2} | \ll N_1$ imply $|\xi_2| \sim N_1$. 
We define
\begin{align*}
& F(\xi_1, \boldsymbol{\eta_1} ,\xi_2,\boldsymbol{\eta_2}) \\
& =  
(\xi_1 \eta_2 - \xi_2 \eta_1)\bigl( 3 (\xi_1^2 + \xi_1 \xi_2 + \xi_2^2 ) - 
(\eta_1^2 + \eta_1 \eta_2 + \eta_2^2) \bigr)\\
& \qquad + ({\xi}_1 {\eta}_2 + {\xi}_2({\eta}_1 +{\eta}_2))|\boldsymbol{{\eta}_1'}|^2 - 2 ({\xi}_1 {\eta}_1 - {\xi}_2 {\eta}_2) 
\boldsymbol{{\eta}_1'} \cdot \boldsymbol{{\eta}_2'} - ({\xi}_1({\eta}_1 + {\eta}_2) + {\xi}_2 {\eta}_1) |\boldsymbol{{\eta}_2'}|^2 .
\end{align*}
Recall that this function $F(\xi_1, \boldsymbol{\eta_1} ,\xi_2,\boldsymbol{\eta_2})$ appeared in the proof of Proposition \ref{prop3.4} and provided a transversality of the three hypersurfaces. 
\begin{lem}\label{lemma3.6}
Assume \textit{Assumption} \textnormal{\ref{assumption1}}, $|{\eta_1}| \sim N_1$ and 
$|\boldsymbol{\eta_2} | \ll N_1$. Then we have 
$|F(\xi_1, \boldsymbol{\eta_1} ,\xi_2,\boldsymbol{\eta_2})| \gtrsim N_1^4$.
\end{lem}
%%%%%%%%%%%%%%%%%%%%%%%%%%%%%%%%%%%%%
%%%%%%%%%%%%%%%%%%%%%%%%%%%%%%%%%%%%%
%%%%%%%%%%%%%%%%%%%%%%%%%%%%%%%%%%%%%
%%%%%%%%%%%%%%%%%%%%%%%%%%%%%%%%%%%%%
%%%%%%%%%%%%%%%%%%%%%%%%%%%%%%%%%%%%%
\begin{proof}
We show $L_{012}^{\max} \gtrsim N_1^3$ if $|{\eta_1}| \sim N_1$, 
$|\boldsymbol{\eta_2}| \ll N_1$ and $|F(\xi_1, \boldsymbol{\eta_1} ,\xi_2,\boldsymbol{\eta_2})| \ll N_1^4$. 
Since $|{\eta_1}| \sim N_1$, $|\boldsymbol{\eta_2}| \ll N_1$, it is observed that
\begin{align}
 |F(\xi_1, \boldsymbol{\eta_1} ,\xi_2,\boldsymbol{\eta_2})| \ll N_1^4 \Longrightarrow & 
\bigl| \xi_2 \eta_1 \bigl( 3 (\xi_1^2 + \xi_1 \xi_2 + \xi_2^2 ) - 
\eta_1^2 \bigr) - \xi_2 \eta_1 |{\boldsymbol{\eta_1'}}|^2 \bigr| \ll N_1^4 \notag\\
\Longrightarrow & 
\bigl||{\boldsymbol{\eta_1}}|^2-   3 (\xi_1^2 + \xi_1 \xi_2 + \xi_2^2 )   \bigr| 
\ll N_1^2.\label{est01-lemma3.6}
\end{align}
We use the function $\Phi(\xi_1,\boldsymbol{\eta_1} , \xi_2,\boldsymbol{\eta_2})$ which was defined in the proof of Proposition \ref{prop3.5}. It follows from $|\boldsymbol{\eta_2}| \ll N_1$ and \eqref{est01-lemma3.6} that there is $0 < c \ll 1$ such that
\begin{align*}
|\Phi(\xi_1,\boldsymbol{\eta_1} , \xi_2,\boldsymbol{\eta_2})| & \geq |3 \xi_1 \xi_2 (\xi_1+\xi_2) 
+ \xi_2 |{\boldsymbol{\eta_1}}|^2| - |2(\xi_1+ \xi_2)\boldsymbol{\eta_1} \cdot \boldsymbol{\eta_2} 
+\xi_1 |{\boldsymbol{\eta_2}}|^2|\\
& \geq |\xi_2| |3 \xi_1 (\xi_1+\xi_2) + 3 (\xi_1^2 + \xi_1 \xi_2 + \xi_2^2 )  
 | - c N_1^3\\
& = 3|\xi_2| |2\xi_1^2 + 2 \xi_1 \xi_2 + \xi_2^2| - c N_1^3 \\
& \gtrsim N_1^3,
\end{align*}
which completes the proof.
\end{proof}
%%%%%%%%%%%%%%%%%%%%%%%%%%%%%%%%%%%%%
%%%%%%%%%%%%%%%%%%%%%%%%%%%%%%%%%%%%%
%%%%%%%%%%%%%%%%%%%%%%%%%%%%%%%%%%%%%
%%%%%%%%%%%%%%%%%%%%%%%%%%%%%%%%%%%%%
%%%%%%%%%%%%%%%%%%%%%%%%%%%%%%%%%%%%%
Lemma \ref{lemma3.6} suggests that we can obtain \eqref{goal01-prop3.2} by the same argument as in the proof of 
Proposition \ref{prop3.4}. We omit the details.

%%%%%%%%%%%%%%%%%%%%%%%%%%%%%%%%%%%%%
%%%%%%%%%%%%%%%%%%%%%%%%%%%%%%%%%%%%%
%%%%%%%%%%%%%%%%%%%%%%%%%%%%%%%%%%%%%
%%%%%%%%%%%%%%%%%%%%%%%%%%%%%%%%%%%%%
%%%%%%%%%%%%%%%%%%%%%%%%%%%%%%%%%%%%%
Lastly, we consider the case (Ic) $|\boldsymbol{\eta_1}| \sim |\boldsymbol{\eta_2}| \sim N_1$.

In this case, we perform an angular decomposition in the
$\boldsymbol{\eta}$-space.
In the same way as above (see \cite{BH11}), for $A\in \N$ we choose a maximally separated set $\{\overline{\omega}_A^j \}_{j \in \overline{\Omega}_A} $ of spherical caps of ${\mathbb S}^{d-2}$ of aperture $A^{-1}$, i.e. the angle $\angle{(\theta_1,\theta_2)}$ between 
any two vectors in $\theta_1$, $\theta_2 \in \overline{\omega}_A^j$ satisfies
$
\left| \angle{(\theta_1,\theta_2)} \right| \leq A^{-1}
$
and the characteristic functions $\{ {\mathbf 1}_{\overline{\omega}_A^j} \}$ satisfy
$
1 \leq \sum_{j \in \overline{\Omega}_A} {\mathbf 1}_{\overline{\omega}_A^j}(\theta) \leq 2^d$, for all $ \theta \in {\mathbb S}^{d-2}$.
Further, we define the function
\begin{equation*}
\overline{\alpha} (j_1,j_2) = \inf \left\{ \left| \angle{( \pm \theta_1, \theta_2)} \right| : \ \theta_1 \in \overline{\omega}_A^{j_1}, \ \theta_2 \in \overline{\omega}_A^{j_2} \right\}.
\end{equation*}
For each $j \in \overline{\Omega}_A$ we define 
\begin{equation*}
{\overline{\mathcal{S}}}_{j}^A = 
\left\{ (\tau, \xi, \boldsymbol{\eta}) \in \R \cross \R \cross (\R^{d-1} \setminus \{0\}) \, : \, 
 \frac{\boldsymbol{\eta}}{|\boldsymbol{\eta}|} \in \overline{\omega}_A^j \right\}
\end{equation*}
and the corresponding localization operator
\begin{equation*}
\F (\overline{R}_j^A u) (\tau, \xi,\boldsymbol{\eta}) 
= {\mathbf{1}}_{\omega_j^A} \bigl( \frac{\boldsymbol{\eta}}{|\boldsymbol{\eta}|} \bigr)
 \F u (\tau , \xi,\boldsymbol{\eta}).
\end{equation*}
Let $k = ( k_{(1)}, \ldots, k_{(d)} ) \in \Z^d$. We define regular cubes 
$\{ \mathcal{C}_k^A\}_{k \in \Z^d}$ whose side length is $A^{-1} N_1 $ and $\{ \tilde{\mathcal{C}}_k^A\}_{k \in \Z^d}$ as
\[ \mathcal{C}_k^A  = \{ x=(x_1, \ldots, x_d) \in \R^d  \ | \ x_i \in A^{-1} N_1 [k_{(i)}, k_{(i)}+1) \ 
\textnormal{for all} \ i=1, \ldots, d.
\},\]
we set
$\tilde{\mathcal{C}}_k^A  = \R \times \mathcal{C}_k^A$, and
lastly we define $\mathcal{E}_{j,k}^{A} = {\overline{\mathcal{S}}}_{j}^A \cap \tilde{\mathcal{C}}_k^A$.

\begin{prop}\label{prop3.7}
Assume \textit{Assumption} \textnormal{\ref{assumption1}} 
and $|\boldsymbol{\eta_1}| \sim |\boldsymbol{\eta_2}| \sim N_1$. 
Let $\bar{\alpha}(j_1,j_2) \sim A^{-1}$ and $k_1$, $k_2 \in \Z^d$. Then we get
\begin{align}
& \left|\int_{*}{  \ha{w}_{{N_0, L_0}}(\tau, \xi, \boldsymbol{\eta} ) 
\ha{u}_{N_1, L_1}|_{\mathcal{E}_{j_1,k_1}^{A} } (\tau_1, \xi_1, \boldsymbol{\eta_1} )  
\ha{v}_{N_2, L_2}|_{\mathcal{E}_{j_2,k_2}^{A} } (\tau_2, \xi_2, \boldsymbol{\eta_2} ) 
}
d\sigma_1 d\sigma_2 \right|\notag \\
& \qquad \qquad  \lesssim  A^{-\frac{d-3}{2}} 
N_1^{\frac{d-6}{2}}   (L_0 L_1 L_2)^{\frac{1}{2}} \|\ha{u}_{N_1, L_1} \|_{L^2} \| \ha{v}_{N_2, L_2} \|_{L^2} 
\|\ha{w}_{{N_0, L_0}} \|_{L^2},\label{est01-prop3.7}
\end{align}
where $d \sigma_j = d\tau_j d \xi_j d \boldsymbol{\eta}_j $ and $*$ denotes 
$(\tau, \xi, \boldsymbol{\eta}) = 
(\tau_1 + \tau_2, \xi_1+ \xi_2, \boldsymbol{\eta_1} + \boldsymbol{\eta_2}).$
\end{prop}
\begin{proof}
After rotation, we can assume $|\eta_1 \eta_2' - \eta_2 \eta_1'| \sim A^{-1} N_1^2$ and 
$|\boldsymbol{\eta_j'}| \lesssim A^{-1} N_1$. 
Recall that $\eta_j$ and $\eta_j'$ are first and 
second components of $\boldsymbol{\eta_j}$, respectively. 
For simplicity, we use 
$\boldsymbol{\check{\eta}_j} \in \R^{d-3}$ which satisfies 
$\boldsymbol{\eta_j} = (\eta_j, \boldsymbol{\eta_j'}) = (\eta_j, \eta_j', \boldsymbol{\check{\eta}_j})$. 
Similarly to the proof of Proposition \ref{prop3.4}, for fixed $\xi_1$, $\xi_2$, $\boldsymbol{\check{\eta}_1}$, 
$\boldsymbol{\check{\eta}_2}$, it suffices to show
\begin{align}
& \left|\int_{\tilde{*}}{  \ha{w}_{{N_0, L_0}}(\tau, \xi, \boldsymbol{\eta} ) 
\ha{u}_{N_1, L_1}|_{\mathcal{E}_{j_1,k_1}^{A} } (\tau_1, \xi_1, \boldsymbol{\eta_1} )  
\ha{v}_{N_2, L_2}|_{\mathcal{E}_{j_2,k_2}^{A} } (\tau_2, \xi_2, \boldsymbol{\eta_2} ) 
}
d\tilde{\sigma}_1 d\tilde{\sigma}_2 \right| \notag \\
 \lesssim & A^{\frac{1}{2}} N_{1}^{-2}   (L_0 L_1 L_2)^{\frac{1}{2}} 
\|\ha{u}_{N_1, L_1} (\xi_1,\boldsymbol{\check{\eta}_1}) \|_{L^2_{\tau \eta \eta' }} \| \ha{v}_{N_2, L_2} (\xi_2,\boldsymbol{\check{\eta}_2})\|_{L^2_{\tau \eta \eta'  }} 
\|\ha{w}_{{N_0, L_0}}(\xi, \boldsymbol{\check{\eta}}) \|_{L^2_{\tau \eta \eta' }},\label{est02-prop3.7}
\end{align}
where $d \tilde{\sigma}_j = d\tau_j d \eta_j d \eta_j' $ and 
$\tilde{*}$ denotes $(\tau, \eta, \eta' ) = 
(\tau_1 + \tau_2, \eta_1 + \eta_2,  \eta_1'+ \eta_2').$ 
We follow the proof of Proposition \ref{prop3.4}. 
Assume that $\xi_1$, $\xi_2$, $\boldsymbol{\check{\eta}_1}$, 
$\boldsymbol{\check{\eta}_2}$ are fixed. 
We use the functions $\tilde{f}_{\xi_1, \boldsymbol{\check{\eta}_1}}$, $\tilde{g}_{\xi_2, \boldsymbol{\check{\eta}_2}}$ on $\R^3$ that are defined as
\begin{align*}
\tilde{f}_{\xi_1, \boldsymbol{\check{\eta}_1}}(\tau_1, \eta_1, \eta_1' ) & = 
\ha{u}_{N_1, L_1}|_{\mathcal{E}_{j_1,k_1}^{A} } (N_1^3\tau_1, \xi_1, N_1 \eta_1, N_1 \eta_1',\boldsymbol{\check{\eta}_1} ), \\
\tilde{g}_{\xi_2, \boldsymbol{\check{\eta}_2}}(\tau_2, \eta_2, \eta_2') & = 
\ha{v}_{N_2, L_2}|_{\mathcal{E}_{j_2,k_2}^{A} } (N_1^3\tau_2, \xi_2, N_1\eta_2, N_1 \eta_2', \boldsymbol{\check{\eta}_2}),
\end{align*}
and show the following estimate:
\begin{equation}
\| \tilde{f}_{\xi_1, \boldsymbol{\check{\eta}_1}} |_{S_1} * \tilde{g}_{\xi_2, \boldsymbol{\check{\eta}_2}} |_{S_2} \|_{L^2(S_3)} \lesssim A^{\frac{1}{2}}  
\| \tilde{f}_{\xi_1, \boldsymbol{\check{\eta}_1}} \|_{L^2(S_1)} \| \tilde{g}_{\xi_2, \boldsymbol{\check{\eta}_2}} \|_{L^2(S_2)}.\label{est06-prop3.7}
\end{equation}
Here, $c_0$, $c_1$, $c_2 \in \R$, $\tilde{\xi} = N_1^{-1} \xi$, $\tilde{\xi}_j = N_1^{-1} \xi_j$, 
$\boldsymbol{\overline{\eta}_j} = N_1^{-1}\boldsymbol{\check{\eta}_j}$, 
$\boldsymbol{\overline{\eta}} = N_1^{-1}\boldsymbol{\check{\eta}}$ and for
\[
\phi_{\xi_j, \boldsymbol{\check{\eta}_j}, c_j}  (\eta, \eta')  = (\xi_j(\xi_j^2 + |\boldsymbol{\eta}|^2) + c_j, 
\eta, \eta' ),\]
the surfaces are given as \begin{align*}
S_1 =  \{ & \phi_{\tilde{\xi}_1, \boldsymbol{\overline{\eta}_1}, c_1} (\eta_1, \eta_1') \in \R^3 \ | \ 
(\eta_1, \eta_1')  \in \supp (\tilde{f}_{\xi_1, \boldsymbol{\check{\eta}_1}}\circ \phi_{\tilde{\xi}_1, \boldsymbol{\overline{\eta}_1}, c_1})\}, \\
S_2 = \{ & \phi_{\tilde{\xi}_2, \boldsymbol{\overline{\eta}_2}, c_2}(\eta_2, \eta_2') \in \R^3 \ | \ 
(\eta_2, \eta_2')   
\in \supp (\tilde{g}_{\xi_2, \boldsymbol{\check{\eta}_2}}\circ \phi_{\tilde{\xi}_2, \boldsymbol{\overline{\eta}_2}, c_2}) \},\\
S_3 = \bigl\{ & (\psi_{\tilde{\xi},  \boldsymbol{\bar{\eta}} } (\eta, \eta'), \eta, \eta' ) \in \R^3  \ | \  
\psi_{\xi, \boldsymbol{\check{\eta}}}  (\eta, \eta') = \xi(\xi^2 + |\boldsymbol{\eta}|^2 ) + c_0 \bigr\}.
\end{align*}
Since $\textnormal{diam} (S_1) \lesssim A^{-1}$, $\textnormal{diam} (S_2) \lesssim A^{-1}$, we can assume 
$\textnormal{diam} (S_3) \lesssim A^{-1}$. We easily confirm that $S_1$, $S_2$, $S_3$ satisfy the necessary regularity and diameter conditions  to use the nonlinear Loomis-Whitney inequality. 
Thus, here we only confirm that $S_1$, $S_2$, $S_3$ satisfy the suitable transversality condition. 
We define $\lambda_i \in S_i$ as 
\begin{equation*}
\lambda_1=\phi_{\tilde{\xi}_1, \boldsymbol{\overline{\eta}_1}, c_1} (\eta_1, \eta_1'), \quad 
\lambda_2 = \phi_{\tilde{\xi}_2, \boldsymbol{\overline{\eta}_2}, c_2}(\eta_2, \eta_2'), \quad
\lambda_3 =  (\psi_{\tilde{\xi},  \boldsymbol{\bar{\eta}} } (\eta, \eta'), \eta, \eta' ).
\end{equation*}
The unit normals ${\mathfrak{n}}_i$ on $\lambda_i$ are described explicitly as
\begin{equation*}
{\mathfrak{n}}_i(\lambda_i) = 
\frac{1}{\sqrt{1+ 4{\tilde{\xi}_i}^2 (\eta_i^2+{\eta_i'}^2) }} 
\left(-1, \ 2 \tilde{\xi}_i \eta_i, \ 2 \tilde{\xi}_i \eta_i' \right),
\end{equation*}
for $i=1$, $2$, and the same for ${\mathfrak{n}}_3(\lambda_3)$. 
Letting $\widehat{\boldsymbol{\eta_1}}  = (\widehat{\eta_1}, \widehat{\eta_1}')$, 
$\widehat{\boldsymbol{\eta_2}}  = (\widehat{\eta_2}, \widehat{\eta_2}')$, 
$\widehat{\boldsymbol{\eta}}  = (\widehat{\eta}, {\widehat{\eta}}')$, 
$\widehat{\boldsymbol{\eta_1}}  + \widehat{\boldsymbol{\eta_2}}  = 
\widehat{\boldsymbol{\eta}}$ and
\begin{equation*} 
\widehat{\lambda}_1:=\phi_{\tilde{\xi}_1, \boldsymbol{\overline{\eta}_1}, c_1} (\widehat{\boldsymbol{\eta_1}} )  \in S_1, \ \ 
\widehat{\lambda}_2:=\phi_{\tilde{\xi}_2, \boldsymbol{\overline{\eta}_2}, c_2} (\widehat{\boldsymbol{\eta_2}} )\in S_2, \ \ 
\widehat{\lambda}_3:=(\psi_{\tilde{\xi},  \boldsymbol{\bar{\eta}} } (\widehat{\boldsymbol{\eta}}), \widehat{\boldsymbol{\eta}} ) \in S_3,
\end{equation*}
we show
\begin{equation*}
A^{-1} \lesssim |\textnormal{det} N(\widehat{\lambda}_1, \widehat{\lambda}_2, \widehat{\lambda}_3)|,
\end{equation*}
which means the transversality of $S_1$, $S_2$, $S_3$ and completes the proof. 
We observe
\begin{align*}
 |\textnormal{det} N(\widehat{\lambda}_1, \widehat{\lambda}_2, \widehat{\lambda}_3)|  \gtrsim & 
 \left|\textnormal{det}
\begin{pmatrix}
-1 & -1 & - 1 \\
2 \tilde{\xi}_1 \widehat{\eta_1}  & 2 \tilde{\xi}_2 \widehat{\eta_2} & 2 \tilde{\xi} \widehat{\eta} \\
2 \tilde{\xi}_1 \widehat{\eta_1}'   & 2 \tilde{\xi}_2 \widehat{\eta_2}'  & 2 \tilde{\xi} \widehat{\eta}'
\end{pmatrix} \right| \notag \\
\gtrsim &  \bigl| (\widehat{\eta_1} \widehat{\eta_2}' - \widehat{\eta_2} \widehat{\eta_1}')(\tilde{\xi}_1^2 + \tilde{\xi}_1 
\tilde{\xi}_2 + \tilde{\xi}_2^2 ) \bigr|\\
\gtrsim & A^{-1}.
\end{align*}
Here we used the assumptions $|\eta_1 \eta_2' - \eta_2 \eta_1'| \sim A^{-1} N_1^2$ and 
$\max(|\xi_1|, |\xi_2|) \sim N_1$ which imply 
$|\widehat{{\eta_1}} \widehat{{\eta_2}}' - 
\widehat{{\eta_2}} \widehat{{\eta_1}}'|\sim A^{-1}$ and $\max(|\tilde{\xi}_1|, |\tilde{\xi}_2|) \sim 1$, respectively.
\end{proof}
%%%%%%%%%%%%%%%%%%%%%%%%%%%%%%%%%%%%%
%%%%%%%%%%%%%%%%%%%%%%%%%%%%%%%%%%%%%
%%%%%%%%%%%%%%%%%%%%%%%%%%%%%%%%%%%%%
%%%%%%%%%%%%%%%%%%%%%%%%%%%%%%%%%%%%%
%%%%%%%%%%%%%%%%%%%%%%%%%%%%%%%%%%%%%
\begin{prop}\label{prop3.8}
Assume \textit{Assumption} \textnormal{\ref{assumption1}} 
and $|\boldsymbol{\eta_1}| \sim |\boldsymbol{\eta_2}| \sim N_1$. 
Let $\bar{\alpha}(j_1,j_2) \sim A^{-1}$. Then we get
\begin{align*}
& \left|\int_{*}{  \ha{w}_{{N_0, L_0}}(\tau, \xi, \boldsymbol{\eta} ) 
\ha{u}_{N_1, L_1}|_{ {\overline{\mathcal{S}}}_{j_1}^A } (\tau_1, \xi_1, \boldsymbol{\eta_1} )  
\ha{v}_{N_2, L_2}|_{ {\overline{\mathcal{S}}}_{j_2}^A } (\tau_2, \xi_2, \boldsymbol{\eta_2} ) 
}
d\sigma_1 d\sigma_2 \right| \\
& \qquad \qquad \qquad \lesssim  N_0^{\frac{d-4}{2}+2 \e} N_{1}^{-1-\frac{3}{2}\e}    (L_0 L_1 L_2)^{\frac{1}{2}} \|\ha{u}_{N_1, L_1} \|_{L^2} \| \ha{v}_{N_2, L_2} \|_{L^2} 
\|\ha{w}_{{N_0, L_0}} \|_{L^2},
\end{align*}
where $d \sigma_j = d\tau_j d \xi_j d \boldsymbol{\eta}_j$ and $*$ denotes $(\tau, \xi, \boldsymbol{\eta}) 
= (\tau_1 + \tau_2, \xi_1+ \xi_2, \boldsymbol{\eta_1} + \boldsymbol{\eta_2} ).$
\end{prop}
Before we state a proof, let us see that Proposition \ref{prop3.8} establishes \eqref{goal01-prop3.2} in the case \textnormal{(Ic)}.
%%%%%%%%%%%%%%%%%%%%%%%%%%%%%%%%%%%%%
%%%%%%%%%%%%%%%%%%%%%%%%%%%%%%%%%%%%%
\begin{proof}[\underline{Proof of \eqref{goal01-prop3.2} in the case \textnormal{(Ic)}}]
\label{Proofof(4.1)}
For convenience, we use
\begin{equation*}
\overline{I}_{j_1,j_2}^A:=
\left|\int_{*}{  \ha{w}_{{N_0, L_0}}(\tau, \xi, \boldsymbol{\eta} ) 
\ha{u}_{N_1, L_1}|_{ {\overline{\mathcal{S}}}_{j_1}^A } (\tau_1, \xi_1, \boldsymbol{\eta_1} )  
\ha{v}_{N_2, L_2}|_{ {\overline{\mathcal{S}}}_{j_2}^A } (\tau_2, \xi_2, \boldsymbol{\eta_2} ) 
}
d\sigma_1 d\sigma_2 \right|.
\end{equation*}
We observe that
\begin{align*}
& \left|\int_{*}{  \ha{w}_{{N_0, L_0}}(\tau, \xi, \boldsymbol{\eta}) 
\ha{u}_{N_1, L_1}(\tau_1, \xi_1, \boldsymbol{\eta_1})  \ha{v}_{N_2, L_2}(\tau_2, \xi_2, \boldsymbol{\eta_2}) 
}
d\sigma_1 d\sigma_2 \right| \\
&\lesssim \sum_{2 \leq A \leq N_1^6} \sum_{\bar{\alpha}(j_1,j_2) \sim A^{-1}}\overline{I}_{j_1,j_2}^A 
+ \sum_{\bar{\alpha}(j_1,j_2) \lesssim N_1^{-6}}\overline{I}_{j_1,j_2}^{N_1^6}.
\end{align*}
For the former term, by using Proposition \ref{prop3.8} and the almost orthogonality of $j_1$, $j_2$ which satisfy $\bar{\alpha}(j_1,j_2) \sim A^{-1}$, we get
\begin{align*}
& \sum_{2 \leq A \leq N_1^6} \sum_{\bar{\alpha}(j_1,j_2) \sim A^{-1}}\overline{I}_{j_1,j_2}^A\\
& \lesssim \sum_{2 \leq A \leq N_1^6} N_0^{\frac{d-4}{2}+2 \e} N_{1}^{-1-\frac{3}{2}\e}   (L_0 L_1 L_2)^{\frac{1}{2}} \|\ha{u}_{N_1, L_1} \|_{L^2} \| \ha{v}_{N_2, L_2} \|_{L^2} 
\|\ha{w}_{{N_0, L_0}} \|_{L^2}\\
& \lesssim N_0^{\frac{d-4}{2}+2\e} N_{1}^{-1-\e}   (L_0 L_1 L_2)^{\frac{1}{2}} \|\ha{u}_{N_1, L_1} \|_{L^2} \| \ha{v}_{N_2, L_2} \|_{L^2} 
\|\ha{w}_{{N_0, L_0}} \|_{L^2}.
\end{align*}
For the latter term, since the size of the set 
$\{(\xi_1, \boldsymbol{\eta_1}) | (\tau_1, \xi_1, \boldsymbol{\eta_1}) \in {\overline{\mathcal{S}}}_{j_1}^{N_1^6}\}$ is less than $\sim N_1^{-5(d-2)+2} \leq N_1^{-3}$, we easily obtain
\begin{equation*}
\overline{I}_{j_1,j_2}^{N_1^6} \lesssim N_1^{-\frac{3}{2}} L_1^{\frac{1}{2}} \|\ha{u}_{N_1, L_1}|_{ {\overline{\mathcal{S}}}_{j_1}^A } \|_{L^2} \| \ha{v}_{N_2, L_2}|_{ {\overline{\mathcal{S}}}_{j_2}^A } \|_{L^2} 
\|\ha{w}_{{N_0, L_0}} \|_{L^2},
\end{equation*}
which completes the proof of \eqref{goal01-prop3.2} in the case \textnormal{(Ic)}.
\end{proof}
The next subsection is devoted to the proof of Proposition \ref{prop3.8}. 
Note that, as in the proof of Proposition \ref{prop3.4}, by rotating $\boldsymbol{\eta_1}$, $\boldsymbol{\eta_2}$, we can assume 
$|\eta_1 \eta_2' - \eta_2 \eta_1'| \sim A^{-1} N_1^2$ and $|\boldsymbol{\eta_j'}| \lesssim A^{-1} N_1$. 
Further, by performing the invertible linear transformation 
$(\xi_j, \eta_j) \to (\xi_j+\eta_j, \sqrt{3} (\xi_j-\eta_j))$, 
it is easily observed that Proposition \ref{prop3.8} is equivalent to
Proposition \ref{prop3.9} below. 

\subsection{Proof of Proposition \ref{prop3.8} }\label{subsec:proof-prop}
As justified by the above discussion, in this subsection we assume the following:
\begin{manualass}{1'}\label{assumption1prime}$ \, $\\
(1) $L_{012}^{\max} \ll (N_{012}^{\max})^3$, $ \ $ 
(2) $1 \ll N_0 \lesssim N_1 \sim N_2$,  $ \ $ 
(3) $\textnormal{max}(|\xi_1+\eta_1|, |\xi_2+\eta_2|) \geq 2^{-5} N_1$.
\end{manualass}
\begin{prop}\label{prop3.9}
 In addition to Assumption \ref{assumption1prime}, suppose that
$|\xi_j-\eta_j| \sim N_1$, $|\boldsymbol{\eta_j'}| \lesssim A^{-1} N_1$ where $j=1,2$ and 
$|(\xi_1-\eta_1) \eta_2' - (\xi_2-\eta_2) \eta_1'| \sim A^{-1} N_1^2$. Then we get
\begin{equation}
\begin{split}
& \left|\int_{*}{  h_{{N_0, L_0}}(\tau, \xi, \boldsymbol{\eta} ) 
f_{N_1, L_1} (\tau_1, \xi_1, \boldsymbol{\eta_1} )  
g_{N_2, L_2} (\tau_2, \xi_2, \boldsymbol{\eta_2} ) 
}
d\sigma_1 d\sigma_2 \right| \\
& \qquad \qquad \qquad \lesssim  N_0^{\frac{d-4}{2}+2\e} N_{1}^{-1-\frac{3}{2}\e}   (L_0 L_1 L_2)^{\frac{1}{2}} 
\|f_{N_1, L_1} \|_{L^2} \| g_{N_2, L_2} \|_{L^2} 
\|h_{{N_0, L_0}}\|_{L^2},\label{est01-prop3.9}
\end{split}
\end{equation}
where functions $f_{N_1, L_1}$, $g_{N_2, L_2}$, $h_{{N_0, L_0}}$ satisfy
\begin{align}
& \supp f_{N_1, L_1}  \subset G_{N_1, L_1}, \quad 
\supp g_{N_2, L_2}  \subset G_{N_2, L_2}, \quad
\supp h_{N_0, L_0}  \subset G_{N_0, L_0},\label{assumption-fgh}\\
& G_{N,L} := \{ (\tau, \xi, \boldsymbol{\eta}) \in  \R^{d+1} \, | \, 
\LR{|(\xi, \boldsymbol{\eta})|} \sim N,  \LR{\tau- (\xi^3+ \eta^3) - (\xi+\eta) |\boldsymbol{\eta'}|^2} \sim L\}.\notag
\end{align}
\end{prop}
We consider Proposition \ref{prop3.9} instead of Proposition \ref{prop3.8}. 
The advantage in this way is that we can reuse the propositions and lemmas that were established in the paper by the second author \cite{Ki2019} which was concerned with the 
$2$D Zakharov-Kuznetsov equation. In \cite{Ki2019}, the following symmetrized $2$D Zakharov-Kuznetsov equation was considered.
\begin{equation*}
\partial_t u + (\partial_x^3 + \partial_y^3) u = 4^{-\frac{1}{3}} 
(\partial_x+ \partial_y) (u^2), \quad (t,x,y) \in \R \times \R^2.
\end{equation*}
This equation is equivalent to the original $2$D Zakharov-Kuznetsov equation, which can be seen by applying the above linear transformation $(\xi_j, \eta_j) \to (\xi_j+\eta_j, \sqrt{3} (\xi_j-\eta_j))$ to the original $2$D Zakharov-Kuznetsov equation. See \cite{GH14}. 

Now we turn to Proposition \ref{prop3.9}. Note that the assumptions in Proposition \ref{prop3.9} suggest that we can assume $A^{-1} N_1 \lesssim N_0$. 
We divide the proof into the two cases 
\begin{equation*}
| \sin \angle \left( (\xi_1, \eta_1), (\xi_2, \eta_2) \right)| \gtrsim 1 \quad 
\textnormal{and} \quad| \sin \angle \left( (\xi_1, \eta_1), (\xi_2, \eta_2) \right)| \ll 1.
\end{equation*} 
First, we consider the case $| \sin \angle \left( (\xi_1, \eta_1), (\xi_2, \eta_2) \right)| \gtrsim 1$.
\begin{defn}
Let $M \gg 1$ be a dyadic number and $\ell = ( \ell_{(1)}, \ell_{(2)} ) \in \Z^2$. We define square-tiles 
$\{ \mathcal{T}_\ell^M\}_{\ell \in \Z^2}$ whose side length is $M^{-1} N_1 $ and $\{ \tilde{\mathcal{T}}_\ell^M\}_{k \in \Z^2}$ as follows:
\begin{align*}
& \mathcal{T}_\ell^M : = \{ (\xi,\eta) \in \R^2 \ | \ (\xi ,\eta) \in M^{-1} N_1 \bigl(
[ \ell_{(1)}, \ell_{(1)} + 1) \times [ \ell_{(2)}, \ell_{(2)} + 1) \bigr) \}\\
& \tilde{\mathcal{T}}_\ell^M  : = \R \times \mathcal{T}_\ell^M \times \R^{d-2}.
\end{align*}
\end{defn}
%%%%%%%%%%%%%%%%%%%%%%%%%%%%%%%%%%%%%
%%%%%%%%%%%%%%%%%%%%%%%%%%%%%%%%%%%%%
%%%%%%%%%%%%%%%%%%%%%%%%%%%%%%%%%%%%%
%%%%%%%%%%%%%%%%%%%%%%%%%%%%%%%%%%%%%
%%%%%%%%%%%%%%%%%%%%%%%%%%%%%%%%%%%%%
\begin{defn}[Whitney type decomposition]\label{definition3}
 Let $A$, $M$, $\widehat{M}$ be dyadic such that $1 \ll \widehat{M} \leq M \leq A$ and 
\begin{align*}
\overline{\Phi} (\xi_1, \eta_1, \xi_2 ,\eta_2) & = \xi_1 \xi_2(\xi_1 + \xi_2) + \eta_1 \eta_2 (\eta_1 + \eta_2), \\
\overline{F} (\xi_1, \eta_1, \xi_2 ,\eta_2) & = \xi_1 \eta_2 +  \xi_2 \eta_1 + 2 (\xi_1 \eta_1 + \xi_2 \eta_2).
\end{align*}
We define
\begin{align*}
Z_M^1 & = \{ (\ell_1, \ell_2) \in \Z^2 \times \Z^2 \, | \, 
|\overline{\Phi}(\xi_1, \eta_1, \xi_2, \eta_2)| \geq M^{-1} N_1^3  \ \ \textnormal{for all} \ (\xi_j, \eta_j) \in 
\mathcal{T}_{\ell_j}^M \},\\
Z_M^2  & = \{ (\ell_1, \ell_2) \in \Z^2 \times \Z^2 \, | \, 
| \overline{F} (\xi_1, \eta_1, \xi_2 ,\eta_2)| \geq M^{-1} N_1^2  \ \ \textnormal{for all} \ (\xi_j, \eta_j) \in 
\mathcal{T}_{\ell_j}^M \},\\
Z_M & = Z_M^1 \cup Z_M^2 \subset \Z^2 \times \Z^2, 
\qquad R_M = \bigcup_{(\ell_1, \ell_2) \in Z_M} \mathcal{T}_{\ell_1}^M \times 
\mathcal{T}_{\ell_2}^M \subset \R^2 \times \R^2.
\end{align*}
It is clear that $M_1 \leq M_2 \Longrightarrow R_{M_1} \subset R_{M_2}$. 
Further, we define
\begin{equation*}
Q_M = 
\begin{cases}
R_M \setminus R_{\frac{M}{2}} \quad \textnormal{for} \  M > \widehat{M},\\
 \ R_{\widehat{M}}  \ \qquad \quad \textnormal{for} \  M = \widehat{M}.
\end{cases}
\end{equation*}
and a set of pairs of integer pair $Z_M' \subset Z_M$ as
\begin{equation*}
\bigcup_{(\ell_1, \ell_2) \in Z_M'} \mathcal{T}_{\ell_1}^M \times 
\mathcal{T}_{\ell_2}^M = Q_M.
\end{equation*}
We easily see that $Z_M'$ is uniquely defined and 
\begin{equation*}
M_1 \not= M_2 \Longrightarrow Q_{M_1} \cap Q_{M_2} = \emptyset, \quad 
\bigcup_{\widehat{M} \leq M \leq M_0} Q_{M} = R_{M_0}
\end{equation*}
where $M_0 \geq \widehat{M}$ is dyadic. Thus, we can decompose $\R^2 \times \R^2$ as
\begin{equation*}
\R^2 \times \R^2 = \left( \bigcup_{\widehat{M} \leq M \leq M_0} Q_{M}\right) \cup (R_{M_0})^c.
\end{equation*}
Lastly, we define
\begin{align*}
\mathcal{A}  & = \{ (\tau_1, \xi_1, \boldsymbol{\eta_1}) 
\times (\tau_2, \xi_2, \boldsymbol{\eta_2}) \in \R^{d+1} \times \R^{d+1} \, | \, 
| \sin \angle \left( (\xi_1, \eta_1), (\xi_2, \eta_2) \right)| \gtrsim 1  \},\\
\tilde{Z}_{M} & = \{ (\ell_1, \ell_2) \in Z_M' \, | \, 
\left( \tilde{\mathcal{T}}_{\ell_1}^M \times \tilde{\mathcal{T}}_{\ell_2}^M\right) \cap \left( 
G_{N_1, L_1} \times G_{N_2, L_2} \right) \cap \mathcal{A} \not= 
\emptyset \}.
\end{align*}
\end{defn}
%%%%%%%%%%%%%%%%%%%%%%%%%%%%%%%%%%%%%
%%%%%%%%%%%%%%%%%%%%%%%%%%%%%%%%%%%%%
%%%%%%%%%%%%%%%%%%%%%%%%%%%%%%%%%%%%%
%%%%%%%%%%%%%%%%%%%%%%%%%%%%%%%%%%%%%
%%%%%%%%%%%%%%%%%%%%%%%%%%%%%%%%%%%%%
%%%%%%%%%%%%%%%%%%%%%%%%%%%%%%%%%%%%%
%%%%%%%%%%%%%%%%%%%%%%%%%%%%%%%%%%%%%
By the same argument as for the $2$D case in \cite{Ki2019}, we can obtain the following estimate.
\begin{prop}\label{prop3.10}
Assume $|\boldsymbol{\eta_j'}| \lesssim A^{-1} N_1$ and \eqref{assumption-fgh}. 
Let $A$, $M$ be dyadic which satisfy $1 \ll M \leq A$ and $(\ell_1,\ell_2) \in \tilde{Z}_{M}$. 
Then we get
\begin{equation}
\begin{split}
& \left|\int_{*}{  h_{{N_0, L_0}}(\tau, \xi, \boldsymbol{\eta} ) 
f_{N_1, L_1}|_{\tilde{\mathcal{T}}_{\ell_1}^M} (\tau_1, \xi_1, \boldsymbol{\eta_1} )
g_{N_2, L_2}|_{\tilde{\mathcal{T}}_{\ell_2}^M} (\tau_2, \xi_2, \boldsymbol{\eta_2} )
}
d\sigma_1 d\sigma_2 \right| \\
& \qquad \lesssim  A^{-\frac{d-2}{2}} M^{\frac{1}{2}}
N_1^{\frac{d-6}{2}}    (L_0 L_1 L_2)^{\frac{1}{2}} 
\|f_{N_1, L_1}|_{\tilde{\mathcal{T}}_{\ell_1}^M} \|_{L^2} \| g_{N_2, L_2}|_{\tilde{\mathcal{T}}_{\ell_2}^M} \|_{L^2} 
\|h_{{N_0, L_0}}\|_{L^2},\label{est01-prop3.10}
\end{split}
\end{equation}
where $d \sigma_j = d\tau_j d \xi_j d \boldsymbol{\eta}_j$ and $*$ denotes $(\tau, \xi, \boldsymbol{\eta}) = (\tau_1 + \tau_2, \xi_1+ \xi_2, \boldsymbol{\eta_1} + \boldsymbol{\eta_2}).$
\end{prop}
\begin{proof}
Since $|\boldsymbol{\eta_j'}| \lesssim A^{-1} N_1$, for fixed $\boldsymbol{\eta_1'}$, $\boldsymbol{\eta_2'}$, it suffices to show
\begin{equation}
\begin{split}
& \left|\int_{\hat{*}}{  h_{{N_0, L_0}}(\tau, \xi, \boldsymbol{\eta} ) 
f_{N_1, L_1}|_{\tilde{\mathcal{T}}_{\ell_1}^M} (\tau_1, \xi_1, \boldsymbol{\eta_1} )
g_{N_2, L_2}|_{\tilde{\mathcal{T}}_{\ell_2}^M} (\tau_2, \xi_2, \boldsymbol{\eta_2} )
}
d\hat{\sigma}_1 d\hat{\sigma}_2 \right| \\
&  \lesssim M^{\frac{1}{2}} N_{1}^{-2}   (L_0 L_1 L_2)^{\frac{1}{2}} 
\|f_{N_1, L_1}|_{\tilde{\mathcal{T}}_{\ell_1}^M} (\boldsymbol{\eta_1'}) \|_{L^2_{\tau \xi \eta}} 
\|g_{N_2, L_2}|_{\tilde{\mathcal{T}}_{\ell_2}^M} (\boldsymbol{\eta_2'})\|_{L^2_{\tau \xi \eta}} 
\|h_{{N_0, L_0}}(\boldsymbol{\eta'}) \|_{L^2_{\tau \xi \eta}} ,\label{est02-prop3.10}
\end{split}
\end{equation}
where $d \hat{\sigma}_j = d\tau_j d \xi_j d \eta_j$ and 
$\hat{*}$ denotes $(\tau,\xi, \eta) = (\tau_1 + \tau_2,\xi_1+\xi_2, \eta_1 + \eta_2).$ 
\eqref{est02-prop3.10} is established by the same argument as for Propositions 3.3-3.5 in \cite{Ki2019} 
which considered the Cauchy problem of the $2$D Zakharov-Kuznetsov equation. 
The only difference is that, in \cite{Ki2019} it was assumed that $f_{N_1, L_1}$, $g_{N_2, L_2}$, $h_{{N_0, L_0}}$ satisfy
\begin{align*}
& \supp f_{N_1, L_1}  \subset \overline{G}_{N_1, L_1}, \quad 
\supp g_{N_2, L_2}  \subset \overline{G}_{N_2, L_2}, \quad
\supp h_{N_0, L_0}  \subset \overline{G}_{N_0, L_0},\\
& \overline{G}_{N,L} := \{ (\tau, \xi, \boldsymbol{\eta}) \in  \R^{d+1} \, | \, 
\LR{|(\xi, \boldsymbol{\eta})|} \sim N,  \LR{\tau- (\xi^3+ \eta^3)} \sim L\},\notag
\end{align*}
instead of \eqref{assumption-fgh}. 
We will see that, because of the assumptions $ M \leq A$ and $|\boldsymbol{\eta_j'}| \lesssim A^{-1} N_1$,  the proofs of Propositions 3.3-3.5 in \cite{Ki2019} can be transferred. 
Firstly, either $|\overline{\Phi} (\xi_1, \eta_1, \xi_2 ,\eta_2)| \geq M^{-1} N_1^3$ or 
$|\overline{F} (\xi_1, \eta_1, \xi_2 ,\eta_2)| \geq M^{-1} N_1^2$ holds 
under the assumption $(\ell_1,\ell_2) \in \tilde{Z}_{M}$. 
We first assume $|\overline{\Phi} (\xi_1, \eta_1, \xi_2 ,\eta_2)| \geq M^{-1} N_1^3$ and show \eqref{est02-prop3.10}. 
For simplicity, we use
\begin{align*}
f_{\boldsymbol{\eta_1'}}(\tau_1, \xi_1, \eta_1) & :=f_{N_1, L_1}|_{\tilde{\mathcal{T}}_{\ell_1}^M} (\tau_1, \xi_1, \boldsymbol{\eta_1} ), \\
g_{\boldsymbol{\eta_2'}}(\tau_2, \xi_2, \eta_2) & := g_{N_2, L_2}|_{\tilde{\mathcal{T}}_{\ell_2}^M} (\tau_2, \xi_2, \boldsymbol{\eta_2} ),\\
h_{\boldsymbol{\eta'}}(\tau, \xi, \eta) &:= h_{{N_0, L_0}}(\tau, \xi, \boldsymbol{\eta} ).
\end{align*}
Since
\begin{align*}
3 L_{012}^{\max} & \geq \Bigl| 
\tau_1+ \tau_2 - ((\xi_1+\xi_2)^3 + (\eta_1+\eta_2)^3) -(\xi_1+\xi_2 +\eta_1+\eta_2) 
|\boldsymbol{\eta_1'}+\boldsymbol{\eta_2'}|^2\\
& - (\tau_1 - (\xi_1^3 +\eta_1^3) -(\xi_1+\eta_1)|\boldsymbol{\eta_1'}|^2) 
- (\tau_2 - (\xi_2^3 +\eta_2^3) -(\xi_2+\eta_2)|\boldsymbol{\eta_2'}|^2) \Bigr|\\
& \gtrsim |\overline{\Phi} (\xi_1, \eta_1, \xi_2 ,\eta_2)| 
+ \mathcal{O}(A^{-2}N_1)  \gtrsim M^{-1} N_1^3,
\end{align*} 
the following estimates which correspond to Proposition 3.3 in \cite{Ki2019} immediately yields \eqref{est02-prop3.10}.
\begin{align}
& \left\| {\mathbf{1}}_{G_{N_0, L_0}} \int f_{\boldsymbol{\eta_1'}}(\tau_1, \xi_1, \eta_1) 
g_{\boldsymbol{\eta_2'}} (\tau- \tau_1, \xi-\xi_1, \eta- \eta_1) 
d\hat{\sigma}_1 \right\|_{L_{\tau \xi \eta}^2} \notag \\ 
& \qquad \qquad \qquad \qquad \qquad \qquad 
\lesssim (MN_1 )^{-\frac{1}{2}} (L_1 L_2)^{\frac{1}{2}} 
\|f_{\boldsymbol{\eta_1'}} \|_{L_{\tau \xi \eta}^2}
\|g_{\boldsymbol{\eta_2'}} \|_{L_{\tau \xi \eta}^2},\label{bilinear01-prop3.10}\\
%%%%%%%%%%%%%%%%%%%%%%%%%%%%%%%%%%%%%
%%%%%%%%%%%%%%%%%%%%%%%%%%%%%%%%%%%%%
& \left\| {\mathbf{1}}_{G_{N_1, L_1} \cap \tilde{\mathcal{T}}_{k_1}^A} \int 
g_{\boldsymbol{\eta_2'}} (\tau_1, \xi_2, \eta_2)  
h_{\boldsymbol{\eta'}} (\tau_1+ \tau_2, \xi_1+\xi_2, \eta_1+ \eta_2) d\hat{\sigma}_2 
\right\|_{L_{\tau \xi \eta}^2}
\notag \\ 
& \qquad \qquad \qquad \qquad \qquad \qquad  
\lesssim (M N_1)^{-\frac{1}{2}} (L_0 L_2)^{\frac{1}{2}} 
\|g_{\boldsymbol{\eta_2'}} \|_{L_{\tau \xi \eta}^2} 
\|h_{\boldsymbol{\eta'}}\|_{L^2}, \label{bilinear02-prop3.10}\\
%%%%%%%%%%%%%%%%%%%%%%%%%%%%%%%%%%%%%
%%%%%%%%%%%%%%%%%%%%%%%%%%%%%%%%%%%%%
& \left\| {\mathbf{1}}_{G_{N_2,L_2} \cap \tilde{\mathcal{T}}_{k_2}^A} \int h_{\boldsymbol{\eta'}} (\tau_1+ \tau_2, \xi_1+\xi_2, \eta_1+ \eta_2)  f_{\boldsymbol{\eta_1'}}(\tau_1, \xi_1, \eta_1)  d \hat{\sigma}_1 
\right\|_{L_{\tau \xi \eta}^2} \notag \\ 
& \qquad \qquad \qquad \qquad \qquad  \qquad 
\lesssim (MN_1)^{-\frac{1}{2}} (L_0 L_1)^{\frac{1}{2}} \|h_{\boldsymbol{\eta'}}\|_{L_{\tau \xi \eta}^2} 
\|f_{\boldsymbol{\eta_1'}} \|_{L_{\tau \xi \eta}^2}.\label{bilinear03-prop3.10}
\end{align}
Here we sketch the proof of \eqref{bilinear01-prop3.10} only. 
The other estimates \eqref{bilinear02-prop3.10} 
and \eqref{bilinear03-prop3.10} can be obtained in the same way as for \eqref{bilinear01-prop3.10}. 
We first observe that the assumptions imply
\begin{equation}
\max \left( |(\xi_1^2 - (\xi - \xi_1)^2|, \, |(\eta_1^2 - (\eta - \eta_1)^2| \right)  \gtrsim N_1^2.
\label{est-space-modu-prop3.10}
\end{equation}
If \eqref{est-space-modu-prop3.10} does not hold, we can assume one of the following.
\begin{align*}
(1) \quad |\xi_1 - (\xi - \xi_1) | \ll N_1  \quad \textnormal{and} \quad 
|\eta_1 - (\eta - \eta_1) | \ll N_1,\\
(2) \quad |\xi_1 - (\xi - \xi_1) | \ll N_1  \quad \textnormal{and} \quad 
|\eta_1 + (\eta - \eta_1) | \ll N_1,\\
(3) \quad |\xi_1 + (\xi - \xi_1) | \ll N_1  \quad \textnormal{and} \quad  
|\eta_1 - (\eta - \eta_1) | \ll N_1,\\
(4) \quad |\xi_1 + (\xi - \xi_1) | \ll N_1  \quad \textnormal{and} \quad 
|\eta_1 + (\eta - \eta_1) | \ll N_1.
\end{align*}
(1) and (4) contradict the angular assumption $| \sin \angle \left( (\xi_1, \eta_1), (\xi-\xi_1, \eta-\eta_1) \right)| \gtrsim 1$. We show (2) contradicts one of the assumptions. 
Clearly, $\max( |\xi_1|, |\xi- \xi_1|) \gtrsim N_1$ holds because of the angular condition 
$| \sin \angle \left( (\xi_1, \eta_1), (\xi-\xi_1, \eta-\eta_1) \right)| \gtrsim 1$. 
Without loss of generality, we can assume $|\xi_1| \gtrsim N_1$. 
This and the inequality $ |\xi_1 - (\xi - \xi_1) | \ll  N_1$ in (2) yield $\min (|\xi|, |\xi- \xi_1|) \gtrsim N_1$ which, combined with 
$|\eta| =  |\eta_1 + (\eta - \eta_1) | \ll N_1$ in (2), gives
\begin{align*}
3 L_{012}^{\max}
& \geq 3 \max  \left(|\tau - \xi^3 - \eta^3|, |\tau_1 - \xi_1^3 - \eta_1^3|, 
|\tau-\tau_1 - (\xi- \xi_1)^3 - (\eta-\eta_1)^3 |\right) + \mathcal{O}(A^{-2} N_1^3)\\ 
& \geq
|\xi \xi_1 (\xi-\xi_1) + \eta \eta_1 (\eta- \eta_1)|+ \mathcal{O}(A^{-2} N_1^3)\\
& \geq
|\xi \xi_1 (\xi-\xi_1)| - | \eta \eta_1 (\eta- \eta_1)|+ \mathcal{O}(A^{-2} N_1^3)\\
& \gtrsim N_1^3 + \mathcal{O}(A^{-2} N_1^3) \gtrsim N_1^3
\end{align*}
which contradicts $L_{012}^{\max} \ll N_1^3$. Similarly, we can show that (3) contradicts at least one of the assumptions. Without loss of generality, we assume $|\xi_1^2 - (\xi - \xi_1)^2| \gtrsim N_1^2.$ 

We turn to show \eqref{bilinear01-prop3.10}. By the Cauchy-Schwarz inequality, we get
\begin{align*}
& \left\| {\mathbf{1}}_{G_{N_0, L_0}} \int f_{\boldsymbol{\eta_1'}}(\tau_1, \xi_1, \eta_1) 
g_{\boldsymbol{\eta_2'}} (\tau- \tau_1, \xi-\xi_1, \eta- \eta_1) 
d\hat{\sigma}_1 \right\|_{L_{\tau \xi \eta}^2} \\
\leq & \left\| {\mathbf{1}}_{G_{N_0, L_0}} \left( \bigl| f_{\boldsymbol{\eta_1'}} \bigr|^2 * 
\bigl| g_{\boldsymbol{\eta_2'}} \bigr|^2
 \right)^{1/2} |E(\tau, \xi, \eta)|^{1/2} \right\|_{L_{\tau \xi \eta}^2}  \\
\leq & \sup_{(\tau, \xi, \eta) \in G_{N_0, L_0}} |E(\tau, \xi, \eta)|^{1/2} 
 \left\|\bigl| f_{\boldsymbol{\eta_1'}} \bigr|^2 * 
\bigl| g_{\boldsymbol{\eta_2'}} \bigr|^2 
\right\|_{L_{\tau \xi \eta}^1}^{1/2}\\
\leq & \sup_{(\tau, \xi, \eta) \in G_{N_0, L_0}} |E(\tau, \xi, \eta)|^{1/2} 
\|f_{\boldsymbol{\eta_1'}} \|_{L_{\tau \xi \eta}^2}
\|g_{\boldsymbol{\eta_2'}} \|_{L_{\tau \xi \eta}^2},
\end{align*}
where $E(\tau, \xi, \eta) \subset \R^3$ is defined by
\begin{equation*}
E(\tau, \xi, \eta) := \{ (\tau_1, \xi_1, \eta_1) \in \supp (f_{\boldsymbol{\eta_1'}})
\, | \, (\tau-\tau_1, \xi- \xi_1, \eta-\eta_1) \in \supp (g_{\boldsymbol{\eta_2'}}) \}.
\end{equation*}
Thus, it suffices to show
\begin{equation}
\sup_{(\tau, \xi, \eta) \in G_{N_0, L_0}} |E(\tau, \xi, \eta)| \lesssim (M N_1)^{-1} L_1 L_2.\label{est03-prop3.10}
\end{equation}
For fixed $(\xi_1, \eta_1)$, we easily have
\begin{equation}
\sup_{(\tau, \xi, \eta) \in G_{N_0, L_0}} | \{ \tau_1 \, | \, (\tau_1, \xi_1, \eta_1) \in E(\tau, \xi, \eta) \}| 
\lesssim \min(L_1, L_2).\label{est04-prop3.10}
\end{equation}
Let $C(\xi,\boldsymbol{\eta'}, \xi_1, \boldsymbol{\eta_1'})= (\xi_1+\eta_1) |\boldsymbol{\eta_1'}|^2 
+ (\xi-\xi_1+\eta-\eta_1)|\boldsymbol{\eta'}-\boldsymbol{\eta_1'}|^2$. We observe
\begin{align*}
 \max (L_1, L_2)  
\gtrsim & |(\tau_1 - \xi_1^3 - \eta_1^3) 
+ (\tau- \tau_1) - (\xi- \xi_1)^3 - (\eta- \eta_1)^3-C(\xi,\boldsymbol{\eta}, \xi_1, \boldsymbol{\eta_1}) |\\
 = & |(\tau- \xi^3 -\eta^3) + 3(\xi \xi_1 (\xi-\xi_1) + \eta \eta_1 (\eta- \eta_1)) 
-C(\xi,\boldsymbol{\eta}, \xi_1, \boldsymbol{\eta_1})|.
\end{align*} 
Thus, we deduce from 
$ |\partial_{\xi_1} \left( \xi \xi_1 (\xi-\xi_1) \right)| = |\xi_1^2 - (\xi - \xi_1)^2| \gtrsim N_1^{2}$, 
$|\boldsymbol{\eta_1'}| \lesssim A^{-1} N_1$ and $|\boldsymbol{\eta'} - \boldsymbol{\eta_1'}| \lesssim A^{-1} N_1$ 
that, for fixed $\eta_1$, it holds that
\begin{equation}
\sup_{(\tau, \xi, \eta) \in G_{N_0, L_0}} | \{ \xi_1 \, | \, (\tau_1, \xi_1, \eta_1) \in E(\tau, \xi, \eta) \}| 
\lesssim N_1^{-2} \max(L_1, L_2).\label{est05-prop3.10}
\end{equation}
Lastly, since $(\tau_1, \xi_1, \eta_1) \in \supp (f_{\boldsymbol{\eta_1'}})$ implies 
$(\xi_1,\eta_1) \in {\mathcal{T}}_{\ell_1}^M$, we have
\begin{equation}
\sup_{(\tau, \xi, \eta) \in G_{N_0, L_0}} | \{ \eta_1 \, | \, (\tau_1, \xi_1, \eta_1) \in E(\tau, \xi, \eta) \}| 
\lesssim M^{-1} N_1.\label{est06-prop3.10}
\end{equation}
The estimates \eqref{est04-prop3.10}-\eqref{est06-prop3.10} complete the proof of \eqref{est03-prop3.10}.

Next we show \eqref{est02-prop3.10} under the assumption 
$|\overline{F} (\xi_1, \eta_1, \xi_2 ,\eta_2)| \geq M^{-1} N_1^2$ by following the proof for Proposition 3.5 in \cite{Ki2019}. 
By Fubini's theorem, \eqref{est02-prop3.10} reduces to
\begin{align}
& \left| \int h_{\boldsymbol{\eta'}} (\varphi_{\boldsymbol{\eta_1'},c_1} (\xi_1, \eta_1) + \varphi_{\boldsymbol{\eta_2'},c_2} (\xi_2, \eta_2))  
f_{\boldsymbol{\eta_1'}} (\varphi_{\boldsymbol{\eta_1'},c_1} (\xi_1, \eta_1) ) g (\varphi_{\boldsymbol{\eta_2'},c_2}(\xi_2, \eta_2)) 
d \xi_1d \eta_1 d\xi_2 d\eta_2 \right| \notag\\
& \qquad \qquad \qquad \qquad   \lesssim  M^{\frac{1}{2}}N_1^{-2} 
\| f_{\boldsymbol{\eta_1'}} \circ \varphi_{\boldsymbol{\eta_1'},c_1}\|_{L_{\xi \eta}^2} \|g_{\boldsymbol{\eta_2'}} 
\circ \varphi_{\boldsymbol{\eta_2'},c_2} \|_{L_{\xi \eta}^2} 
\|h_{\boldsymbol{\eta'}} \|_{L_{\tau \xi \eta }^2}, \label{est04-prop1.8a}
\end{align}
where $h_{\boldsymbol{\eta'}} (\tau, \xi, \eta)$ 
is supported in $c_0 \leq \tau - \xi^3 - \eta^3 - (\xi+\eta) |\boldsymbol{\eta'}|^2 \leq c_0 +1$ and 
\begin{equation*}
\varphi_{\boldsymbol{\eta_j'}, c_j} (\xi,\eta) = ( \xi^3 + \eta^3 + (\xi+\eta) |\boldsymbol{\eta_j'}|^2 + c_j, \, \xi, \, \eta) \quad \textnormal{for} \ j=1,2.
\end{equation*}
Note that if $\boldsymbol{\eta_1'}=\boldsymbol{\eta_2'}=\boldsymbol{\eta'}=0$, \eqref{est04-prop1.8a} corresponds exactly to the inequality (3.17) in \cite{Ki2019}. 
Similarly to the proof of Proposition \ref{prop3.4}, we define
\begin{align*}
 \tilde{f}_{\boldsymbol{\eta_1'}} (\tau_1, \xi_1 , \eta_1) & = f_{\boldsymbol{\eta_1'}} (N_1^3 \tau_1 , N_1 \xi_1, N_1 \eta_1), \\
 \tilde{g}_{\boldsymbol{\eta_2'}} (\tau_2, \xi_2, \eta_2) & = g_{\boldsymbol{\eta_2'}} (N_1^3 \tau_2, N_1 \xi_2, N_1 \eta_2), \\
 \tilde{h}_{\boldsymbol{\eta'}} (\tau, \xi, \eta) & = h_{\boldsymbol{\eta'}} (N_1^3 \tau, N_1 \xi, N_1 \eta), 
\end{align*}
and prove
\begin{equation}
\| \tilde{f}_{\boldsymbol{\eta_1'}} |_{\overline{S}_1} * \tilde{g}_{\boldsymbol{\eta_2'}} |_{\overline{S}_2} 
\|_{L^2(\overline{S}_3)} \lesssim M^{\frac{1}{2}}  
\| \tilde{f}_{\boldsymbol{\eta_1'}} \|_{L^2(\overline{S}_1)} 
\| \tilde{g}_{\boldsymbol{\eta_2'}} \|_{L^2(\overline{S}_2)},\label{est06-prop1.8a}
\end{equation}
where $\boldsymbol{\tilde{\eta}_j} = N_1^{-1} \boldsymbol{\eta}_j'$, 
$\boldsymbol{\tilde{\eta}} = N_1^{-1}\boldsymbol{\eta'}$, 
$\tilde{c_j} = N_1^{-3} c_j$ and 
\begin{align*}
\overline{S}_1 =&  \{ \varphi_{\boldsymbol{\tilde{\eta}_1},\tilde{c}_1} (\xi_1, \eta_1) \in \R^3 \ | \ 
(\xi_1, \eta_1) \in \supp (\tilde{f}_{\boldsymbol{\eta_1'}}\circ \varphi_{\boldsymbol{\tilde{\eta}_1}, \tilde{c}_1})\}, \\
\overline{S}_2 =&  \{ \varphi_{\boldsymbol{\tilde{\eta}_2}, \tilde{c}_2}(\xi_2,\eta_2) \in \R^3 \ | \ (\xi_2, \eta_2) \in \supp (\tilde{g}_{\boldsymbol{\eta_2'}}\circ \varphi_{\boldsymbol{\tilde{\eta}_2}, \tilde{c}_2}) \},\\
\overline{S}_3 = & \left\{ (\varphi_{\tilde{\boldsymbol{\eta}}}  (\xi,\eta), \xi, \eta) \in \R^3  \ | \ 
\  \varphi_{\tilde{\boldsymbol{\eta}}} (\xi,\eta) = \xi^3 + \eta^3 + (\xi+\eta) |\boldsymbol{\tilde{\eta}}|^2 + \frac{c_0'}{N_1^{3}} \right\}.
\end{align*}
Clearly, $\overline{S}_1$, $\overline{S}_2$, $\overline{S}_3$ satisfy necessary regularity and diameter conditions to apply the nonlinear Loomis-Whitney inequality. 
Define $\lambda_i \in \overline{S}_i$ as
\begin{equation*}
\lambda_1=\varphi_{\boldsymbol{\tilde{\eta}_1},\tilde{c}_1} (\xi_1, \eta_1), 
\quad \lambda_2 =   \varphi_{\boldsymbol{\tilde{\eta}_2}, \tilde{c}_2}(\xi_2,\eta_2), \quad
\lambda_3 = (\varphi_{\tilde{\boldsymbol{\eta}}}  (\xi,\eta), \xi, \eta),
\end{equation*}
then the unit normals ${\mathfrak{n}}_i$ on $\lambda_i$ can be described explicitly as
\begin{equation*}
{\mathfrak{n}}_i(\lambda_i) = 
\frac{1}{\sqrt{1+ (3  \xi_i^2 + |\boldsymbol{\tilde{\eta}_i}|^2)^2 + 
(3 \eta_i^2 +|\boldsymbol{\tilde{\eta}_i}|^2)^2}} 
\left(-1, \ 3  \xi_i^2 + |\boldsymbol{\tilde{\eta}_i}|^2, \ 3 \eta_i^2 +|\boldsymbol{\tilde{\eta}_i}|^2\right), 
\end{equation*}
for $i=1$, $2$, and the same for ${\mathfrak{n}}_3(\lambda_3)$. 
We define
\begin{equation*}
{\mathfrak{n}}_i^0(\lambda_i) = \frac{1}{\sqrt{1+ 9 \xi_i^4 + 9 \eta_i^4}} 
\left(-1, \ 3  \xi_i^2, \ 3 \eta_i^2 \right),
\end{equation*}
for $i=1$, $2$, and the same for ${\mathfrak{n}}_3^0(\lambda_3)$. 
Since $|\boldsymbol{\tilde{\eta}_j}| \lesssim A^{-1}$, $|\boldsymbol{\tilde{\eta}}| \lesssim A^{-1}$, we easily get $|{\mathfrak{n}}_j (\lambda_j) - {\mathfrak{n}}_j^0(\lambda_j)| \ll A^{-1} \leq M^{-1}$. 
Therefore we only need to show
\begin{equation*}
 |\textnormal{det} ({\mathfrak{n}}_1^0(\lambda_1), {\mathfrak{n}}_2^0(\lambda_2) , {\mathfrak{n}}_3^0(\lambda_3) )| \gtrsim M^{-1}
\end{equation*}
with $(\xi_1, \eta_1) + (\xi_2, \eta_2) = (\xi, \eta)$. We calculate
\begin{align*}
 |\textnormal{det} ({\mathfrak{n}}_1^0(\lambda_1), {\mathfrak{n}}_2^0(\lambda_2) , {\mathfrak{n}}_3^0(\lambda_3) )| \gtrsim & 
\left|\textnormal{det}
\begin{pmatrix}
-1 & -1 & - 1 \\
 3 \xi_1^2  &  3 \xi_2^2  & 3 \xi^2 \\
 3 \eta_1^2   & 3 \eta_2^2  & 3 \eta^2
\end{pmatrix} \right| \notag \\
\gtrsim & \, |\xi_1 \eta_2 - \xi_2 \eta_1 |
| \xi_1 \eta_2 +  \xi_2 \eta_1 + 2 (\xi_1 \eta_1 + \xi_2 \eta_2)| \\
\gtrsim & \, M^{-1}.
\end{align*}
Here we used $|\sin \angle \left( (\xi_1, \eta_1), (\xi_2, \eta_2) \right)| \gtrsim 1$, and 
$(\xi_1, \eta_1) \in 
\supp (\tilde{f}_{\boldsymbol{\eta_1'}}\circ \varphi_{\boldsymbol{\tilde{\eta}_1}, \tilde{c}_1})$, 
$ (\xi_2, \eta_2) \in \supp (\tilde{g}_{\boldsymbol{\eta_2'}}\circ \varphi_{\boldsymbol{\tilde{\eta}_2}, \tilde{c}_2})$ which implies $|\overline{F} (\xi_1, \eta_1, \xi_2 ,\eta_2)| \geq M^{-1}$.
\end{proof}
%%%%%%%%%%%%%%%%%%%%%%%%%%%%%%%%%%%%%
%%%%%%%%%%%%%%%%%%%%%%%%%%%%%%%%%%%%%
%%%%%%%%%%%%%%%%%%%%%%%%%%%%%%%%%%%%%
%%%%%%%%%%%%%%%%%%%%%%%%%%%%%%%%%%%%%
%%%%%%%%%%%%%%%%%%%%%%%%%%%%%%%%%%%%%
%%%%%%%%%%%%%%%%%%%%%%%%%%%%%%%%%%%%%
%%%%%%%%%%%%%%%%%%%%%%%%%%%%%%%%%%%%%
%%%%%%%%%%%%%%%%%%%%%%%%%%%%%%%%%%%%%
The key ingredient to show \eqref{est01-prop3.9} is the almost orthogonality of $\ell_1$ and $\ell_2$ which satisfy $(\ell_1, \ell_2) \in \tilde{Z}_M$. 
However, in \cite{Ki2019} it was found that there exist pairs of tiles which do not satisfy the almost orthogonality. 
Thus we perform the decompositions which was introduced in
\cite{Ki2019}, see \cite[Remark 3.3]{Ki2019} for the details.
\begin{defn}[Def.~3 in \cite {Ki2019}]\label{definition5}
Let $\mathcal{K}_0$, $\mathcal{K}_1$, $\mathcal{K}_2$, $\mathcal{K}_0'$, $\mathcal{K}_1'$, 
$\mathcal{K}_2' \subset \R^2$ and  $\tilde{\mathcal{K}}_0$, $\tilde{\mathcal{K}}_1$, 
$\tilde{\mathcal{K}}_2$, 
$\tilde{\mathcal{K}}_0'$, 
$\tilde{\mathcal{K}}_1'$, $\tilde{\mathcal{K}}_2' \subset \R^{d+1}$ be defined as follows:
\begin{align*}
\mathcal{K}_0 & = \left\{ (\xi, \eta) \in \R^2 \, | \, \left| 
\eta -(\sqrt{2} - 1)^{\frac{4}{3}} \xi \right| 
\leq 2^{-20} N_1 \right\},\\
\mathcal{K}_1 & = \left\{ (\xi, \eta) \in \R^2 \, | \, \left| 
\eta - ( \sqrt{2}+ 1 )^{\frac{2}{3}} (\sqrt{2} + \sqrt{3} ) \xi \right| 
\leq 2^{-20} N_1 \right\},\\
\mathcal{K}_2 & = \left\{ (\xi, \eta) \in \R^2 \, | \, \left| 
\eta + ( \sqrt{2}+ 1 )^{\frac{2}{3}} (\sqrt{3} - \sqrt{2} ) \xi \right| 
\leq 2^{-20} N_1 \right\},\\
\mathcal{K}_0' & =  \left\{ (\xi, \eta) \in \R^2 \ | \ 
 (\eta, \xi) \in \mathcal{K}_0  \right\},\\
\mathcal{K}_1' & = \left\{ (\xi, \eta) \in \R^2 \ | \ (\eta, \xi) \in \mathcal{K}^1 \right\},\\
\mathcal{K}_2' & = \left\{ (\xi, \eta) \in \R^2 \ | \ (\eta, \xi) \in \mathcal{K}^2 \right\},\\
\tilde{\mathcal{K}}_i & = \R \times \mathcal{K}_i \times \R^{d-2}, \quad 
\tilde{\mathcal{K}}_i' = \R \times \mathcal{K}_i' \times \R^{d-2}
 \ \ \textnormal{for} \ i = 0, 1,2.
\end{align*}
We define the subsets of $\R^2 \times \R^2$ and $\R^{d+1} \times \R^{d+1}$ as
\begin{align*}
\mathcal{K} \, = & ( \mathcal{K}_0 \times ( \mathcal{K}_1\cup \mathcal{K}_2 ) ) \cup 
( ( \mathcal{K}_1\cup \mathcal{K}_2 ) \times  \mathcal{K}_0 ) \subset \R^2 \times \R^2,\\
\tilde{\mathcal{K}} \, = & ( \tilde{\mathcal{K}}_0 \times ( \tilde{\mathcal{K}}_1 
\cup \tilde{\mathcal{K}}_2 ) ) \cup 
( 
( \tilde{\mathcal{K}}_1\cup \tilde{\mathcal{K}}_2 ) \times  \tilde{\mathcal{K}}_0 ) \subset \R^{d+1} 
\times \R^{d+1},\\
\mathcal{K}' = & \left( \mathcal{K}_0' \times \left( \mathcal{K}_1' \cup \mathcal{K}_2' \right) \right) \cup 
\left( 
\left( \mathcal{K}_1' \cup \mathcal{K}_2' \right) \times  \mathcal{K}_0' \right) \subset \R^2 \times \R^2,\\
\tilde{\mathcal{K}}'  = & ( \tilde{\mathcal{K}}_0' \times ( \tilde{\mathcal{K}}_1' 
\cup \tilde{\mathcal{K}}_2' ) ) \cup 
( 
( \tilde{\mathcal{K}}_1'\cup \tilde{\mathcal{K}}_2' ) \times  \tilde{\mathcal{K}}_0' ) \subset \R^{d+1} 
\times \R^{d+1},
\end{align*}
and their complements as 
\begin{align*}
(\mathcal{K})^c & =  (\R^2 \times \R^2) \setminus \mathcal{K} , \quad \ 
(\tilde{\mathcal{K}})^c =  (\R^{d+1} \times \R^{d+1}) \setminus \tilde{\mathcal{K}}\\
(\mathcal{K}')^c & = (\R^2 \times \R^2) \setminus \mathcal{K}', \quad 
(\tilde{\mathcal{K}}')^c = (\R^{d+1} \times \R^{d+1}) \setminus \tilde{\mathcal{K}}'.
\end{align*}
Lastly, we define
\begin{equation*}
\widehat{Z}_{M} = \{ (\ell_1, \ell_2) \in \tilde{Z}_{M} \, | \, 
\left( \mathcal{T}_{\ell_1}^M \times \mathcal{T}_{\ell_2}^M\right) \cap \left( 
(\mathcal{K})^c \cap (\mathcal{K}')^c \right) \not= 
\emptyset \},
\end{equation*}
and $\overline{Z}_{M}$ as the collection of $(\ell_1, \ell_2) \in \Z^2 \times \Z^2$ which satisfies 
\begin{align*}
\mathcal{T}_{\ell_1}^{M} & \times \mathcal{T}_{\ell_2}^{M} \not\subset \bigcup_{\widehat{M} \leq M' \leq M} 
\bigcup_{(\ell_1', \ell_2') \in \widehat{Z}_{M}} \left( \mathcal{T}_{\ell_1'}^{M'} \times \mathcal{T}_{\ell_2'}^{M'} \right),
\\ 
 \left( \tilde{\mathcal{T}}_{\ell_1}^{M} \times \tilde{\mathcal{T}}_{\ell_2}^{M} \right) & \cap \left( 
G_{N_1, L_1} \times G_{N_2, L_2} \right) \cap \mathcal{A}  \cap \left( 
(\tilde{\mathcal{K}})^c \cap (\tilde{\mathcal{K}}')^c \right) \not= 
\emptyset. 
\end{align*}
\end{defn}
%%%%%%%%%%%%%%%%%%%%%%%%%%%%%%%%%%%%%
%%%%%%%%%%%%%%%%%%%%%%%%%%%%%%%%%%%%%
%%%%%%%%%%%%%%%%%%%%%%%%%%%%%%%%%%%%%
%%%%%%%%%%%%%%%%%%%%%%%%%%%%%%%%%%%%%
%%%%%%%%%%%%%%%%%%%%%%%%%%%%%%%%%%%%%
%%%%%%%%%%%%%%%%%%%%%%%%%%%%%%%%%%%%%
%%%%%%%%%%%%%%%%%%%%%%%%%%%%%%%%%%%%%
%%%%%%%%%%%%%%%%%%%%%%%%%%%%%%%%%%%%%
%%%%%%%%%%%%%%%%%%%%%%%%%%%%%%%%%%%%%
%%%%%%%%%%%%%%%%%%%%%%%%%%%%%%%%%%%%%
%%%%%%%%%%%%%%%%%%%%%%%%%%%%%%%%%%%%%
%%%%%%%%%%%%%%%%%%%%%%%%%%%%%%%%%%%%%
\begin{lem}[Lemma 3.7 in \cite{Ki2019}]\label{lemma3.11}
For fixed $\ell_1 \in \Z^2$, the number of $\ell_2 \in \Z^2$ such that 
$(\ell_1, \ell_2) \in \widehat{Z}_{M}$ is finite (uniformly bounded). Furthermore, the same claim holds true if we replace $\widehat{Z}_{M}$ by $\overline{Z}_{M}$.
\end{lem}
%%%%%%%%%%%%%%%%%%%%%%%%%%%%%%%%%%%%%
%%%%%%%%%%%%%%%%%%%%%%%%%%%%%%%%%%%%%
%%%%%%%%%%%%%%%%%%%%%%%%%%%%%%%%%%%%%
%%%%%%%%%%%%%%%%%%%%%%%%%%%%%%%%%%%%%
%%%%%%%%%%%%%%%%%%%%%%%%%%%%%%%%%%%%%
Now we show \eqref{est01-prop3.9} under the assumption 
$(\xi_1, \eta_1) \times (\xi_2, \eta_2) \in (\mathcal{K})^c \cap (\mathcal{K}')^c$.
\begin{prop}\label{prop3.12}
Assume the same conditions as in Proposition \ref{prop3.9}. Suppose further that 
$| \sin \angle \left( (\xi_1, \eta_1), (\xi_2, \eta_2) \right)| \gtrsim 1$ and 
$(\xi_1, \eta_1) \times (\xi_2, \eta_2) \in (\mathcal{K})^c \cap (\mathcal{K}')^c$. Then we have
\begin{equation}
\begin{split}
& \left|\int_{*}{  h_{{N_0, L_0}}(\tau, \xi, \boldsymbol{\eta} ) 
f_{N_1, L_1} (\tau_1, \xi_1, \boldsymbol{\eta_1} )  
g_{N_2, L_2} (\tau_2, \xi_2, \boldsymbol{\eta_2} ) 
}
d\sigma_1 d\sigma_2 \right| \\
& \qquad  \lesssim  A^{-\frac{d-3}{2}} 
N_1^{\frac{d-6}{2}}    (L_0 L_1 L_2)^{\frac{1}{2}} 
\|f_{N_1, L_1} \|_{L^2} \| g_{N_2, L_2} \|_{L^2} 
\|h_{{N_0, L_0}}\|_{L^2},\label{est01-prop3.12}
\end{split}
\end{equation}
where functions $f_{N_1, L_1}$, $g_{N_2, L_2}$, $h_{{N_0, L_0}}$ satisfy 
\eqref{assumption-fgh}.
\end{prop}
%%%%%%%%%%%%%%%%%%%%%%%%%%%%%%%%%%%%%
%%%%%%%%%%%%%%%%%%%%%%%%%%%%%%%%%%%%%
%%%%%%%%%%%%%%%%%%%%%%%%%%%%%%%%%%%%%
%%%%%%%%%%%%%%%%%%%%%%%%%%%%%%%%%%%%%
%%%%%%%%%%%%%%%%%%%%%%%%%%%%%%%%%%%%%
\begin{proof}
By the definitions of $\widehat{Z}_M$ and $ \overline{Z}_{A}$, 
$\left( 
G_{N_1, L_1} \times G_{N_2, L_2} \right) \cap \mathcal{A}  \cap 
(\tilde{\mathcal{K}})^c \cap (\tilde{\mathcal{K}}')^c $ are contained in
\begin{equation*}
\bigcup_{\widehat{M} \leq M \leq A} 
\bigcup_{(\ell_1, \ell_2) \in \widehat{Z}_M} 
 \left( \tilde{\mathcal{T}}_{\ell_1}^{M} \times \tilde{\mathcal{T}}_{\ell_2}^{M} \right)
 \cup \bigcup_{(\ell_1, \ell_2) \in \overline{Z}_{A}} 
 \left( \tilde{\mathcal{T}}_{\ell_1}^{A} \times \tilde{\mathcal{T}}_{\ell_2}^{A} \right).
\end{equation*}
Therefore, we get
\begin{align*}
& \textnormal{(LHS) of \eqref{est01-prop3.12}} \\ 
\leq & \sum_{\widehat{M} \leq M \leq A} \sum_{(\ell_1, \ell_2) \in \widehat{Z}_M} 
\left|\int_{*}{  h_{N_0,L_0}(\tau, \xi, \boldsymbol{\eta}) 
f_{N_1,L_1}|_{\tilde{\mathcal{T}}_{\ell_1}^M}(\tau_1, \xi_1, \boldsymbol{\eta_1})  g_{N_2, L_2}|_{\tilde{\mathcal{T}}_{\ell_2}^M}(\tau_2, \xi_2, \boldsymbol{\eta_2}) 
}
d\sigma_1 d\sigma_2 \right|\\
& \quad + \sum_{(\ell_1, \ell_2) \in \overline{Z}_{A}} 
\left|\int_{*}{  h_{N_0,L_0}(\tau, \xi, \boldsymbol{\eta}) 
f_{N_1,L_1}|_{\tilde{\mathcal{T}}_{\ell_1}^{A}}(\tau_1, \xi_1, \boldsymbol{\eta_1})
  g_{N_2, L_2}|_{\tilde{\mathcal{T}}_{\ell_2}^{A}}(\tau_2, \xi_2, \boldsymbol{\eta_2}) 
}
d\sigma_1 d\sigma_2 \right|\\
& =: \sum_{\widehat{M} \leq M \leq A} \sum_{(\ell_1, \ell_2) \in \widehat{Z}_M}  I_1 + \sum_{(\ell_1, \ell_2) \in \overline{Z}_{A}}  I_2.
\end{align*}
For the former term, we deduce from Proposition \ref{prop3.10} and Lemma \ref{lemma3.11} that
\begin{align*}
 \sum_{(\ell_1, \ell_2) \in \widehat{Z}_M}  I_1  \lesssim & 
\sum_{(\ell_1, \ell_2) \in \widehat{Z}_M} 
A^{-\frac{d-2}{2}} M^{\frac{1}{2}}
N_1^{\frac{d-6}{2}}    (L_0 L_1 L_2)^{\frac{1}{2}}  \|f_{N_1,L_1}|_{\tilde{\mathcal{T}}_{\ell_1}^M} \|_{L^2} 
\| g_{N_2, L_2}|_{\tilde{\mathcal{T}}_{\ell_2}^M} \|_{L^2} 
\|h_{N_0,L_0} \|_{L^2}\\
\lesssim & \ 
A^{-\frac{d-2}{2}} M^{\frac{1}{2}}
N_1^{\frac{d-6}{2}}    (L_0 L_1 L_2)^{\frac{1}{2}} \|f_{N_1,L_1} \|_{L^2} 
\| g_{N_2, L_2} \|_{L^2} 
\|h_{N_0,L_0} \|_{L^2}.
\end{align*}
Consequently, we obtain
\begin{equation*}
\sum_{\widehat{M} \leq M \leq A} \sum_{(\ell_1, \ell_2) \in \widehat{Z}_M}   I_1 
 \lesssim  
 A^{-\frac{d-3}{2}} 
N_1^{\frac{d-6}{2}}  (L_0 L_1 L_2)^{\frac{1}{2}} \|f_{N_1,L_1} \|_{L^2} 
\| g_{N_2, L_2} \|_{L^2} 
\|h_{N_0,L_0} \|_{L^2}.
\end{equation*}
Next we consider the latter term. 
The assumption $|\boldsymbol{\eta_j'}| \lesssim A^{-1} N_1$ implies that space variables of 
$\supp (f_{N_1,L_1}|_{\tilde{\mathcal{T}}_{\ell_1}^{A}})$ and 
$\supp(g_{N_2, L_2}|_{\tilde{\mathcal{T}}_{\ell_2}^{A}})$ are confined to regular cubes which side lengths are 
comparable to $A^{-1} N_1$, respectively. Since the linear transformation 
$(\xi_j, \eta_j) \to (\xi_j+\eta_j, \sqrt{3} (\xi_j-\eta_j))$ is invertible, Proposition \ref{prop3.7} yields
\begin{equation*}
I_2 \lesssim A^{-\frac{d-3}{2}} 
N_1^{\frac{d-6}{2}}  (L_0 L_1 L_2)^{\frac{1}{2}}  \|f_{N_1,L_1}|_{\tilde{\mathcal{T}}_{\ell_1}^{A}} \|_{L^2} 
\| g_{N_2, L_2}|_{\tilde{\mathcal{T}}_{\ell_2}^{A}} \|_{L^2} \|h_{N_0,L_0} \|_{L^2}.
\end{equation*}
Hence, by Lemma 
\ref{lemma3.11}, we get
\begin{align*}
\sum_{(\ell_1, \ell_2) \in \overline{Z}_{A}}  I_2 & \lesssim 
A^{-\frac{d-3}{2}} 
N_1^{\frac{d-6}{2}}  (L_0 L_1 L_2)^{\frac{1}{2}} \sum_{(\ell_1, \ell_2) \in \overline{Z}_{A}} \|f_{N_1,L_1}|_{\tilde{\mathcal{T}}_{\ell_1}^{A}} \|_{L^2} 
\| g_{N_2, L_2}|_{\tilde{\mathcal{T}}_{\ell_2}^{A}} \|_{L^2} \|h_{N_0,L_0} \|_{L^2}\\
& \lesssim A^{-\frac{d-3}{2}} 
N_1^{\frac{d-6}{2}}  (L_0 L_1 L_2)^{\frac{1}{2}} \|f_{N_1,L_1} \|_{L^2} 
\| g_{N_2, L_2} \|_{L^2} 
\|h_{N_0,L_0} \|_{L^2}.
\end{align*}
This completes the proof.
\end{proof}
%%%%%%%%%%%%%%%%%%%%%%%%%%%%%%%%%%%%%
%%%%%%%%%%%%%%%%%%%%%%%%%%%%%%%%%%%%%
%%%%%%%%%%%%%%%%%%%%%%%%%%%%%%%%%%%%%
%%%%%%%%%%%%%%%%%%%%%%%%%%%%%%%%%%%%%
%%%%%%%%%%%%%%%%%%%%%%%%%%%%%%%%%%%%%
Next we deal with the case 
$(\xi_1 , \eta_1) \times (\xi_2, \eta_2) \in \left( \mathcal{K} \cup \mathcal{K}' \right)$. 
The strategy of proof is the same as for the case $(\xi_1, \eta_1) \times (\xi_2, \eta_2) \in (\mathcal{K})^c \cap (\mathcal{K}')^c$. 
By symmetry, it suffices to show the estimate \eqref{est01-prop3.9} for the case $(\xi_1 , \eta_1) \times (\xi_2, \eta_2) \in (\mathcal{K}_1 \cup \mathcal{K}_2 ) \times \mathcal{K}_0$.
%%%%%%%%%%%%%%%%%%%%%%%%%%%%%%%%%%%%%
%%%%%%%%%%%%%%%%%%%%%%%%%%%%%%%%%%%%%
%%%%%%%%%%%%%%%%%%%%%%%%%%%%%%%%%%%%%
%%%%%%%%%%%%%%%%%%%%%%%%%%%%%%%%%%%%%
%%%%%%%%%%%%%%%%%%%%%%%%%%%%%%%%%%%%%
%%%%%%%%%%%%%%%%%%%%%%%%%%%%%%%%%%%%%
\begin{defn}[Def.~4 \cite{Ki2019}]\label{definition6}
Let $m=(n, z) \in \N  \times \Z$. We define the increasing sequence $\{ a_{M,n} \}_{n \in \N} $ as 
\begin{equation*}
a_{M,1} = 0, \qquad a_{M,n+1} = a_{M,n} + \frac{N_1}{\sqrt{(n+1)M}}.
\end{equation*}
and sets $\mathcal{R}_{M,m,1}$, 
$\mathcal{R}_{M,m,2}$ as follows:
\begin{align*}
\mathcal{R}_{M,m,1} = &
\left\{ (\xi, \eta)  \in \R^2  \, \left| \, 
\begin{aligned} & a_{M,n} \leq | \eta-   ( \sqrt{2}+ 1 )^{\frac{2}{3}} (\sqrt{2} + \sqrt{3} ) \xi | 
< a_{M,n+1}, \\ 
& z M^{-1}  N_1 \leq \eta-( \sqrt{2}+ 1 )^{\frac{2}{3}}\xi < (z+1) M^{-1}N_1
 \end{aligned} \right.
\right\},\\
\mathcal{R}_{M,m,2} = &
\left\{ (\xi, \eta)  \in \R^2  \, \left| \, 
\begin{aligned} & a_{M,n} \leq | \eta +  ( \sqrt{2}+ 1 )^{\frac{2}{3}} (\sqrt{3} - \sqrt{2} ) \xi | 
< a_{M,n+1}, \\ 
& z M^{-1}  N_1 \leq \eta-( \sqrt{2}+ 1 )^{\frac{2}{3}}\xi < (z+1) M^{-1}N_1
 \end{aligned} \right.
\right\}\\
\tilde{\mathcal{R}}_{M,m,1} = & \R \times \mathcal{R}_{M,m,1} \times \R^{d-2}, \quad 
\tilde{\mathcal{R}}_{M,m,2} =  \R \times \mathcal{R}_{M,m,2} \times \R^{d-2}.
\end{align*}
%%%%%%%%%%%%%%%%%%%%%%%%%%%%%%%%%%%%%
%%%%%%%%%%%%%%%%%%%%%%%%%%%%%%%%%%%%%
%%%%%%%%%%%%%%%%%%%%%%%%%%%%%%%%%%%%%
%%%%%%%%%%%%%%%%%%%%%%%%%%%%%%%%%%%%%
%%%%%%%%%%%%%%%%%%%%%%%%%%%%%%%%%%%%%
%%%%%%%%%%%%%%%%%%%%%%%%%%%%%%%%%%%%%
%%%%%%%%%%%%%%%%%%%%%%%%%%%%%%%%%%%%%
%%%%%%%%%%%%%%%%%%%%%%%%%%%%%%%%%%%%%
We will perform the Whitney type decomposition by using the above sets 
instead of simple square tiles. 
We define for $i=1,2$ that
\begin{align*}
M_{M,i}^1 & = \left\{ (m, \ell) \in (\N \times \Z) \times \Z^2 \ \left| \ 
\begin{aligned} & |\overline{\Phi}(\xi_1, \eta_1, \xi_2, \eta_2)| \geq M^{-1} N_1^3 \ \\ 
&\textnormal{for any} \ (\xi_1, \eta_1) \in \mathcal{R}_{M,m,i} \  \textnormal{and} \
(\xi_2, \eta_2) \in \mathcal{T}_{\ell}^M 
 \end{aligned} \right.
\right\},\\
M_{M,i}^2 & = \left\{ (m, \ell) \in (\N \times \Z) \times \Z^2 \ \left| \ 
\begin{aligned} & |\overline{F}(\xi_1, \eta_1, \xi_2, \eta_2)| \geq M^{-1} N_1^3 \ \\ 
&\textnormal{for any} \ (\xi_1, \eta_1) \in \mathcal{R}_{M,m,i} \  \textnormal{and} \
(\xi_2, \eta_2) \in \mathcal{T}_{\ell}^M 
 \end{aligned} \right.
\right\},\\
M_{M,i} & = M_{M,i}^1 \cup M_{M,i}^2 \subset  (\N \times \Z) \times \Z^2, \\
R_{M,i} & = \bigcup_{(m, \ell) \in M_{M,i}} \mathcal{R}_{M,m,i} \times 
\mathcal{T}_{\ell}^M \subset \R^2 \times \R^2.
\end{align*}
Furthermore, we define ${M}_{M,i}' \subset M_{M,i}$ as the collection of 
$(m,\ell) \in (\N \times \Z) \times \Z^2$ such that 
\begin{equation*}
\mathcal{R}_{M,m,i} \times 
\mathcal{T}_{\ell}^M \subset \bigcup_{\widehat{M} \leq M' <M} R_{M', i}.
\end{equation*}
By using ${M}_{M,i}'$, we define
\begin{align*}
Q_{M,i} = 
\begin{cases}
R_{M,i} \setminus {\displaystyle \bigcup_{(m,\ell) \in {M}_{M,i}'}} (\mathcal{R}_{M,m,i} \times 
\mathcal{T}_{\ell}^M) \ \textnormal{for} \   M > \widehat{M},\\
 \quad \  R_{\widehat{M},i}  \, \quad \qquad \qquad \qquad \qquad \textnormal{for} \  M = \widehat{M},
\end{cases}
\end{align*}
and $\tilde{M}_{M,i} = M_{M,i} \setminus {M}_{M,i}'$. Clearly, the followings hold.
\begin{equation*}
\bigcup_{(m,\ell) \in \tilde{M}_{M,i}} \mathcal{R}_{M,m,i} \times 
\mathcal{T}_{\ell}^M = Q_{M,i}, \quad 
\bigcup_{\widehat{M} \leq M \leq M_0} Q_{M,i} = R_{M_0,i},
\end{equation*}
where $M_0 \geq \widehat{M}$ is dyadic. 
Lastly, we define
\begin{align*}
\widehat{Z}_{M,i} & = \{ (m, \ell) \in \tilde{M}_{M,i} \, | \, 
(  \tilde{\mathcal{R}}_{M,m,i} \times 
\tilde{\mathcal{T}}_{\ell}^M ) \cap \left( 
G_{N_1, L_1} \times G_{N_2, L_2} \right) \cap ( \tilde{\mathcal{K}}_i \times 
\tilde{\mathcal{K}}_0 ) \not= 
\emptyset \},\\
\overline{Z}_{M,i} & = \{ (m, \ell) \in  M_{M,i}^c \, | \, 
(  \tilde{\mathcal{R}}_{M,m,i} \times 
\tilde{\mathcal{T}}_{\ell}^M ) \cap \left( 
G_{N_1, L_1} \times G_{N_2, L_2} \right) \cap ( \tilde{\mathcal{K}}_i \times 
\tilde{\mathcal{K}}_0 )  \not= 
\emptyset \},
\end{align*}
where $M_{M,i}^c = (\N \times \Z) \times \Z^2 \setminus M_{M,i}$. We easily see that
\begin{equation*}
 \left( 
G_{N_1, L_1} \times G_{N_2, L_2} \right) \cup ( \tilde{\mathcal{K}}_i \times 
\tilde{\mathcal{K}}_0  ) \subset \bigcup_{(m,\ell) \in \widehat{Z}_{M,i}} (  \tilde{\mathcal{R}}_{M,m,i} \times 
\tilde{\mathcal{T}}_{\ell}^M ) \cup \bigcup_{(m,\ell) \in \overline{Z}_{M,i}} (  \tilde{\mathcal{R}}_{M,m,i} \times 
\tilde{\mathcal{T}}_{\ell}^M ).
\end{equation*}
\end{defn}
\begin{lem}[Lemma 3.9 in \cite{Ki2019}]\label{lemma3.13}
Let $i=1,2$. For fixed $m \in \N \times \Z$, the number of $k \in \Z^2$ such that 
$(m, k) \in \widehat{Z}_{M,i}$ is finitely many. On the other hand, for 
fixed $k \in \Z^2$, the number of $m \in \N \times \Z$ such that 
$(m, k) \in \widehat{Z}_{M,i}$ is finitely many. 
Furthermore, the claim holds true whether we replace $\widehat{Z}_{M,i}$ by $\overline{Z}_{M,i}$ in 
the above statements.
\end{lem}
%%%%%%%%%%%%%%%%%%%%%%%%%%%%%%%%%%%%%
%%%%%%%%%%%%%%%%%%%%%%%%%%%%%%%%%%%%%
%%%%%%%%%%%%%%%%%%%%%%%%%%%%%%%%%%%%%
%%%%%%%%%%%%%%%%%%%%%%%%%%%%%%%%%%%%%
%%%%%%%%%%%%%%%%%%%%%%%%%%%%%%%%%%%%%
%%%%%%%%%%%%%%%%%%%%%%%%%%%%%%%%%%%%%
\begin{prop}\label{prop3.14}
In addition to the hypothesis of Proposition \ref{prop3.9}, assume that 
$| \sin \angle \left( (\xi_1, \eta_1), (\xi_2, \eta_2) \right)| \gtrsim 1$ and 
$(\xi_1, \eta_1) \times (\xi_2, \eta_2) \in (\mathcal{K}_1 \cup \mathcal{K}_2 ) \times \mathcal{K}_0$. Then we have
\begin{equation}
\begin{split}
& \left|\int_{*}{  h_{{N_0, L_0}}(\tau, \xi, \boldsymbol{\eta} ) 
f_{N_1, L_1} (\tau_1, \xi_1, \boldsymbol{\eta_1} )  
g_{N_2, L_2} (\tau_2, \xi_2, \boldsymbol{\eta_2} ) 
}
d\sigma_1 d\sigma_2 \right| \\
& \qquad  \lesssim  A^{-\frac{d-3}{2}} 
N_1^{\frac{d-6}{2}}    (L_0 L_1 L_2)^{\frac{1}{2}} 
\|f_{N_1, L_1} \|_{L^2} \| g_{N_2, L_2} \|_{L^2} 
\|h_{{N_0, L_0}}\|_{L^2},\label{est01-prop3.14}
\end{split}
\end{equation}
where functions $f_{N_1, L_1}$, $g_{N_2, L_2}$, $h_{{N_0, L_0}}$ satisfy 
\eqref{assumption-fgh}.
\end{prop}
%%%%%%%%%%%%%%%%%%%%%%%%%%%%%%%%%%%%%
%%%%%%%%%%%%%%%%%%%%%%%%%%%%%%%%%%%%%
%%%%%%%%%%%%%%%%%%%%%%%%%%%%%%%%%%%%%
%%%%%%%%%%%%%%%%%%%%%%%%%%%%%%%%%%%%%
%%%%%%%%%%%%%%%%%%%%%%%%%%%%%%%%%%%%%
%%%%%%%%%%%%%%%%%%%%%%%%%%%%%%%%%%%%%
%%%%%%%%%%%%%%%%%%%%%%%%%%%%%%%%%%%%%
%%%%%%%%%%%%%%%%%%%%%%%%%%%%%%%%%%%%%
\begin{proof}
To avoid redundancy, we only treat the case $(\xi_1 , \eta_1) \times (\xi_2, \eta_2) \in \mathcal{K}_1 \times \mathcal{K}_0$. 
The case $(\xi_1 , \eta_1) \times (\xi_2, \eta_2) \in \mathcal{K}_2 \times \mathcal{K}_0$ can be dealt with in the similar way. Similarly to the proof of Lemma \ref{prop3.12}, 
by the following inclusion
\begin{equation*}
 \left( 
G_{N_1, L_1} \times G_{N_2, L_2} \right) \cup ( \tilde{\mathcal{K}}_1 \times 
\tilde{\mathcal{K}}_0  ) \subset \bigcup_{(m,\ell) \in \widehat{Z}_{M,1}} (  \tilde{\mathcal{R}}_{M,m,1} \times 
\tilde{\mathcal{T}}_{\ell}^M ) \cup \bigcup_{(m,\ell) \in \overline{Z}_{M,1}} (  \tilde{\mathcal{R}}_{M,m,1} \times 
\tilde{\mathcal{T}}_{\ell}^M ),
\end{equation*}
we get
\begin{align*}
& \textnormal{(LHS) of \eqref{est01-prop3.14}} \\ 
\leq & \sum_{\widehat{M} \leq M \leq A} \sum_{(m, \ell) \in \widehat{Z}_{M,1}} 
\left|\int_{*}{  h_{N_0,L_0}(\tau, \xi, \boldsymbol{\eta}) 
f_{N_1,L_1}|_{\tilde{\mathcal{R}}_{M,m,1}}(\tau_1, \xi_1, \boldsymbol{\eta_1})  g_{N_2, L_2}|_{\tilde{\mathcal{T}}_{\ell}^M}(\tau_2, \xi_2, \boldsymbol{\eta_2}) 
}
d\sigma_1 d\sigma_2 \right|\\
& \quad + \sum_{(m,\ell) \in \overline{Z}_{A,1}} 
\left|\int_{*}{  h_{N_0,L_0}(\tau, \xi, \boldsymbol{\eta}) 
f_{N_1,L_1}|_{\tilde{\mathcal{R}}_{A,m,1}}(\tau_1, \xi_1, \boldsymbol{\eta_1})
  g_{N_2, L_2}|_{\tilde{\mathcal{T}}_{\ell}^{A}}(\tau_2, \xi_2, \boldsymbol{\eta_2}) 
}
d\sigma_1 d\sigma_2 \right|\\
& =: \sum_{\widehat{M} \leq M \leq A} \sum_{(m, \ell) \in \widehat{Z}_{M,1}}  I_1 + 
\sum_{(m,\ell) \in \overline{Z}_{A,1}}   I_2.
\end{align*}
The former term is estimated by Proposition \ref{prop3.10} and Lemma \ref{lemma3.13} as
\begin{align*}
& \sum_{(m, \ell) \in \widehat{Z}_{M,1}} I_1 \\ \lesssim & 
\sum_{(m, \ell) \in \widehat{Z}_{M,1}}
A^{-\frac{d-2}{2}} M^{\frac{1}{2}}
N_1^{\frac{d-6}{2}}    (L_0 L_1 L_2)^{\frac{1}{2}}  \|f_{N_1,L_1}|_{\tilde{\mathcal{R}}_{A,m,1}} \|_{L^2} 
\| g_{N_2, L_2}|_{\tilde{\mathcal{T}}_{\ell}^M} \|_{L^2} 
\|h_{N_0,L_0} \|_{L^2}\\
\lesssim & \ 
A^{-\frac{d-2}{2}} M^{\frac{1}{2}}
N_1^{\frac{d-6}{2}}    (L_0 L_1 L_2)^{\frac{1}{2}} \|f_{N_1,L_1} \|_{L^2} 
\| g_{N_2, L_2} \|_{L^2} 
\|h_{N_0,L_0} \|_{L^2},
\end{align*}
which yields
\begin{equation*}
\sum_{\widehat{M} \leq M \leq A} \sum_{(m, \ell) \in \widehat{Z}_{M,1}}   I_1 
 \lesssim  
 A^{-\frac{d-3}{2}} 
N_1^{\frac{d-6}{2}}  (L_0 L_1 L_2)^{\frac{1}{2}} \|f_{N_1,L_1} \|_{L^2} 
\| g_{N_2, L_2} \|_{L^2} 
\|h_{N_0,L_0} \|_{L^2}.
\end{equation*}
We deal with the latter term in the same manner as that for the proof of Proposition \ref{prop3.12}. 
The assumption $|\boldsymbol{\eta_2'}| \lesssim A^{-1} N_1$ means that support of 
$g_{N_2, L_2}|_{\tilde{\mathcal{T}}_{\ell_2}^{A}}$ is contained in a regular cube which side length is  
comparable to $A^{-1} N_1$. Thus, by the almost orthogonality and Proposition \ref{prop3.7}, we obtain
\begin{equation*}
I_2 \lesssim A^{-\frac{d-3}{2}} 
N_1^{\frac{d-6}{2}}  (L_0 L_1 L_2)^{\frac{1}{2}}  \|f_{N_1,L_1}|_{\tilde{\mathcal{R}}_{A,m,1}} \|_{L^2} 
\| g_{N_2, L_2}|_{\tilde{\mathcal{T}}_{\ell}^{A}} \|_{L^2} \|h_{N_0,L_0} \|_{L^2}.
\end{equation*}
Hence, it follows from Lemma 
\ref{lemma3.13} that
\begin{align*}
\sum_{(m,\ell) \in \overline{Z}_{A,1}}  I_2 & \lesssim 
A^{-\frac{d-3}{2}} 
N_1^{\frac{d-6}{2}}  (L_0 L_1 L_2)^{\frac{1}{2}} \sum_{(m,\ell) \in \overline{Z}_{A,1}} 
\|f_{N_1,L_1}|_{\tilde{\mathcal{R}}_{A,m,1}} \|_{L^2} 
\| g_{N_2, L_2}|_{\tilde{\mathcal{T}}_{\ell}^{A}} \|_{L^2} \|h_{N_0,L_0} \|_{L^2}\\
& \lesssim A^{-\frac{d-3}{2}} 
N_1^{\frac{d-6}{2}}  (L_0 L_1 L_2)^{\frac{1}{2}} \|f_{N_1,L_1} \|_{L^2} 
\| g_{N_2, L_2} \|_{L^2} 
\|h_{N_0,L_0} \|_{L^2}.
\end{align*}
This completes the proof.
\end{proof}
%%%%%%%%%%%%%%%%%%%%%%%%%%%%%%%%%%%%%
%%%%%%%%%%%%%%%%%%%%%%%%%%%%%%%%%%%%%
%%%%%%%%%%%%%%%%%%%%%%%%%%%%%%%%%%%%%
%%%%%%%%%%%%%%%%%%%%%%%%%%%%%%%%%%%%%
%%%%%%%%%%%%%%%%%%%%%%%%%%%%%%%%%%%%%
%%%%%%%%%%%%%%%%%%%%%%%%%%%%%%%%%%%%%
%%%%%%%%%%%%%%%%%%%%%%%%%%%%%%%%%%%%%
Next, we consider the case $| \sin \angle \left( (\xi_1, \eta_1), (\xi_2, \eta_2) \right)| \ll 1$. 
Similarly to the case $| \sin \angle \left( (\xi_1, \eta_1), (\xi_2, \eta_2) \right)| \gtrsim 1$, 
we follow the proof for the 2D Zakharov-Kuznetsov equation. 
\begin{defn}\label{definition7}
Let $M$ be dyadic. Define
\begin{align*}
& \Theta_k^M = \left[\frac{\pi}{M} \, (k-2), \ \frac{\pi}{M} \, (k+2) \right] 
\cup \left[-\pi + \frac{\pi}{M} \, (k-2), \ - \pi +\frac{\pi}{M} \, (k+2) \right],\\
& {\mathfrak{D}}_k^M  = \{ ( r \cos \theta, r \sin \theta) \in  \R^2 
\, | \, r\geq 0 , \ \theta \in \Theta_k^M  \},\\
& \tilde{\mathfrak{D}}_k^M = \R \times  {\mathfrak{D}}_k^M \times \R^{d-2}.
\end{align*}
Let $\mathcal{I}$, $(\mathcal{I})^c \subset \R^2 \times \R^2$ be defined as follows:
\begin{align*}
& \mathcal{I}  = \left( {\mathfrak{D}}_{0}^{2^{11}} \times {\mathfrak{D}}_{0}^{2^{11}} \right) 
\cup \left( {\mathfrak{D}}_{2^{10}}^{2^{11}} \times {\mathfrak{D}}_{2^{10}}^{2^{11}} \right) , & 
& \tilde{\mathcal{I}}  = \left( \tilde{{\mathfrak{D}}}_{0}^{2^{11}} \times 
\tilde{{\mathfrak{D}}}_{0}^{2^{11}} \right) 
\cup \left( \tilde{{\mathfrak{D}}}_{2^{10}}^{2^{11}} \times \tilde{{\mathfrak{D}}}_{2^{10}}^{2^{11}} \right), \\
& (\mathcal{I})^c  = \left( \R^2 \times \R^2 \right) \setminus  \mathcal{I}, & 
& (\tilde{\mathcal{I}})^c  = \left( \R^{d+1} \times \R^{d+1} \right) \setminus  \tilde{\mathcal{I}}.
\end{align*}
\end{defn}
Note that
\begin{align*}
 {\mathfrak{D}}_{0}^{2^{11}} & =\left\{ ( |(\xi,\eta)| \cos \theta, |(\xi, \eta)| \sin \theta) \in  \R^2 
\, | \, \textnormal{min} \left( |\theta|, |\theta - \pi| \right)  \leq 2^{-10}\pi \right\},\\
 {\mathfrak{D}}_{2^{10}}^{2^{11}} & =\left\{ ( |(\xi,\eta)| \cos \theta, |(\xi, \eta)| \sin \theta) \in  \R^2 
\, | \, \textnormal{min} \left( \left| \theta-\frac{\pi}{2}\right|, 
\left| \theta + \frac{\pi}{2} \right| \right) \leq 2^{-10}  \pi
 \right\}.
\end{align*}
We begin with the case $(\xi_1 , \eta_1) \times (\xi_2, \eta_2) \in  (\mathcal{I})^c$. 
Note that 
$\max (|\xi_1+\eta_1|, |\xi_2+\eta_2|) \geq 2^{-5}N_1$ in \textit{Assumption} \ref{assumption1prime} allows us to assume
\begin{equation}
(\xi_1,\eta_1) \times (\xi_2,\eta_2) \notin \left( {\mathfrak{D}}_{2^9 \times 3}^{2^{11}} \times {\mathfrak{D}}_{2^9 \times 3}^{2^{11}} \right).\label{ang-ass-def7}
\end{equation}
Remark that
\begin{equation*}
 {\mathfrak{D}}_{2^{9} \times 3}^{2^{11}} =\left\{ ( |(\xi,\eta)| \cos \theta, |(\xi, \eta)| \sin \theta) \in  \R^2 
\, | \, \textnormal{min} \left( \left| \theta-\frac{3 \pi}{4}\right|, 
\left| \theta + \frac{\pi}{4} \right| \right) \leq 2^{-10} \pi 
 \right\}.
\end{equation*}
\begin{prop}\label{prop3.15}
Assume $|\boldsymbol{\eta_j'}| \lesssim A^{-1} N_1$, \eqref{ang-ass-def7} and \eqref{assumption-fgh}. 
Let $A$, $M$ be dyadic which satisfy $1 \ll M \leq A$ and $(k_1,k_2)$ satisfies 
${\mathfrak{D}}_{k_1}^M \times {\mathfrak{D}}_{k_2}^M \subset (\mathcal{I})^c$, 
$16 \leq |k_1 - k_2| \leq 32$. 
Then we get
\begin{equation}
\begin{split}
& \left|\int_{*}{  h_{{N_0, L_0}}(\tau, \xi, \boldsymbol{\eta} ) 
f_{N_1, L_1}|_{\tilde{\mathfrak{D}}_{k_1}^M} (\tau_1, \xi_1, \boldsymbol{\eta_1} )
g_{N_2, L_2}|_{\tilde{\mathfrak{D}}_{k_2}^M} (\tau_2, \xi_2, \boldsymbol{\eta_2} )
}
d\sigma_1 d\sigma_2 \right| \\
& \qquad \lesssim  A^{-\frac{d-2}{2}} M^{\frac{1}{2}}
N_1^{\frac{d-6}{2}}    (L_0 L_1 L_2)^{\frac{1}{2}} 
\|f_{N_1, L_1}|_{\tilde{\mathfrak{D}}_{k_1}^M} \|_{L^2} \| g_{N_2, L_2}|_{\tilde{\mathfrak{D}}_{k_2}^M} \|_{L^2} 
\|h_{{N_0, L_0}}\|_{L^2},\label{est01-prop3.15}
\end{split}
\end{equation}
where $d \sigma_j = d\tau_j d \xi_j d \boldsymbol{\eta}_j$ and $*$ denotes $(\tau, \xi, \boldsymbol{\eta}) = (\tau_1 + \tau_2, \xi_1+ \xi_2, \boldsymbol{\eta_1} + \boldsymbol{\eta_2}).$
\end{prop}
\begin{proof}
It suffices to show
\begin{equation}
\begin{split}
& \left|\int_{\hat{*}}{  h_{{N_0, L_0}}(\tau, \xi, \boldsymbol{\eta} ) 
f_{N_1, L_1}|_{\tilde{\mathfrak{D}}_{k_1}^M} (\tau_1, \xi_1, \boldsymbol{\eta_1} )
g_{N_2, L_2}|_{\tilde{\mathfrak{D}}_{k_2}^M} (\tau_2, \xi_2, \boldsymbol{\eta_2} )
}
d\hat{\sigma}_1 d\hat{\sigma}_2 \right| \\
&  \lesssim M^{\frac{1}{2}} N_{1}^{-2}   (L_0 L_1 L_2)^{\frac{1}{2}} 
\|f_{N_1, L_1}|_{\tilde{\mathfrak{D}}_{k_1}^M} (\boldsymbol{\eta_1'}) \|_{L^2_{\tau \xi \eta}} 
\|g_{N_2, L_2}|_{\tilde{\mathfrak{D}}_{k_2}^M} (\boldsymbol{\eta_2'})\|_{L^2_{\tau \xi \eta}} 
\|h_{{N_0, L_0}}(\boldsymbol{\eta'}) \|_{L^2_{\tau \xi \eta}} ,\label{est02-prop3.15}
\end{split}
\end{equation}
where $d \hat{\sigma}_j = d\tau_j d \xi_j d \eta_j$ and 
$\hat{*}$ denotes $(\tau,\xi, \eta) = (\tau_1 + \tau_2,\xi_1+\xi_2, \eta_1 + \eta_2).$ 
As we saw in the proof of Proposition \ref{prop3.10}, since $M \leq A$ and $|\boldsymbol{\eta_j'}| \lesssim A^{-1} N_1$, we can show \eqref{est02-prop3.15} by following the proof of Proposition 3.14 in 
\cite{Ki2019}. We omit the proof.
\end{proof}
%%%%%%%%%%%%%%%%%%%%%%%%%%%%%%%%%%%%%
%%%%%%%%%%%%%%%%%%%%%%%%%%%%%%%%%%%%%
%%%%%%%%%%%%%%%%%%%%%%%%%%%%%%%%%%%%%
%%%%%%%%%%%%%%%%%%%%%%%%%%%%%%%%%%%%%
%%%%%%%%%%%%%%%%%%%%%%%%%%%%%%%%%%%%%
%%%%%%%%%%%%%%%%%%%%%%%%%%%%%%%%%%%%%
%%%%%%%%%%%%%%%%%%%%%%%%%%%%%%%%%%%%%
%%%%%%%%%%%%%%%%%%%%%%%%%%%%%%%%%%%%%
\begin{prop}\label{prop3.16}
In addition to the hypothesis of Proposition \ref{prop3.9}, assume that 
$| \sin \angle \left( (\xi_1, \eta_1), (\xi_2, \eta_2) \right)| \ll 1$, \eqref{ang-ass-def7} and 
$(\xi_1, \eta_1) \times (\xi_2, \eta_2) \in (\mathcal{I})^c$. Then we have
\begin{equation}
\begin{split}
& \left|\int_{*}{  h_{{N_0, L_0}}(\tau, \xi, \boldsymbol{\eta} ) 
f_{N_1, L_1} (\tau_1, \xi_1, \boldsymbol{\eta_1} )  
g_{N_2, L_2} (\tau_2, \xi_2, \boldsymbol{\eta_2} ) 
}
d\sigma_1 d\sigma_2 \right| \\
& \qquad  \lesssim  A^{-\frac{d-3}{2}} 
N_1^{\frac{d-6}{2}}    (L_0 L_1 L_2)^{\frac{1}{2}} 
\|f_{N_1, L_1} \|_{L^2} \| g_{N_2, L_2} \|_{L^2} 
\|h_{{N_0, L_0}}\|_{L^2},\label{est01-prop3.16}
\end{split}
\end{equation}
where functions $f_{N_1, L_1}$, $g_{N_2, L_2}$, $h_{{N_0, L_0}}$ satisfy 
\eqref{assumption-fgh}.
\end{prop}
%%%%%%%%%%%%%%%%%%%%%%%%%%%%%%%%%%%%%
%%%%%%%%%%%%%%%%%%%%%%%%%%%%%%%%%%%%%
%%%%%%%%%%%%%%%%%%%%%%%%%%%%%%%%%%%%%
\begin{proof}
We define that
\begin{equation*}
J_{M}^{(\mathcal{I})^c} = \{ (k_1, k_2) \, | \, 0 \leq k_1,k_2 \leq M -1, \ \left( {\mathfrak{D}}_{k_1}^M \times {\mathfrak{D}}_{k_2}^M \right) \subset (\mathcal{I})^c \cap \left( {\mathfrak{D}}_{2^9 \times 3}^{2^{11}} \times {\mathfrak{D}}_{2^9 \times 3}^{2^{11}} \right)^c .\}
\end{equation*}
We perform the Whitney type decomposition as 
\begin{equation*}
(\mathcal{I})^c \cap \left( {\mathfrak{D}}_{2^9 \times 3}^{2^{11}} \times {\mathfrak{D}}_{2^9 \times 3}^{2^{11}} \right)^c = \bigcup_{64 \leq M \leq A} \ 
\bigcup_{\tiny{\substack{(k_1,k_2) \in J_{M}^{(\mathcal{I})^c}\\ 16 \leq |k_1 - k_2|\leq 32}}} 
{\mathfrak{D}}_{k_1}^M \cross {\mathfrak{D}}_{k_2}^M 
\cup \bigcup_{\tiny{\substack{(k_1,k_2) \in J_{A}^{(\mathcal{I})^c}\\|k_1 - k_2|\leq 16}}} 
{\mathfrak{D}}_{k_1}^{A} \cross {\mathfrak{D}}_{k_2}^{A}.
\end{equation*}
Note that $| \sin \angle \left( (\xi_1, \eta_1), (\xi_2, \eta_2) \right)| \ll 1$ implies 
$M \gg 1$. We observe
\begin{align*}
& \textnormal{(LHS) of \eqref{est01-prop3.16}} \\ 
\leq & \sum_{1 \ll M \leq A} \sum_{\tiny{\substack{(k_1,k_2) \in J_{M}^{(\mathcal{I})^c}\\ 16 \leq |k_1 - k_2|\leq 32}}} 
\left|\int_{*}{  h_{N_0,L_0}(\tau, \xi, \boldsymbol{\eta}) 
f_{N_1, L_1}|_{\tilde{\mathfrak{D}}_{k_1}^M} (\tau_1, \xi_1, \boldsymbol{\eta_1} )
g_{N_2, L_2}|_{\tilde{\mathfrak{D}}_{k_2}^M} (\tau_2, \xi_2, \boldsymbol{\eta_2} )
}
d\sigma_1 d\sigma_2 \right|\\
& \quad + \sum_{\tiny{\substack{(k_1,k_2) \in J_{A}^{(\mathcal{I})^c}\\|k_1 - k_2|\leq 16}}} 
\left|\int_{*}{  h_{N_0,L_0}(\tau, \xi, \boldsymbol{\eta}) 
f_{N_1, L_1}|_{\tilde{\mathfrak{D}}_{k_1}^A} (\tau_1, \xi_1, \boldsymbol{\eta_1} )
g_{N_2, L_2}|_{\tilde{\mathfrak{D}}_{k_2}^A} (\tau_2, \xi_2, \boldsymbol{\eta_2} )
}
d\sigma_1 d\sigma_2 \right|\\
& =: \sum_{1 \ll M \leq A} \sum_{\tiny{\substack{(k_1,k_2) \in J_{M}^{(\mathcal{I})^c}\\ 16 \leq |k_1 - k_2|\leq 32}}}   I_1 + \sum_{\tiny{\substack{(k_1,k_2) \in J_{A}^{(\mathcal{I})^c}\\|k_1 - k_2|\leq 16}}}   I_2.
\end{align*}
The former term is dealt with by Proposition \ref{prop3.15} as follows.
\begin{align*}
& \sum_{1 \ll M \leq A} \sum_{\tiny{\substack{(k_1,k_2) \in J_{M}^{(\mathcal{I})^c}\\ 16 \leq |k_1 - k_2|\leq 32}}}   I_1\\ & \lesssim \sum_{1 \ll M \leq A}  A^{-\frac{d-2}{2}} M^{\frac{1}{2}}
N_1^{\frac{d-6}{2}}    (L_0 L_1 L_2)^{\frac{1}{2}} \|f_{N_1,L_1} \|_{L^2} 
\| g_{N_2, L_2} \|_{L^2} 
\|h_{N_0,L_0} \|_{L^2}\\
& \lesssim A^{-\frac{d-3}{2}} 
N_1^{\frac{d-6}{2}}    (L_0 L_1 L_2)^{\frac{1}{2}} \|f_{N_1,L_1} \|_{L^2} 
\| g_{N_2, L_2} \|_{L^2} 
\|h_{N_0,L_0} \|_{L^2}.
\end{align*}
For the latter term, we only consider the case $|(\xi,\eta)| \gg A^{-1} N_1$. 
The case $|(\xi,\eta)| \lesssim A^{-1} N_1$ can be treated by Proposition \ref{prop3.7}. 
By Lemma 3.12 in \cite{Ki2019} and $|\boldsymbol{\eta_j'}| \lesssim A^{-1} N_1$, we easily observe that $|(\xi,\eta)| \gg A^{-1} N_1$ gives 
$|\Phi(\xi_1,\boldsymbol{\eta_1} , \xi_2,\boldsymbol{\eta_2})| \gtrsim A^{-1} N_1^3$. 
Thus, it suffices to show the following bilinear estimates.
\begin{align*}
& \left\| {\mathbf{1}}_{G_{N_0, L_0}} \int  f_{N_1, L_1}|_{\tilde{\mathfrak{D}}_{k_1}^A}(\tau_1, \xi_1, \boldsymbol{\eta_1} ) 
g_{N_2, L_2}|_{\tilde{\mathfrak{D}}_{k_2}^A} (\tau- \tau_1, \xi-\xi_1, \boldsymbol{\eta}- \boldsymbol{\eta_1} ) d\sigma_1 \right\|_{L^2} \notag \\ 
& \qquad \qquad \qquad \qquad \qquad 
\lesssim A^{-\frac{d-2}{2}} N_1^{\frac{d-3}{2}} (L_1 L_2)^{\frac{1}{2}} \| f_{N_1, L_1}|_{\tilde{\mathfrak{D}}_{k_1}^A}\|_{L^2} 
\|g_{N_2, L_2}|_{\tilde{\mathfrak{D}}_{k_2}^A} \|_{L^2},\\
& \left\| {\mathbf{1}}_{G_{N_1, L_1} \cap \tilde{\mathfrak{D}}_{k_1}^A} \int g_{N_2, L_2}|_{\tilde{\mathfrak{D}}_{k_2}^A}
(\tau_2,\xi_2,\boldsymbol{\eta_2}) 
h_{N_0,L_0} (\tau_1+ \tau_2, \xi_1+\xi_2, \boldsymbol{\eta_1}+ \boldsymbol{\eta_2}) d\sigma_2 
\right\|_{L^2}
\notag \\ 
& \qquad \qquad \qquad \qquad \qquad 
\lesssim A^{-\frac{d-1}{2}} N_1^{\frac{d-3}{2}} (L_0 L_2)^{\frac{1}{2}} 
\|g_{N_2, L_2}|_{\tilde{\mathfrak{D}}_{k_2}^A} \|_{L^2} 
\|h_{N_0,L_0}\|_{L^2},\\
& \left\| {\mathbf{1}}_{G_{N_2,L_2} \cap \tilde{\mathfrak{D}}_{k_2}^A} \int 
h_{N_0,L_0} (\tau_1+ \tau_2, \xi_1+\xi_2, \boldsymbol{\eta_1}+ \boldsymbol{\eta_2} )   
f_{N_1, L_1}|_{\tilde{\mathfrak{D}}_{k_1}^A}(\tau_1, \xi_1, \boldsymbol{\eta_1} )  d \sigma_1 
\right\|_{L^2} \notag \\ 
& \qquad \qquad \qquad \qquad \qquad 
\lesssim A^{-\frac{d-1}{2}} N_1^{\frac{d-3}{2}} (L_0 L_1)^{\frac{1}{2}} \|h_{N_0,L_0} \|_{L^2}
\|f_{N_1, L_1}|_{\tilde{\mathfrak{D}}_{k_1}^A} \|_{L^2},
\end{align*}
that are verified by showing
\begin{align*}
& \left\| {\mathbf{1}}_{G_{N_0, L_0}} \int  f_{N_1, L_1}|_{\tilde{\mathfrak{D}}_{k_1}^A}(\tau_1, \xi_1, \boldsymbol{\eta_1} ) 
g_{N_2, L_2}|_{\tilde{\mathfrak{D}}_{k_2}^A} (\tau- \tau_1, \xi-\xi_1, \boldsymbol{\eta}- \boldsymbol{\eta_1} ) d\hat{\sigma}_1 \right\|_{L^2} \notag \\ 
& \qquad \qquad \qquad \qquad \qquad 
\lesssim N_1^{-\frac{1}{2}} (L_1 L_2)^{\frac{1}{2}} \| f_{N_1, L_1}|_{\tilde{\mathfrak{D}}_{k_1}^A} 
(\boldsymbol{\eta_1'})\|_{L^2} 
\|g_{N_2, L_2}|_{\tilde{\mathfrak{D}}_{k_2}^A} (\boldsymbol{\eta_2'})\|_{L^2},\\
& \left\| {\mathbf{1}}_{G_{N_1, L_1} \cap \tilde{\mathfrak{D}}_{k_1}^A} \int g_{N_2, L_2}|_{\tilde{\mathfrak{D}}_{k_2}^A}
(\tau_2,\xi_2,\boldsymbol{\eta_2}) 
h_{N_0,L_0} (\tau_1+ \tau_2, \xi_1+\xi_2, \boldsymbol{\eta_1}+ \boldsymbol{\eta_2}) d\hat{\sigma}_2 
\right\|_{L^2}
\notag \\ 
& \qquad \qquad \qquad \qquad \qquad 
\lesssim ( A N_1)^{-\frac{1}{2}} (L_0 L_2)^{\frac{1}{2}} 
\|g_{N_2, L_2}|_{\tilde{\mathfrak{D}}_{k_2}^A} (\boldsymbol{\eta_2'})\|_{L^2} 
\|h_{N_0,L_0} (\boldsymbol{\eta'})\|_{L^2},\\
& \left\| {\mathbf{1}}_{G_{N_2,L_2} \cap \tilde{\mathfrak{D}}_{k_2}^A} \int 
h_{N_0,L_0} (\tau_1+ \tau_2, \xi_1+\xi_2, \boldsymbol{\eta_1}+ \boldsymbol{\eta_2} )   
f_{N_1, L_1}|_{\tilde{\mathfrak{D}}_{k_1}^A}(\tau_1, \xi_1, \boldsymbol{\eta_1} )  d \hat{\sigma}_1 
\right\|_{L^2} \notag \\ 
& \qquad \qquad \qquad \qquad \qquad 
\lesssim ( A N_1)^{-\frac{1}{2}} (L_0 L_1)^{\frac{1}{2}} \|h_{N_0,L_0} (\boldsymbol{\eta'})\|_{L^2}
\|f_{N_1, L_1}|_{\tilde{\mathfrak{D}}_{k_1}^A}(\boldsymbol{\eta_1'}) \|_{L^2},
\end{align*}
respectively. These estimates are established in the same manner as for Proposition 3.13 in \cite{Ki2019}. 
We omit the details.
\end{proof}
%%%%%%%%%%%%%%%%%%%%%%%%%%%%%%%%%%%%%
%%%%%%%%%%%%%%%%%%%%%%%%%%%%%%%%%%%%%
%%%%%%%%%%%%%%%%%%%%%%%%%%%%%%%%%%%%%
%%%%%%%%%%%%%%%%%%%%%%%%%%%%%%%%%%%%%
%%%%%%%%%%%%%%%%%%%%%%%%%%%%%%%%%%%%%
Next we treat the case $(\xi_1 , \eta_1) \times (\xi_2, \eta_2) \in \mathcal{I}$. 
By symmetry, we may assume $(\xi_1 , \eta_1) \times (\xi_2, \eta_2) \in  {\mathfrak{D}}_{0}^{2^{11}} \times {\mathfrak{D}}_{0}^{2^{11}}$ and show the following.
\begin{prop}\label{prop3.17}
Under the hypothesis of Proposition \ref{prop3.9}, we have
\begin{equation}
\begin{split}
& \left|\int_{*}{  h_{{N_0, L_0}}(\tau, \xi, \boldsymbol{\eta} ) 
f_{N_1, L_1}|_{\tilde{\mathfrak{D}}_{0}^{2^{11}}} (\tau_1, \xi_1, \boldsymbol{\eta_1} )  
g_{N_2, L_2}|_{\tilde{\mathfrak{D}}_{0}^{2^{11}}} (\tau_2, \xi_2, \boldsymbol{\eta_2} ) 
}
d\sigma_1 d\sigma_2 \right| \\
& \qquad  \lesssim  N_0^{\frac{d-4}{2}+2 \e} N_{1}^{-1-\frac{3}{2}\e}    (L_0 L_1 L_2)^{\frac{1}{2}} 
\|f_{N_1, L_1}|_{\tilde{\mathfrak{D}}_{0}^{2^{11}}} \|_{L^2} 
\| g_{N_2, L_2}|_{\tilde{\mathfrak{D}}_{0}^{2^{11}}} \|_{L^2} 
\|h_{{N_0, L_0}}\|_{L^2},\label{est01-prop3.17}
\end{split}
\end{equation}
where functions $f_{N_1, L_1}$, $g_{N_2, L_2}$, $h_{{N_0, L_0}}$ satisfy 
\eqref{assumption-fgh}.
\end{prop}
We note that the proof is almost the same as that for Proposition 3.18 in \cite{Ki2019}. 
Therefore, we only give a sketch of the proof here. 
%%%%%%%%%%%%%%%%%%%%%%%%%%%%%%%%%%%%%
%%%%%%%%%%%%%%%%%%%%%%%%%%%%%%%%%%%%%
%%%%%%%%%%%%%%%%%%%%%%%%%%%%%%%%%%%%%
%%%%%%%%%%%%%%%%%%%%%%%%%%%%%%%%%%%%%
%%%%%%%%%%%%%%%%%%%%%%%%%%%%%%%%%%%%%
%%%%%%%%%%%%%%%%%%%%%%%%%%%%%%%%%%%%%
\begin{defn}
Let $M \gg 1$ and $K$ be dyadic which satisfy $2^{10} \leq K \leq 2^{-10}M$. We define that
\begin{align*}
\mathfrak{K}_M^K & = \left\{ k \in \N \ | \ \frac{M}{K} \leq k \leq 2 \frac{M}{K}, \quad M- 2 \frac{M}{K} \leq k \leq M-  \frac{M}{K} \right\},\\
\mathfrak{K}_M & = \left\{ k \in \N \ | \ 0 \leq k \leq 2^{10}, \quad M- 2^{10}  \leq k \leq M- 1 \right\}.
\end{align*}
\end{defn}
%%%%%%%%%%%%%%%%%%%%%%%%%%%%%%%%%%%%%
%%%%%%%%%%%%%%%%%%%%%%%%%%%%%%%%%%%%%
%%%%%%%%%%%%%%%%%%%%%%%%%%%%%%%%%%%%%
The following proposition corresponds to Proposition 3.19 in \cite{Ki2019}. 
\begin{prop}\label{prop3.18}
Suppose that $|\boldsymbol{\eta_j'}| \lesssim A^{-1} N_1$ and functions 
$f_{N_1, L_1}$, $g_{N_2, L_2}$, $h_{{N_0, L_0}}$ satisfy \eqref{assumption-fgh}. 
Let $M$ be dyadic such that $1 \ll M \leq A$, $|k_1-k_2| \leq 32$ and 
\begin{equation*}
\left( {\mathfrak{D}}_{k_1}^M \times {\mathfrak{D}}_{k_2}^M \right) \subset \mathcal{I}.
\end{equation*}
Then we have
\begin{align}
& \left\| {\mathbf{1}}_{G_{N_1, L_1} \cap \tilde{\mathfrak{D}}_{k_1}^M} \int g_{N_2, L_2}|_{\tilde{\mathfrak{D}}_{k_2}^M}
(\tau_2,\xi_2,\boldsymbol{\eta_2}) 
h_{N_0,L_0} (\tau_1+ \tau_2, \xi_1+\xi_2, \boldsymbol{\eta_1}+ \boldsymbol{\eta_2}) d\sigma_2 
\right\|_{L^2}
\notag \\ 
& \qquad \qquad \qquad \qquad 
\lesssim A^{-\frac{d-2}{2}}M^{-\frac{1}{2}} N_1^{\frac{d-3}{2}} (L_0 L_2)^{\frac{1}{2}} 
\|g_{N_2, L_2}|_{\tilde{\mathfrak{D}}_{k_2}^M} \|_{L^2} 
\|h_{N_0,L_0}\|_{L^2},\label{bilinear-est01-prop3.18}\\
& \left\| {\mathbf{1}}_{G_{N_2,L_2} \cap \tilde{\mathfrak{D}}_{k_2}^M} \int 
h_{N_0,L_0} (\tau_1+ \tau_2, \xi_1+\xi_2, \boldsymbol{\eta_1}+ \boldsymbol{\eta_2} )   
f_{N_1, L_1}|_{\tilde{\mathfrak{D}}_{k_1}^M}(\tau_1, \xi_1, \boldsymbol{\eta_1} )  d \sigma_1 
\right\|_{L^2} \notag \\ 
& \qquad \qquad \qquad \qquad 
\lesssim A^{-\frac{d-2}{2}} M^{-\frac{1}{2}}N_1^{\frac{d-3}{2}} (L_0 L_1)^{\frac{1}{2}} \|h_{N_0,L_0} \|_{L^2}
\|f_{N_1, L_1}|_{\tilde{\mathfrak{D}}_{k_1}^M} \|_{L^2}.\label{bilinear-est02-prop3.18}
\end{align}
In addition to the above assumptions,\\
\textnormal{(1)} assume $N_0 \gg M^{-1} N_1$, then we have
\begin{align}
& \left\| {\mathbf{1}}_{G_{N_0, L_0}} \int  f_{N_1, L_1}|_{\tilde{\mathfrak{D}}_{k_1}^M}(\tau_1, \xi_1, \boldsymbol{\eta_1} ) 
g_{N_2, L_2}|_{\tilde{\mathfrak{D}}_{k_2}^M} (\tau- \tau_1, \xi-\xi_1, \boldsymbol{\eta}- \boldsymbol{\eta_1} ) d\sigma_1 \right\|_{L^2} \notag \\ 
& \qquad \qquad \qquad  
\lesssim A^{-\frac{d-2}{2}} (MN_0)^{-\frac{1}{2}}N_1^{\frac{d-2}{2}} (L_1 L_2)^{\frac{1}{2}} \| f_{N_1, L_1}|_{\tilde{\mathfrak{D}}_{k_1}^M}\|_{L^2} 
\|g_{N_2, L_2}|_{\tilde{\mathfrak{D}}_{k_2}^M} \|_{L^2}.\label{bilinear-est02a-prop3.18}
\end{align}
\textnormal{(2)} assume $k_1 \in \mathfrak{K}_M^K$, then we have
\begin{align}
& \left\| {\mathbf{1}}_{G_{N_0, L_0}} \int  f_{N_1, L_1}|_{\tilde{\mathfrak{D}}_{k_1}^M}(\tau_1, \xi_1, \boldsymbol{\eta_1} ) 
g_{N_2, L_2}|_{\tilde{\mathfrak{D}}_{k_2}^M} (\tau- \tau_1, \xi-\xi_1, \boldsymbol{\eta}- \boldsymbol{\eta_1} ) d\sigma_1 \right\|_{L^2} \notag \\ 
& \qquad \qquad \qquad  
\lesssim A^{-\frac{d-2}{2}} K^{\frac{1}{4}}N_1^{\frac{d-3}{2}} (L_1 L_2)^{\frac{1}{2}} \| f_{N_1, L_1}|_{\tilde{\mathfrak{D}}_{k_1}^M}\|_{L^2} 
\|g_{N_2, L_2}|_{\tilde{\mathfrak{D}}_{k_2}^M} \|_{L^2}.\label{bilinear-est03-prop3.18}
\end{align}
\textnormal{(3)} assume $M \ll A$, $k_1 \in \mathfrak{K}_M$ and either $16 \leq |k_1-k_2| \leq 32$ or 
$|\xi| \geq M^{-3/2}N_1$, then we have
\begin{align}
& \left\| {\mathbf{1}}_{G_{N_0, L_0}} \int  f_{N_1, L_1}|_{\tilde{\mathfrak{D}}_{k_1}^M}(\tau_1, \xi_1, \boldsymbol{\eta_1} ) 
g_{N_2, L_2}|_{\tilde{\mathfrak{D}}_{k_2}^M} (\tau- \tau_1, \xi-\xi_1, \boldsymbol{\eta}- \boldsymbol{\eta_1} ) d\sigma_1 \right\|_{L^2} \notag \\ 
& \qquad \qquad \qquad  
\lesssim A^{-\frac{d-2}{2}} M^{\frac{1}{4}}N_1^{\frac{d-3}{2}} (L_1 L_2)^{\frac{1}{2}} \| f_{N_1, L_1}|_{\tilde{\mathfrak{D}}_{k_1}^M}\|_{L^2} 
\|g_{N_2, L_2}|_{\tilde{\mathfrak{D}}_{k_2}^M} \|_{L^2}.\label{bilinear-est04-prop3.18}
\end{align}
\end{prop}
%%%%%%%%%%%%%%%%%%%%%%%%%%%%%%%%%%%%%
%%%%%%%%%%%%%%%%%%%%%%%%%%%%%%%%%%%%%
%%%%%%%%%%%%%%%%%%%%%%%%%%%%%%%%%%%%%
%%%%%%%%%%%%%%%%%%%%%%%%%%%%%%%%%%%%%
%%%%%%%%%%%%%%%%%%%%%%%%%%%%%%%%%%%%%
%%%%%%%%%%%%%%%%%%%%%%%%%%%%%%%%%%%%%
%%%%%%%%%%%%%%%%%%%%%%%%%%%%%%%%%%%%%
%%%%%%%%%%%%%%%%%%%%%%%%%%%%%%%%%%%%%
%%%%%%%%%%%%%%%%%%%%%%%%%%%%%%%%%%%%%
\begin{proof}
\eqref{bilinear-est01-prop3.18} and \eqref{bilinear-est02-prop3.18} are given by 
\begin{align*}
& \left\| {\mathbf{1}}_{G_{N_1, L_1} \cap \tilde{\mathfrak{D}}_{k_1}^M} \int g_{N_2, L_2}|_{\tilde{\mathfrak{D}}_{k_2}^M}
(\tau_2,\xi_2,\boldsymbol{\eta_2}) 
h_{N_0,L_0} (\tau_1+ \tau_2, \xi_1+\xi_2, \boldsymbol{\eta_1}+ \boldsymbol{\eta_2}) d\hat{\sigma}_2 
\right\|_{L^2}
\notag \\ 
& \qquad \qquad \qquad \qquad \qquad 
\lesssim ( M N_1)^{-\frac{1}{2}} (L_0 L_2)^{\frac{1}{2}} 
\|g_{N_2, L_2}|_{\tilde{\mathfrak{D}}_{k_2}^M} (\boldsymbol{\eta_2'})\|_{L^2} 
\|h_{N_0,L_0} (\boldsymbol{\eta'})\|_{L^2},\\
& \left\| {\mathbf{1}}_{G_{N_2,L_2} \cap \tilde{\mathfrak{D}}_{k_2}^M} \int 
h_{N_0,L_0} (\tau_1+ \tau_2, \xi_1+\xi_2, \boldsymbol{\eta_1}+ \boldsymbol{\eta_2} )   
f_{N_1, L_1}|_{\tilde{\mathfrak{D}}_{k_1}^M}(\tau_1, \xi_1, \boldsymbol{\eta_1} )  d \hat{\sigma}_1 
\right\|_{L^2} \notag \\ 
& \qquad \qquad \qquad \qquad \qquad 
\lesssim ( M N_1)^{-\frac{1}{2}} (L_0 L_1)^{\frac{1}{2}} \|h_{N_0,L_0} (\boldsymbol{\eta'})\|_{L^2}
\|f_{N_1, L_1}|_{\tilde{\mathfrak{D}}_{k_1}^M}(\boldsymbol{\eta_1'}) \|_{L^2},
\end{align*}
respectively. 
These estimates are obtained in the same manner as for \eqref{bilinear02-prop3.10} and \eqref{bilinear03-prop3.10} in Proposition \ref{prop3.10}, respectively. We omit the proof. 

Next we consider \eqref{bilinear-est02a-prop3.18}. We will show
\begin{align}
& \left\| {\mathbf{1}}_{G_{N_0, L_0}} \int  f_{N_1, L_1}|_{\tilde{\mathfrak{D}}_{k_1}^M}(\tau_1, \xi_1, \boldsymbol{\eta_1} ) 
g_{N_2, L_2}|_{\tilde{\mathfrak{D}}_{k_2}^M} (\tau- \tau_1, \xi-\xi_1, \boldsymbol{\eta}- \boldsymbol{\eta_1} ) d\hat{\sigma}_1 \right\|_{L^2} \notag \\ 
& \qquad \qquad  
\lesssim  (MN_0)^{-\frac{1}{2}} (L_1 L_2)^{\frac{1}{2}} \| f_{N_1, L_1}|_{\tilde{\mathfrak{D}}_{k_1}^M} 
(\boldsymbol{\eta_1'})\|_{L^2} 
\|g_{N_2, L_2}|_{\tilde{\mathfrak{D}}_{k_2}^M} (\boldsymbol{\eta_2'})\|_{L^2}.\label{bilinear-est02b-prop3.18}
\end{align}
We write $(\xi_1, \eta_1)=r_1(\cos \theta_1, \sin \theta_1)$, 
$(\xi-\xi_1, \eta-\eta_1)= r_2(\cos \theta_2, \sin \theta_2)$. 
Similarly to the proof of \eqref{bilinear01-prop3.10}, it suffices to show
\begin{equation}
|\partial_{r_1}(\xi \xi_1 (\xi-\xi_1) + \eta \eta_1 (\eta- \eta_1))| \gtrsim N_0 N_1.\label{est02c-prop3.18}
\end{equation}
We may assume $|(\xi,\eta)| \geq N_0/2 \gg M^{-1} N_1$. 
By the assumption $|k_1-k_2| \leq 32$, we easily confirm that $|\eta| \leq 2|\xi| \sim N_0$ which implies 
\eqref{est02c-prop3.18}.

Lastly, we consider \eqref{bilinear-est03-prop3.18} and \eqref{bilinear-est04-prop3.18}. It suffices to show
\begin{align}
& \left\| {\mathbf{1}}_{G_{N_0, L_0}} \int  f_{N_1, L_1}|_{\tilde{\mathfrak{D}}_{k_1}^M}(\tau_1, \xi_1, \boldsymbol{\eta_1} ) 
g_{N_2, L_2}|_{\tilde{\mathfrak{D}}_{k_2}^M} (\tau- \tau_1, \xi-\xi_1, \boldsymbol{\eta}- \boldsymbol{\eta_1} ) d\hat{\sigma}_1 \right\|_{L^2} \notag \\ 
& \qquad \qquad \qquad \qquad 
\lesssim N_1^{-\frac{1}{2}} (L_1 L_2)^{\frac{1}{2}} \| f_{N_1, L_1}|_{\tilde{\mathfrak{D}}_{k_1}^M} 
(\boldsymbol{\eta_1'})\|_{L^2} 
\|g_{N_2, L_2}|_{\tilde{\mathfrak{D}}_{k_2}^M} (\boldsymbol{\eta_2'})\|_{L^2},\label{bilinear-est05-prop3.18}
\end{align}
for $k_1 \in \mathfrak{K}_M^K$, $|k_1-k_2| \leq 32$ and 
\begin{align}
& \left\| {\mathbf{1}}_{G_{N_0, L_0}} \int  f_{N_1, L_1}|_{\tilde{\mathfrak{D}}_{k_1}^M}(\tau_1, \xi_1, \boldsymbol{\eta_1} ) 
g_{N_2, L_2}|_{\tilde{\mathfrak{D}}_{k_2}^M} (\tau- \tau_1, \xi-\xi_1, \boldsymbol{\eta}- \boldsymbol{\eta_1} ) d\hat{\sigma}_1 \right\|_{L^2} \notag \\ 
& \qquad \qquad \qquad  
\lesssim M^{\frac{1}{4}}N_1^{-\frac12} (L_1 L_2)^{\frac{1}{2}} \| f_{N_1, L_1}|_{\tilde{\mathfrak{D}}_{k_1}^M} 
(\boldsymbol{\eta_1'})\|_{L^2} 
\|g_{N_2, L_2}|_{\tilde{\mathfrak{D}}_{k_2}^M} (\boldsymbol{\eta_2'})\|_{L^2},\label{bilinear-est06-prop3.18}
\end{align}
for $k_1 \in \mathfrak{K}_M$ and either $16 \leq |k_1-k_2| \leq 32$ or $|\xi| \geq M^{-3/2}N_1$. 
\eqref{bilinear-est05-prop3.18} and \eqref{bilinear-est06-prop3.18} are established in the same way as that for (3.67) and (3.68) in Proposition 3.19 in \cite{Ki2019}. 
Thus, here we only confirm that it holds
\begin{equation}
\| |\nabla_x|^{\frac{1}{2p}} |\nabla_y|^{\frac{1}{2p}} u_{N,L} \|_{L_t^p L_{xy}^q} \lesssim L^{\frac{1}{2}} 
\| u_{N,L}\|_{L^2_{xyt}}, \quad 
\textnormal{if} \ \ \frac{2}{p} + \frac{2}{q} = 1, \ p >2,\label{Strichartz01-prop3.18}
\end{equation}
where $\supp\widehat{u}_{N,L} \subset G_{N,L}$. 
Let $c \in \R$. \eqref{Strichartz01-prop3.18} is given by the Strichartz estimates.
\begin{equation}
\| |\nabla_x|^{\frac{1}{2p}} |\nabla_y|^{\frac{1}{2p}} 
e^{-t(\partial_x^3+\partial_y^3 -(\partial_x+\partial_y)c)} \varphi \|_{L_t^p L_{xy}^q} \lesssim 
\| \varphi\|_{L^2_{xy}},\label{Strichartz02-prop3.18}
\end{equation} 
if $2/p+2/q = 1$, $p>2$. We can establish \eqref{Strichartz02-prop3.18} by applying Theorem 3.1 in \cite{KPV91}. Employing \eqref{Strichartz01-prop3.18} with $p=q=4$, 
we can show \eqref{bilinear-est05-prop3.18} and \eqref{bilinear-est06-prop3.18} by the same argument as that for  (3.67) and (3.68) in Proposition 3.19 in \cite{Ki2019}, respectively.
\end{proof}
%%%%%%%%%%%%%%%%%%%%%%%%%%%%%%%%%%%%%
%%%%%%%%%%%%%%%%%%%%%%%%%%%%%%%%%%%%%
%%%%%%%%%%%%%%%%%%%%%%%%%%%%%%%%%%%%%
%%%%%%%%%%%%%%%%%%%%%%%%%%%%%%%%%%%%%
First we consider the case $k_1 \in \mathfrak{K}_M^K$.
\begin{prop}\label{prop3.19}
Let $M$ be dyadic such that $1 \ll M \leq A$. 
Assume that $|\boldsymbol{\eta_j'}| \lesssim A^{-1} N_1$, 
$N_0 \gg M^{-1} N_1$, $k_1 \in \mathfrak{K}_M^K$, $k_2$ satisfy
 $ |k_1-k_2| \leq 32$. Then we have
\begin{equation}
\begin{split}
& \left|\int_{*}{  h_{{N_0, L_0}}(\tau, \xi, \boldsymbol{\eta} ) 
f_{N_1, L_1}|_{\tilde{\mathfrak{D}}_{k_1}^M} (\tau_1, \xi_1, \boldsymbol{\eta_1} )  
g_{N_2, L_2}|_{\tilde{\mathfrak{D}}_{k_2}^M}  (\tau_2, \xi_2, \boldsymbol{\eta_2} ) 
}
d\sigma_1 d\sigma_2 \right| \\
&   \lesssim  A^{-\frac{d-2}{2}} N_0^{-\frac{1}{2}}
N_1^{\frac{d-5}{2}}    (L_0 L_1 L_2)^{\frac{1}{2}} 
\|f_{N_1, L_1}|_{\tilde{\mathfrak{D}}_{k_1}^M} \|_{L^2} \| g_{N_2, L_2}|_{\tilde{\mathfrak{D}}_{k_2}^M} \|_{L^2} 
\|h_{{N_0, L_0}}\|_{L^2},\label{est01-prop3.19}
\end{split}
\end{equation}
where functions $f_{N_1, L_1}$, $g_{N_2, L_2}$, $h_{{N_0, L_0}}$ satisfy 
\eqref{assumption-fgh}.
\end{prop}
%%%%%%%%%%%%%%%%%%%%%%%%%%%%%%%%%%%%%
%%%%%%%%%%%%%%%%%%%%%%%%%%%%%%%%%%%%%
%%%%%%%%%%%%%%%%%%%%%%%%%%%%%%%%%%%%%
%%%%%%%%%%%%%%%%%%%%%%%%%%%%%%%%%%%%%
\begin{proof}
We easily confirm that 
$|\overline{\Phi} (\xi_1, \eta_1, \xi_2 ,\eta_2)| \gtrsim N_0 N_1^2$ holds. 
This and Proposition \ref{prop3.18} immediately yield \eqref{est01-prop3.19}. 
\end{proof}
%%%%%%%%%%%%%%%%%%%%%%%%%%%%%%%%%%%%%
%%%%%%%%%%%%%%%%%%%%%%%%%%%%%%%%%%%%%
%%%%%%%%%%%%%%%%%%%%%%%%%%%%%%%%%%%%%
%%%%%%%%%%%%%%%%%%%%%%%%%%%%%%%%%%%%%
\begin{prop}\label{prop3.20}
Let $M$ be dyadic such that $1 \ll M \leq A$. 
Assume that $|\boldsymbol{\eta_j'}| \lesssim A^{-1} N_1$, $N_0 \sim M^{-1} N_1$, $k_1 \in \mathfrak{K}_M^K$, $k_2$ satisfy
 $16 \leq |k_1-k_2| \leq 32$. Then we have
\begin{align*}
& \left|\int_{*}{  h_{{N_0, L_0}}(\tau, \xi, \boldsymbol{\eta} ) 
f_{N_1, L_1}|_{\tilde{\mathfrak{D}}_{k_1}^M} (\tau_1, \xi_1, \boldsymbol{\eta_1} )  
g_{N_2, L_2}|_{\tilde{\mathfrak{D}}_{k_2}^M}  (\tau_2, \xi_2, \boldsymbol{\eta_2} ) 
}
d\sigma_1 d\sigma_2 \right| \\
&   \lesssim  A^{-\frac{d-2}{2}}K^{\frac{1}{2}}N_0^{-\frac{1}{2}}
N_1^{\frac{d-5}{2}}    (L_0 L_1 L_2)^{\frac{1}{2}} 
\|f_{N_1, L_1}|_{\tilde{\mathfrak{D}}_{k_1}^M} \|_{L^2} \| g_{N_2, L_2}|_{\tilde{\mathfrak{D}}_{k_2}^M} \|_{L^2} 
\|h_{{N_0, L_0}}\|_{L^2},
\end{align*}
where functions $f_{N_1, L_1}$, $g_{N_2, L_2}$, $h_{{N_0, L_0}}$ satisfy 
\eqref{assumption-fgh}.
\end{prop}
%%%%%%%%%%%%%%%%%%%%%%%%%%%%%%%%%%%%%
%%%%%%%%%%%%%%%%%%%%%%%%%%%%%%%%%%%%%
\begin{proof}
It suffices to show
\begin{equation}
\begin{split}
& \left|\int_{\hat{*}}{  h_{{N_0, L_0}}(\tau, \xi, \boldsymbol{\eta} ) 
f_{N_1, L_1}|_{\tilde{\mathfrak{D}}_{k_1}^M} (\tau_1, \xi_1, \boldsymbol{\eta_1} )
g_{N_2, L_2}|_{\tilde{\mathfrak{D}}_{k_2}^M} (\tau_2, \xi_2, \boldsymbol{\eta_2} )
}
d\hat{\sigma}_1 d\hat{\sigma}_2 \right| \\
&  \lesssim (MK)^{\frac{1}{2}} N_{1}^{-2}   (L_0 L_1 L_2)^{\frac{1}{2}} 
\|f_{N_1, L_1}|_{\tilde{\mathfrak{D}}_{k_1}^M} (\boldsymbol{\eta_1'}) \|_{L^2_{\tau \xi \eta}} 
\|g_{N_2, L_2}|_{\tilde{\mathfrak{D}}_{k_2}^M} (\boldsymbol{\eta_2'})\|_{L^2_{\tau \xi \eta}} 
\|h_{{N_0, L_0}}(\boldsymbol{\eta'}) \|_{L^2_{\tau \xi \eta}},\label{est01-prop3.20}
\end{split}
\end{equation}
for fixed $\boldsymbol{\eta_1'}$, $\boldsymbol{\eta_2'}$. 
By using Proposition \ref{prop3.18} and smallness of $|\boldsymbol{\eta_1'}|$ and $|\boldsymbol{\eta_2'}|$, \eqref{est01-prop3.20} can be obtained in the same way as for Proposition 3.20 in \cite{Ki2019}. We omit the details.
\end{proof}
%%%%%%%%%%%%%%%%%%%%%%%%%%%%%%%%%%%%%
%%%%%%%%%%%%%%%%%%%%%%%%%%%%%%%%%%%%%
%%%%%%%%%%%%%%%%%%%%%%%%%%%%%%%%%%%%%
%%%%%%%%%%%%%%%%%%%%%%%%%%%%%%%%%%%%%
Next we deal with the case $(k_1,k_2) \in \mathfrak{K}_M \times \mathfrak{K}_M$.
\begin{defn}
Let $M$ and $\nu$ be dyadic such that $1 \ll M \leq A^{2/3}$, $2 \leq \nu \leq A M^{-3/2}$ and $m = ( m_{(1)}, m_{(2)} ) \in \Z^2$. We define rectangle-tiles 
$\{ \mathcal{T}_m^{M,\nu}\}_{m \in \Z^2}$ whose short side is parallel to $\xi$-axis and its length is $M^{-3/2} \nu^{-1} N_1$, long side length is $M^{-1} \nu^{-1} N_1 $ and prisms $\{ \tilde{\mathcal{T}}_m^{M,\nu}\}_{m \in \Z^2}$ as follows.
\begin{align*}
& \mathcal{T}_m^{M,\nu} : = \{ (\xi,\eta) \in \R^2 \, | \, \xi \in M^{-\frac{3}{2}} \nu^{-1} N_1 
[ m_{(1)}, m_{(1)} + 1), \ \eta \in M^{-1} \nu^{-1} N_1 [ m_{(2)}, m_{(2)} + 1)  \}\\
& \tilde{\mathcal{T}}_m^{M,\nu}  : = \R \times \mathcal{T}_m^{M,\nu} \times \R^{d-2}.
\end{align*}
Recall that
\begin{align*}
\overline{\Phi} (\xi_1, \eta_1, \xi_2 ,\eta_2) & = \xi_1 \xi_2(\xi_1 + \xi_2) + \eta_1 \eta_2 (\eta_1 + \eta_2), \\
\overline{F} (\xi_1, \eta_1, \xi_2 ,\eta_2) & = \xi_1 \eta_2 +  \xi_2 \eta_1 + 2 (\xi_1 \eta_1 + \xi_2 \eta_2).
\end{align*}
Let $\mathbf{k} :=(k_1,k_2) \in \mathfrak{K}_M \times \mathfrak{K}_M$. 
We define $Z_{M,\nu,\mathbf{k}}^1 $ as the set of $(m_1, m_2) \in \Z^2 \times \Z^2$ such that
\begin{equation*}
\begin{cases}
 |\overline{\Phi}(\xi_1, \eta_1, \xi_2, \eta_2)| \geq M^{-\frac{3}{2}} \nu^{-1} N_1^3  \ \ \textnormal{for any} \ 
(\xi_1, \eta_1) \times (\xi_2,\eta_2) \in 
\mathcal{T}_{m_1}^{M,\nu} \times \mathcal{T}_{m_2}^{M,\nu}, \\
 \left( \mathcal{T}_{m_1}^{M,\nu} \times \mathcal{T}_{m_2}^{M,\nu} \right) 
\cap \left( \mathfrak{D}_{k_1}^M \times \mathfrak{D}_{k_2}^M \right) 
\not= \emptyset,\\
 |\xi_1+ \xi_2| \lesssim M^{-3/2} N_1  \ \ \textnormal{for any} \ 
(\xi_1, \eta_1) \times (\xi_2,\eta_2) \in 
\mathcal{T}_{m_1}^{M,\nu} \times \mathcal{T}_{m_2}^{M,\nu}.
\end{cases}
\end{equation*}
Similarly, we define $Z_{M,\nu,\mathbf{k}}^2$ 
as the set of $(m_1, m_2) \in \Z^2 \times \Z^2$ such that
\begin{equation*}
\begin{cases}
 |\overline{F}(\xi_1, \eta_1, \xi_2, \eta_2)| \geq M^{-1} \nu^{-1} N_1^2  \ \ \textnormal{for any} \ 
(\xi_1, \eta_1) \times (\xi_2,\eta_2) \in 
\mathcal{T}_{m_1}^{M,\nu} \times \mathcal{T}_{m_2}^{M,\nu}, \\
 \left( \mathcal{T}_{m_1}^{M,\nu} \times \mathcal{T}_{m_2}^{M,\nu} \right) 
\cap \left( \mathfrak{D}_{k_1}^M \times \mathfrak{D}_{k_2}^M \right) 
\not= \emptyset,\\
 |\xi_1+ \xi_2| \lesssim M^{-3/2} N_1  \ \ \textnormal{for any} \ 
(\xi_1, \eta_1) \times (\xi_2,\eta_2) \in 
\mathcal{T}_{m_1}^{M,\nu} \times \mathcal{T}_{m_2}^{M,\nu},
\end{cases}
\end{equation*}
and
\begin{equation*}
Z_{M,\nu}^{\mathbf{k}} = Z_{M,\nu,\mathbf{k}}^1 \, \cup \, Z_{M,\nu,\mathbf{k}}^2 , 
\quad R_{M,\nu}^{\mathbf{k}} = \bigcup_{(m_1, m_2) \in Z_{M,\nu}^{\mathbf{k}}} \mathcal{T}_{m_1}^{M,\nu} \times 
\mathcal{T}_{m_2}^{M,\nu} \subset \R^2 \times \R^2.
\end{equation*}
It is clear that 
$\nu_1 \leq \nu_2 \Longrightarrow  R_{M,\nu_1}^{\mathbf{k}} \subset  R_{M,\nu_2}^{\mathbf{k}}$. 
Further, we define
\begin{equation*}
Q_{M,\nu}^{\mathbf{k}} = 
\begin{cases}
R_{M,\nu}^{\mathbf{k}} \setminus R_{M,\nu/2}^{\mathbf{k}} & \textnormal{for}  \  \nu > 2,\\
 \ R_{M,2}^{\mathbf{k}}  & \textnormal{for}  \  \nu = 2.
\end{cases}
\end{equation*}
and a set of pairs of integer pair $\widehat{Z}_{M,\nu}^{\mathbf{k}} \subset Z_{M,\nu}^{\mathbf{k}}$ as
\begin{equation*}
\bigcup_{(m_1, m_2) \in \widehat{Z}_{M,\nu}^{\mathbf{k}}} \mathcal{T}_{m_1}^{M,\nu} \times 
\mathcal{T}_{m_2}^{M,\nu} = Q_{M,\nu}^{\mathbf{k}}.
\end{equation*}
Clearly, $\widehat{Z}_{M,\nu}^{\mathbf{k}}$ is uniquely defined and 
\begin{equation*}
\nu_1 \not= \nu_2 \Longrightarrow Q_{M,\nu_1}^{\mathbf{k}} \cap Q_{M,\nu_2}^{\mathbf{k}} = \emptyset, \quad 
\bigcup_{2 \leq \nu \leq \nu_0} Q_{M,\nu}^{\mathbf{k}} = R_{M,\nu_0}^{\mathbf{k}}
\end{equation*}
where $\nu_0 \geq 2$ is dyadic. 
Lastly, we define $\overline{Z}_{M,\nu}^{\mathbf{k}}$ as the collection of $(m_1, m_2) \in \Z^2 \times \Z^2$ which satisfies
\begin{equation*}
\begin{cases}
\mathcal{T}_{m_1}^{M,\nu} \times \mathcal{T}_{m_2}^{M,\nu} \not\subset 
\displaystyle{\bigcup_{2 \leq \nu' \leq \nu} 
\bigcup_{(m_1', m_2') \in \widehat{Z}_{M,\nu'}^{\mathbf{k}}}} 
\left( \mathcal{T}_{m_1'}^{M,d'} \times \mathcal{T}_{m_2'}^{M,d'} \right),\\
 \left( \mathcal{T}_{m_1}^{M,\nu} \times \mathcal{T}_{m_2}^{M,\nu} \right) 
\cap \left( \mathfrak{D}_{k_1}^M \times \mathfrak{D}_{k_2}^M \right) 
\not= \emptyset,\\
 |\xi_1+ \xi_2| \lesssim M^{-3/2} N_1  \ \ \textnormal{for any} \ 
(\xi_1, \eta_1) \times (\xi_2,\eta_2) \in 
\mathcal{T}_{m_1}^{M,\nu} \times \mathcal{T}_{m_2}^{M,\nu}.
\end{cases}
\end{equation*}
\end{defn}
%%%%%%%%%%%%%%%%%%%%%%%%%%%%%%%%%%%%%
%%%%%%%%%%%%%%%%%%%%%%%%%%%%%%%%%%%%%
%%%%%%%%%%%%%%%%%%%%%%%%%%%%%%%%%%%%%
%%%%%%%%%%%%%%%%%%%%%%%%%%%%%%%%%%%%%
\begin{prop}\label{prop3.21}
Assume $|\boldsymbol{\eta_j'}| \lesssim A^{-1} N_1$ and \eqref{assumption-fgh}. 
Let $M$ and $\nu$ be dyadic such that $1 \ll M \leq A^{2/3}$, $2 \leq \nu \leq A M^{-3/2}$, 
$16 \leq |k_1-k_2| \leq 32$ and 
$( m_1, m_2 ) \in \widehat{Z}_{M,\nu}^{\mathbf{k}}$. Then we get
\begin{equation}
\begin{split}
& \left|\int_{*}{  h_{{N_0, L_0}}(\tau, \xi, \boldsymbol{\eta} ) 
f_{N_1, L_1}|_{\tilde{\mathcal{T}}_{m_1}^{M,\nu}} (\tau_1, \xi_1, \boldsymbol{\eta_1} )
g_{N_2, L_2}|_{\tilde{\mathcal{T}}_{m_2}^{M,\nu}} (\tau_2, \xi_2, \boldsymbol{\eta_2} )
}
d\sigma_1 d\sigma_2 \right| \\
&  \lesssim  A^{-\frac{d-2}{2}} M {\nu}^{\frac{1}{2}}
N_1^{\frac{d-6}{2}}    (L_0 L_1 L_2)^{\frac{1}{2}} 
\|f_{N_1, L_1}|_{\tilde{\mathcal{T}}_{m_1}^{M,\nu}} \|_{L^2} 
\|g_{N_2, L_2}|_{\tilde{\mathcal{T}}_{m_2}^{M,\nu}} \|_{L^2} 
\|h_{{N_0, L_0}}\|_{L^2},\label{est01-prop3.21}
\end{split}
\end{equation}
where $d \sigma_j = d\tau_j d \xi_j d \boldsymbol{\eta}_j$ and $*$ denotes $(\tau, \xi, \boldsymbol{\eta}) = (\tau_1 + \tau_2, \xi_1+ \xi_2, \boldsymbol{\eta_1} + \boldsymbol{\eta_2}).$
\end{prop}
%%%%%%%%%%%%%%%%%%%%%%%%%%%%%%%%%%%%%
%%%%%%%%%%%%%%%%%%%%%%%%%%%%%%%%%%%%%
%%%%%%%%%%%%%%%%%%%%%%%%%%%%%%%%%%%%%
%%%%%%%%%%%%%%%%%%%%%%%%%%%%%%%%%%%%%
\begin{proof}
It suffices to show the following inequality.
\begin{equation}
\begin{split}
& \left|\int_{\hat{*}}{  h_{{N_0, L_0}}(\tau, \xi, \boldsymbol{\eta} ) 
f_{N_1, L_1}|_{\tilde{\mathcal{T}}_{m_1}^{M,\nu}} (\tau_1, \xi_1, \boldsymbol{\eta_1} )
g_{N_2, L_2}|_{\tilde{\mathcal{T}}_{m_2}^{M,\nu}} (\tau_2, \xi_2, \boldsymbol{\eta_2} )
}
d\hat{\sigma}_1 d\hat{\sigma}_2 \right| \\
&  \lesssim M {\nu}^{\frac{1}{2}}N_{1}^{-2}   (L_0 L_1 L_2)^{\frac{1}{2}} 
\|f_{N_1, L_1}|_{\tilde{\mathcal{T}}_{m_1}^{M,\nu}} (\boldsymbol{\eta_1'}) \|_{L^2_{\tau \xi \eta}} 
\|g_{N_2, L_2}|_{\tilde{\mathcal{T}}_{m_2}^{M,\nu}} (\boldsymbol{\eta_2'})\|_{L^2_{\tau \xi \eta}} 
\|h_{{N_0, L_0}}(\boldsymbol{\eta'}) \|_{L^2_{\tau \xi \eta}}.\label{est02-prop3.21}
\end{split}
\end{equation}
Similarly to the proof of Proposition \ref{prop3.10},  since $|\boldsymbol{\eta_j'}| \lesssim A^{-1} N_1$, we can reuse the proofs of Propositions 3.22 and 3.23 in \cite{Ki2019} 
with slight modifications to get the estimate \eqref{est02-prop3.21}. We omit the proof.
\end{proof}
%%%%%%%%%%%%%%%%%%%%%%%%%%%%%%%%%%%%%
%%%%%%%%%%%%%%%%%%%%%%%%%%%%%%%%%%%%%
%%%%%%%%%%%%%%%%%%%%%%%%%%%%%%%%%%%%%
\begin{lem}\label{lemma3.22}
Let $M$ and $d$ be dyadic such that $1 \ll M \leq A^{2/3}$, $2 \leq \nu \leq A M^{-3/2}$ 
and $k_1$, $k_2 \in \mathfrak{K}_M$. 
For fixed $m_1 \in \Z^2$, the number of $m_2 \in \Z^2$ such that 
$(m_1, m_2) \in \widehat{Z}_{M,\nu}^{\mathbf{k}}$ is finitely many. Furthermore, the same claim holds true if we replace $\widehat{Z}_{M,\nu}^{\mathbf{k}}$ by $\overline{Z}_{M,\nu}^{\mathbf{k}}$.
\end{lem}
%%%%%%%%%%%%%%%%%%%%%%%%%%%%%%%%%%%%%
%%%%%%%%%%%%%%%%%%%%%%%%%%%%%%%%%%%%%
%%%%%%%%%%%%%%%%%%%%%%%%%%%%%%%%%%%%%
%%%%%%%%%%%%%%%%%%%%%%%%%%%%%%%%%%%%%
\begin{prop}\label{prop3.23}
Let $1 \ll M \leq A^{2/3}$. 
Assume that $|\boldsymbol{\eta_j'}| \lesssim A^{-1} N_1$ and $k_1$, $k_2 \in \mathfrak{K}_M$ satisfy
 $16 \leq |k_1-k_2| \leq 32$. Then we have
\begin{equation}
\begin{split}
& \left|\int_{*}{  h_{{N_0, L_0}}(\tau, \xi, \boldsymbol{\eta} ) 
f_{N_1, L_1}|_{\tilde{\mathfrak{D}}_{k_1}^M} (\tau_1, \xi_1, \boldsymbol{\eta_1} )  
g_{N_2, L_2}|_{\tilde{\mathfrak{D}}_{k_2}^M}  (\tau_2, \xi_2, \boldsymbol{\eta_2} ) 
}
d\sigma_1 d\sigma_2 \right| \\
& \qquad  \lesssim  A^{-\frac{d-3}{2}} N_0^{-\frac{1}{4}}
N_1^{\frac{2d-11}{4}}    (L_0 L_1 L_2)^{\frac{1}{2}} 
\|f_{N_1, L_1}|_{\tilde{\mathfrak{D}}_{k_1}^M} \|_{L^2} \| g_{N_2, L_2}|_{\tilde{\mathfrak{D}}_{k_2}^M} \|_{L^2} 
\|h_{{N_0, L_0}}\|_{L^2},\label{est01-prop3.23}
\end{split}
\end{equation}
where functions $f_{N_1, L_1}$, $g_{N_2, L_2}$, $h_{{N_0, L_0}}$ satisfy 
\eqref{assumption-fgh}.
\end{prop}
%%%%%%%%%%%%%%%%%%%%%%%%%%%%%%%%%%%%%
%%%%%%%%%%%%%%%%%%%%%%%%%%%%%%%%%%%%%
%%%%%%%%%%%%%%%%%%%%%%%%%%%%%%%%%%%%%
%%%%%%%%%%%%%%%%%%%%%%%%%%%%%%%%%%%%%
\begin{proof}
First we consider $|\xi_1+\xi_2| \gg M^{-3/2} N_1$. 
It is easily observed that $|\xi_1+\xi_2| \gg M^{-3/2} N_1$ means 
$|\overline{\Phi} (\xi_1, \eta_1, \xi_2 ,\eta_2)| \gg M^{-3/2}N_1^2$. Then, by using \eqref{bilinear-est04-prop3.18} in Proposition \ref{prop3.18}, we obtain
\begin{align*}
& \left|\int_{*}{  h_{{N_0, L_0}}(\tau, \xi, \boldsymbol{\eta} ) 
f_{N_1, L_1}|_{\tilde{\mathfrak{D}}_{k_1}^M} (\tau_1, \xi_1, \boldsymbol{\eta_1} )  
g_{N_2, L_2}|_{\tilde{\mathfrak{D}}_{k_2}^M}  (\tau_2, \xi_2, \boldsymbol{\eta_2} ) 
}
d\sigma_1 d\sigma_2 \right| \\
& \qquad  \lesssim  A^{-\frac{d-2}{2}} M^{\frac{1}{4}}
N_1^{\frac{d-6}{2}}    (L_0 L_1 L_2)^{\frac{1}{2}} 
\|f_{N_1, L_1}|_{\tilde{\mathfrak{D}}_{k_1}^M} \|_{L^2} \| g_{N_2, L_2}|_{\tilde{\mathfrak{D}}_{k_2}^M} \|_{L^2} 
\|h_{{N_0, L_0}}\|_{L^2},
\end{align*}
which completes the proof for the case $|\xi_1+\xi_2| \gg M^{-3/2} N_1$.
 
We consider the case $|\xi_1+\xi_2| \lesssim M^{-3/2} N_1$. 
For simplicity, we assume $\supp f_{N_1, L_1} \subset \tilde{\mathfrak{D}}_{k_1}^M$ and 
$\supp g_{N_2, L_2} \subset \tilde{\mathfrak{D}}_{k_2}^M$, and use
\begin{equation*}
I_{M, \nu}^{m_1, m_2}:=\left|\int_{*}{  h_{{N_0, L_0}}(\tau, \xi, \boldsymbol{\eta} ) 
f_{N_1, L_1}|_{\tilde{\mathcal{T}}_{m_1}^{M,\nu}} (\tau_1, \xi_1, \boldsymbol{\eta_1} )
g_{N_2, L_2}|_{\tilde{\mathcal{T}}_{m_2}^{M,\nu}} (\tau_2, \xi_2, \boldsymbol{\eta_2} )
}
d\sigma_1 d\sigma_2 \right| 
\end{equation*}
Let $\nu_0$ denote the maximal dyadic number which satisfies $\nu_0 \leq A M^{-3/2}$. 
By the definition of 
$ \widehat{Z}_{M,\nu}^{\mathbf{k}}$ and $ \overline{Z}_{M,\nu}^{\mathbf{k}}$, we observe that
\begin{align*}
 \textnormal{(LHS) of \eqref{est01-prop3.23}}
& \leq \sum_{2 \leq \nu \leq \nu_0} \sum_{(m_1, m_2) \in\widehat{Z}_{M,\nu}^{\mathbf{k}}} 
I_{M, \nu}^{m_1, m_2}  + \sum_{(m_1, m_2) \in \overline{Z}_{M,\nu_0}^{\mathbf{k}}} 
I_{M, \nu_0}^{m_1, m_2}.
\end{align*}
It follows from Proposition \ref{prop3.21} and Lemma \ref{lemma3.22} that
\begin{align*}
& \sum_{(m_1, m_2) \in\widehat{Z}_{M,\nu}^{\mathbf{k}}} 
I_{M, \nu}^{m_1, m_2} \\
& \lesssim A^{-\frac{d-2}{2}} M {\nu}^{\frac{1}{2}}
N_1^{\frac{d-6}{2}}    (L_0 L_1 L_2)^{\frac{1}{2}} \sum_{(m_1, m_2) \in\widehat{Z}_{M,\nu}^{\mathbf{k}}}
\|f_{N_1, L_1}|_{\tilde{\mathcal{T}}_{m_1}^{M,\nu}} \|_{L^2} 
\|g_{N_2, L_2}|_{\tilde{\mathcal{T}}_{m_2}^{M,\nu}} \|_{L^2} 
\|h_{{N_0, L_0}}\|_{L^2}\\
& \lesssim A^{-\frac{d-2}{2}} M {\nu}^{\frac{1}{2}}
N_1^{\frac{d-6}{2}}    (L_0 L_1 L_2)^{\frac{1}{2}} 
\|f_{N_1, L_1} \|_{L^2} 
\|g_{N_2, L_2}\|_{L^2} 
\|h_{{N_0, L_0}}\|_{L^2},
\end{align*}
which gives
\begin{align*}
 \sum_{2 \leq \nu \leq \nu_0} & \sum_{(m_1, m_2) \in\widehat{Z}_{M,\nu}^{\mathbf{k}}} 
I_{M, \nu}^{m_1, m_2}\\
\lesssim A^{-\frac{d-3}{2}} & M^{\frac{1}{4}}
N_1^{\frac{d-6}{2}}    (L_0 L_1 L_2)^{\frac{1}{2}} 
\|f_{N_1, L_1}|_{\tilde{\mathfrak{D}}_{k_1}^M} \|_{L^2} \| g_{N_2, L_2}|_{\tilde{\mathfrak{D}}_{k_2}^M} \|_{L^2} 
\|h_{{N_0, L_0}}\|_{L^2}.
\end{align*}
Since $|\xi_1+\xi_2| \lesssim M^{-3/2} N_1$, we can assume $N_0 \sim M^{-1} N_1$. 
Then this completes the desired estimate for the first term. 
For the second term, we first note that $\mathcal{T}_{m}^{M,\nu_0} \subset \R^2$ 
is a rectangle set whose short-side length is $\sim A^{-1}$ and long-side length is $\sim A^{-1} M^{1/2}$. 
Then we can decompose $\mathcal{T}_{m}^{M,\nu_0}$ into $\sim M^{1/2}$ number of square tiles whose side length is $A^{-1}$. 
Thus, by Proposition \ref{prop3.7} and the almost orthogonality, we observe
\begin{equation*}
I_{M, \nu_0}^{m_1, m_2}  \lesssim A^{-\frac{d-3}{2}} M^{\frac{1}{4}}
N_1^{\frac{d-6}{2}}   (L_0 L_1 L_2)^{\frac{1}{2}} 
\|f_{N_1, L_1}|_{\tilde{\mathcal{T}}_{m_1}^{M, \nu_0}} \|_{L^2} 
\| g_{N_2, L_2}|_{\tilde{\mathcal{T}}_{m_2}^{M,\nu_0}} \|_{L^2} 
\|h_{{N_0, L_0}} \|_{L^2}.
\end{equation*}
Consequently, by Lemma \ref{lemma3.22}, we obtain
\begin{align*}
 \sum_{(m_1, m_2) \in \overline{Z}_{M,\nu_0}^{\mathbf{k}}}  I_{M, \nu_0}^{m_1, m_2}  
 \lesssim  A^{-\frac{d-3}{2}} & M^{\frac{1}{4}}
N_1^{\frac{d-6}{2}}   (L_0 L_1 L_2)^{\frac{1}{2}} 
\|f_{N_1, L_1}|_{\tilde{\mathfrak{D}}_{k_1}^M} \|_{L^2} \| g_{N_2, L_2}|_{\tilde{\mathfrak{D}}_{k_2}^M} \|_{L^2} 
\|h_{{N_0, L_0}}\|_{L^2},
\end{align*}
which completes the proof.
\end{proof}
%%%%%%%%%%%%%%%%%%%%%%%%%%%%%%%%%%%%%
%%%%%%%%%%%%%%%%%%%%%%%%%%%%%%%%%%%%%
%%%%%%%%%%%%%%%%%%%%%%%%%%%%%%%%%%%%%
%%%%%%%%%%%%%%%%%%%%%%%%%%%%%%%%%%%%%
\begin{proof}[\underline{Proof of Proposition \ref{prop3.17}}]
We should recall that we can assume $A^{-1} N_1 \lesssim N_0$. 
Let $M$ be dyadic such that $1 \ll M \leq A^{2/3}$ and $M_0$ be the maximal dyadic number which satisfies $M_0 \leq A^{2/3}$. We define
\begin{equation*}
K_{M}^{\mathcal{I}} = \{ (k_1, k_2) \, | \, 0 \leq k_1, k_2 \leq M -1, \ ( {\mathfrak{D}}_{k_1}^M \times {\mathfrak{D}}_{k_2}^M ) \subset ( \mathfrak{D}_{0}^{2^{11}} \times 
\mathfrak{D}_{0}^{2^{11}} ).\}
\end{equation*}
It is observed that
\begin{align*}
\mathfrak{D}_{0}^{2^{11}} \times 
\mathfrak{D}_{0}^{2^{11}} =   \bigcup_{1 \ll M \leq M_0} \ \bigcup_{\tiny{\substack{(k_1,k_2) \in K_{M}^{\mathcal{I}}\\ 16 \leq |k_1 - k_2|\leq 32}}} 
{\mathfrak{D}}_{k_1}^M \cross {\mathfrak{D}}_{k_2}^M \cup 
\bigcup_{\tiny{\substack{(k_1,k_2) \in K_{M_0}^{\mathcal{I}}\\|k_1 - k_2|\leq 16}}} 
{\mathfrak{D}}_{k_1}^{M_0} \cross {\mathfrak{D}}_{k_2}^{M_0}.
\end{align*}
Let us write
\begin{equation*}
I_{M}^{k_1, k_2} := \left|\int_{*}{  h_{{N_0, L_0}}(\tau, \xi, \boldsymbol{\eta} ) 
f_{N_1, L_1}|_{\tilde{\mathfrak{D}}_{k_1}^M} (\tau_1, \xi_1, \boldsymbol{\eta_1} )  
g_{N_2, L_2}|_{\tilde{\mathfrak{D}}_{k_2}^M}  (\tau_2, \xi_2, \boldsymbol{\eta_2} ) 
}
d\sigma_1 d\sigma_2 \right|. 
\end{equation*}
We calculate that
\begin{align*}
& \left|\int_{*}{  h_{{N_0, L_0}}(\tau, \xi, \boldsymbol{\eta} ) 
f_{N_1, L_1}|_{\tilde{\mathfrak{D}}_{0}^{2^{11}}} (\tau_1, \xi_1, \boldsymbol{\eta_1} )  
g_{N_2, L_2}|_{\tilde{\mathfrak{D}}_{0}^{2^{11}}} (\tau_2, \xi_2, \boldsymbol{\eta_2} ) 
}
d\sigma_1 d\sigma_2 \right| \\
&   \lesssim   \sum_{1 \ll M \leq M_0} \ \sum_{\tiny{\substack{(k_1,k_2) \in K_{M}^{\mathcal{I}}\\ 16 \leq |k_1 - k_2|\leq 32}}} I_{M}^{k_1, k_2} + 
\sum_{\tiny{\substack{(k_1,k_2) \in K_{M_0}^{\mathcal{I}}\\|k_1 - k_2|\leq 16}}} 
I_{M_0}^{k_1, k_2}.
\end{align*}
We consider the former term. 
Since $M \leq A^{2/3}$ and $16 \leq |k_1-k_2|$ we may assume $A^{-2/3} N_1 \lesssim N_0$. 
This and $2^{10}K \leq M$ mean 
$A^{-(d-2)/2} K^{1/2}N_0^{-1/2} N_1^{(d-5)/2} \lesssim A^{-(d-3)/2} N_0^{-1/4}N_1^{(2 d -11)/4}$. 
Then by using Propositions \ref{prop3.19}, \ref{prop3.20} and \ref{prop3.23}, we obtain
\begin{align*}
& \sum_{\tiny{\substack{(k_1,k_2) \in K_{M}^{\mathcal{I}}\\ 16 \leq |k_1 - k_2|\leq 32}}} I_{M}^{k_1, k_2} \\
& \lesssim 
A^{-\frac{d-3}{2}} N_0^{-\frac{1}{4}}
N_1^{\frac{2d-11}{4}}    (L_0 L_1 L_2)^{\frac{1}{2}} \sum_{\tiny{\substack{(k_1,k_2) \in K_{M}^{\mathcal{I}}\\ 16 \leq |k_1 - k_2|\leq 32}}}
\|f_{N_1, L_1}|_{\tilde{\mathfrak{D}}_{k_1}^M} \|_{L^2} \| g_{N_2, L_2}|_{\tilde{\mathfrak{D}}_{k_2}^M} \|_{L^2} 
\|h_{{N_0, L_0}}\|_{L^2}\\
& \lesssim 
A^{-\frac{d-3}{2}} N_0^{-\frac{1}{4}}
N_1^{\frac{2d-11}{4}}    (L_0 L_1 L_2)^{\frac{1}{2}} 
\|f_{N_1, L_1} \|_{L^2} \| g_{N_2, L_2} \|_{L^2} 
\|h_{{N_0, L_0}}\|_{L^2}.
\end{align*}
Therefore, we obtain
\begin{align*}
& \sum_{1 \ll M \leq M_0}  \ \sum_{\tiny{\substack{(k_1,k_2) \in K_{M}^{\mathcal{I}}\\ 16 \leq |k_1 - k_2|\leq 32}}} I_{M}^{k_1, k_2}\\
& \quad \lesssim  (\log A) A^{-\frac{d-3}{2}} N_0^{-\frac{1}{4}}
N_1^{\frac{2d-11}{4}}    (L_0 L_1 L_2)^{\frac{1}{2}} 
\|f_{N_1, L_1} \|_{L^2} \| g_{N_2, L_2} \|_{L^2} 
\|h_{{N_0, L_0}}\|_{L^2}.
\end{align*}
This gives the desired estimate since $A^{-1} N_1 \lesssim N_0$. 

For the latter term, letting $2^{10} \leq K \leq 2^{-10}M_0$, we first assume $k_1 \in \mathfrak{K}_{M_0}^K$. 
Define 
\begin{equation*}
K_{M, M_0}^{k_1,k_2} = \{ (k_1', k_2') \, | \, ( {\mathfrak{D}}_{k_1'}^M \times {\mathfrak{D}}_{k_2'}^M ) \subset ( {\mathfrak{D}}_{k_1}^{M_0} \cross {\mathfrak{D}}_{k_2}^{M_0} ).\}
\end{equation*}
Let $M'$ be the maximal dyadic number which satisfies $M' \leq A K^{-1/2}$. 
If $|k_1-k_2| \leq 16$, we have
\begin{equation*}
{\mathfrak{D}}_{k_1}^{M_0} \cross {\mathfrak{D}}_{k_2}^{M_0} = 
 \bigcup_{2 M_0 \leq M \leq M'} \ 
\bigcup_{\tiny{\substack{(k_1',k_2') \in K_{M, M_0}^{k_1,k_2}\\ 16 \leq |k_1' - k_2'|\leq 32}}} 
{\mathfrak{D}}_{k_1'}^M \cross {\mathfrak{D}}_{k_2'}^M \cup 
\bigcup_{\tiny{\substack{(k_1',k_2') \in K_{M', M_0}^{k_1,k_2}\\|k_1' - k_2'|\leq 16}}} 
{\mathfrak{D}}_{k_1'}^{M'} \cross {\mathfrak{D}}_{k_2'}^{M'}.
\end{equation*}
This implies
\begin{equation}
I_{M_0}^{k_1, k_2} \lesssim \sum_{2 M_0 \leq M \leq M'} \ \sum_{\tiny{\substack{(k_1',k_2') \in K_{M, M_0}^{k_1,k_2}\\ 16 \leq |k_1' - k_2'|\leq 32}}}  I_{M}^{k_1', k_2'} + 
\sum_{\tiny{\substack{(k_1',k_2') \in K_{M', M_0}^{k_1,k_2}\\|k_1' - k_2'|\leq 16}}} 
I_{M'}^{k_1', k_2'}.\label{est02-prop3.17}
\end{equation}
For the former term, we may assume $N_0 \gtrsim A^{-1} K^{1/2}N_1$. 
It follows from Propositions \ref{prop3.19} and \ref{prop3.20} that
\begin{align*}
& \sum_{2 M_0 \leq M \leq M'} \ \sum_{\tiny{\substack{(k_1',k_2') \in K_{M, M_0}^{k_1,k_2}\\ 16 \leq |k_1 - k_2|\leq 32}}}  I_{M}^{k_1', k_2'}\\
& \lesssim (\log A) A^{-\frac{d-2}{2}}K^{\frac{1}{2}}N_0^{-\frac{1}{2}}
N_1^{\frac{d-5}{2}}    (L_0 L_1 L_2)^{\frac{1}{2}} 
\|f_{N_1, L_1}|_{\tilde{\mathfrak{D}}_{k_1}^{M_0}} \|_{L^2} 
\| g_{N_2, L_2}|_{\tilde{\mathfrak{D}}_{k_2}^{M_0}} \|_{L^2} 
\|h_{{N_0, L_0}}\|_{L^2}\\
& \lesssim (\log A)  A^{-\frac{2d-5}{4}} K^{\frac{3}{8}}N_0^{-\frac{1}{4}}
N_1^{\frac{2d-11}{4}}    (L_0 L_1 L_2)^{\frac{1}{2}} 
\|f_{N_1, L_1}|_{\tilde{\mathfrak{D}}_{k_1}^{M_0}} \|_{L^2} 
\| g_{N_2, L_2}|_{\tilde{\mathfrak{D}}_{k_2}^{M_0}} \|_{L^2} 
\|h_{{N_0, L_0}}\|_{L^2}.
\end{align*}
We next consider the latter term. 
If $N_0 \gg N_1/M'$ Proposition \ref{prop3.19} yields
\begin{align*}
I_{M'}^{k_1', k_2'} & \lesssim A^{-\frac{d-2}{2}} N_0^{-\frac{1}{2}}
N_1^{\frac{d-5}{2}}    (L_0 L_1 L_2)^{\frac{1}{2}} \|f_{N_1, L_1}|_{\tilde{\mathfrak{D}}_{k_1'}^{M'}} \|_{L^2} 
\| g_{N_2, L_2}|_{\tilde{\mathfrak{D}}_{k_2'}^{M'}} \|_{L^2} 
\|h_{{N_0, L_0}}\|_{L^2}\\
&  \lesssim A^{-\frac{2d-5}{4}} N_0^{-\frac{1}{4}}
N_1^{\frac{2d-11}{4}}    (L_0 L_1 L_2)^{\frac{1}{2}} 
\|f_{N_1, L_1}|_{\tilde{\mathfrak{D}}_{k_1'}^{M'}} \|_{L^2} 
\| g_{N_2, L_2}|_{\tilde{\mathfrak{D}}_{k_2'}^{M'}} \|_{L^2} 
\|h_{{N_0, L_0}}\|_{L^2}.
\end{align*}
Next we assume $N_0 \lesssim N_1/M'$. 
We divide the proof into the two cases. First we assume $|\xi| \gg {M'}^{-1} K^{-1/2} N_1 \sim A^{-1} N_1$ 
which provides $|\overline{\Phi} (\xi_1, \eta_1, \xi_2 ,\eta_2)| \gtrsim A^{-1}N_1^3$. 
Thus, by \eqref{bilinear-est03-prop3.18} in Proposition \ref{prop3.18} and $N_0 \lesssim N_1/M'$, we obtain
\begin{align*}
I_{M'}^{k_1', k_2'} & \lesssim A^{-\frac{d-3}{2}} K^{\frac{1}{4}}N_1^{\frac{d-6}{2}} (L_0 L_1 L_2)^{\frac{1}{2}}
\|f_{N_1, L_1}|_{\tilde{\mathfrak{D}}_{k_1'}^{M'}} \|_{L^2} 
\| g_{N_2, L_2}|_{\tilde{\mathfrak{D}}_{k_2'}^{M'}} \|_{L^2} 
\|h_{{N_0, L_0}}\|_{L^2}\\
& \lesssim A^{-\frac{2d-5}{4}} K^{\frac{3}{8}}N_0^{-\frac{1}{4}}
N_1^{\frac{2d-11}{4}}    (L_0 L_1 L_2)^{\frac{1}{2}} 
\|f_{N_1, L_1}|_{\tilde{\mathfrak{D}}_{k_1'}^{M'}} \|_{L^2} 
\| g_{N_2, L_2}|_{\tilde{\mathfrak{D}}_{k_2'}^{M'}} \|_{L^2} 
\|h_{{N_0, L_0}}\|_{L^2}.
\end{align*}
Next we treat the case $|\xi| \lesssim A^{-1} N_1$. 
Since $N_0 \lesssim N_1/M' \sim A^{-1}K^{1/2}N_1$, $|(\xi,\eta)|$ is confined to a rectangle set whose long-side length is $\sim A^{-1}K^{1/2}N_1$ and 
short-side length is $\sim A^{-1}N_1$. Therefore, after decomposing $|(\xi,\eta)|$ into $\sim K^{1/2}$ square tiles whose side length is $A^{-1}$, we utilize Proposition \ref{prop3.7} and get
\begin{align*}
I_{M'}^{k_1', k_2'} & \lesssim A^{-\frac{d-3}{2}} K^{\frac{1}{4}}N_1^{\frac{d-6}{2}} (L_0 L_1 L_2)^{\frac{1}{2}}
\|f_{N_1, L_1}|_{\tilde{\mathfrak{D}}_{k_1'}^{M'}} \|_{L^2} 
\| g_{N_2, L_2}|_{\tilde{\mathfrak{D}}_{k_2'}^{M'}} \|_{L^2} 
\|h_{{N_0, L_0}}\|_{L^2}\\
& \lesssim A^{-\frac{2d-5}{4}} K^{\frac{3}{8}}N_0^{-\frac{1}{4}}
N_1^{\frac{2d-11}{4}}    (L_0 L_1 L_2)^{\frac{1}{2}} 
\|f_{N_1, L_1}|_{\tilde{\mathfrak{D}}_{k_1'}^{M'}} \|_{L^2} 
\| g_{N_2, L_2}|_{\tilde{\mathfrak{D}}_{k_2'}^{M'}} \|_{L^2} 
\|h_{{N_0, L_0}}\|_{L^2}.
\end{align*}
Collecting the above estimates, we obtain
\begin{equation}
\begin{split}
 I_{M_0}^{k_1, k_2} 
 \lesssim (\log A)  A^{-\frac{2d-5}{4}} K^{\frac{3}{8}} & N_0^{-\frac{1}{4}} 
N_1^{\frac{2d-11}{4}}    (L_0 L_1 L_2)^{\frac{1}{2}} \\
& \times \|f_{N_1, L_1}|_{\tilde{\mathfrak{D}}_{k_1}^{M_0}} \|_{L^2} 
\| g_{N_2, L_2}|_{\tilde{\mathfrak{D}}_{k_2}^{M_0}} \|_{L^2} 
\|h_{{N_0, L_0}}\|_{L^2}.\label{est03-prop3.17}
\end{split}
\end{equation}
Lastly, we assume $k_1 \in \mathfrak{K}_{M_0}$. 
In the same way as the proof for the latter term of \eqref{est02-prop3.17}, we can obtain
\begin{equation}
 I_{M_0}^{k_1, k_2} 
 \lesssim   A^{-\frac{d-3}{2}}  N_0^{-\frac{1}{4}} 
N_1^{\frac{2d-11}{4}}    (L_0 L_1 L_2)^{\frac{1}{2}} 
 \|f_{N_1, L_1}|_{\tilde{\mathfrak{D}}_{k_1}^{M_0}} \|_{L^2} 
\| g_{N_2, L_2}|_{\tilde{\mathfrak{D}}_{k_2}^{M_0}} \|_{L^2} 
\|h_{{N_0, L_0}}\|_{L^2}.\label{est04-prop3.17}
\end{equation}
Consequently, since $K \leq A^{2/3}$, \eqref{est03-prop3.17} and \eqref{est04-prop3.17} complete the proof as follows.
\begin{align*}
& \sum_{\tiny{\substack{(k_1,k_2) \in K_{M_0}^{\mathcal{I}}\\|k_1 - k_2|\leq 16}}} 
I_{M_0}^{k_1, k_2} 
\lesssim \sum_{2^{10} \leq K \leq 2^{-10}M_0} \sum_{\tiny{\substack{k_1 \in \mathfrak{K}_{M_0}^K\\|k_1 - k_2|\leq 16}}} I_{M_0}^{k_1, k_2} + \sum_{\tiny{\substack{k_1 \in \mathfrak{K}_{M_0}\\|k_1 - k_2|\leq 16}}}
I_{M_0}^{k_1, k_2}\\
& \lesssim (\log A) A^{-\frac{d-3}{2}}  N_0^{-\frac{1}{4}} 
N_1^{\frac{2d-11}{4}}    (L_0 L_1 L_2)^{\frac{1}{2}} 
 \|f_{N_1, L_1} \|_{L^2} 
\| g_{N_2, L_2} \|_{L^2} 
\|h_{{N_0, L_0}}\|_{L^2}.
\end{align*}
\end{proof}
%%%%%%%%%%%%%%%%%%%%%%%%%%%%%%%%%%%%%
%%%%%%%%%%%%%%%%%%%%%%%%%%%%%%%%%%%%%
%%%%%%%%%%%%%%%%%%%%%%%%%%%%%%%%%%%%%
%%%%%%%%%%%%%%%%%%%%%%%%%%%%%%%%%%%%%
\begin{proof}[\underline{Proof of Proposition \ref{prop3.9}}]
Collecting Propositions \ref{prop3.12}, \ref{prop3.14}, \ref{prop3.16}, \ref{prop3.17}, since 
$A^{-1}N_1 \lesssim N_0$, we completed the proof of Proposition \ref{prop3.9}.
\end{proof}
%%%%%%%%%%%%%%%%%%%%%%%%%%%%%%%%%%%%%
%%%%%%%%%%%%%%%%%%%%%%%%%%%%%%%%%%%%%
%%%%%%%%%%%%%%%%%%%%%%%%%%%%%%%%%%%%%%%%%%%%%%%%%%%%%%%%%%%%%%%%%%%%%%%%%%%%%%%%%
%%%%%%%%%%%%%%%%%%%%%%%%%%%%%%%%%%%%%%%%%%%%%%%%%%%%%%%%%%%%%%%%%%%%%%%%%%%%%%%%%
%%%%%%%%%%%%%%%%%%%%%%%%%%%%%%%%%  Subsection 3.2  %%%%%%%%%%%%%%%%%%%%%%%%%%%%%%%%%%%%%
%%%%%%%%%%%%%%%%%%%%%%%%%%%%%%%%%%%%%%%%%%%%%%%%%%%%%%%%%%%%%%%%%%%%%%%%%%%%%%%%%
%%%%%%%%%%%%%%%%%%%%%%%%%%%%%%%%%%%%%%%%%%%%%%%%%%%%%%%%%%%%%%%%%%%%%%%%%%%%%%%%%
\section{Proof of the key bilinear estimate: Case 2}\label{sec:proof-bil-2}
It remains to show \eqref{goal02-mthm} when the supports of $\ha{u}_{N_1, L_1}$ and $\ha{v}_{N_2, L_2}$ are both contained in 
$\{(\tau,\xi,\boldsymbol{\eta} ) \in \R \times \R \times \R^{d-1} \, | \, |\xi| \leq 2^{-5} N_{012}^{\max}\}.$ 
Throughout this section, $L_{012}^{\max} \ll (N_{012}^{\max})^3$ and $N_{012}^{\min} \gg 1$ are assumed. 
Let us start with the case $1 \ll N_0 \lesssim N_1 \sim N_2$.
\begin{ass}\label{assumption2}
Let $\alpha$ be dyadic such that $2^{5} \leq \alpha \leq N_1^3$ and we assume that\\
(1) $1 \ll N_0 \lesssim N_1 \sim N_2$,\\
(2) $\alpha^{-1}N_1 \leq \max(|\xi_1|, |\xi_2|) \leq  2 \alpha^{-1}N_1$.
\end{ass}
\begin{prop}\label{prop3.24}
Assume \textit{Assumption} \textnormal{\ref{assumption2}}. Then we obtain
\begin{equation}
\begin{split}
& \left|\int_{*}{  |\xi| \, \ha{w}_{{N_0, L_0}}(\tau, \xi, \boldsymbol{\eta} ) 
\ha{u}_{N_1, L_1}(\tau_1, \xi_1, \boldsymbol{\eta_1} )  \ha{v}_{N_2, L_2}(\tau_2, \xi_2, \boldsymbol{\eta_2} ) 
}
d\sigma_1 d\sigma_2 \right| \\
& \qquad \qquad  \lesssim N_0^{\frac{d-2}{2}}  N_{1}^{-1+\e}   
(L_0 L_1 L_2)^{\frac{1}{2}} \|\ha{u}_{N_1, L_1} \|_{L^2} \| \ha{v}_{N_2, L_2} \|_{L^2} 
\|\ha{w}_{{N_0, L_0}} \|_{L^2},\label{est01-prop3.24}
\end{split}
\end{equation}
where $d \sigma_j = d\tau_j d \xi_j d \boldsymbol{\eta}_j$ and $*$ denotes $(\tau, \xi, \boldsymbol{\eta}) = (\tau_1 + \tau_2, \xi_1+ \xi_2, \boldsymbol{\eta_1} + \boldsymbol{\eta_2}).$
\end{prop}
We first note that $\max(|\xi_1|, |\xi_2|) \leq  2 \alpha^{-1}N_1$ means 
$|\xi| \leq 4 \alpha^{-1} N_1$. 
Then, if $L_{012}^{\max} \gtrsim \alpha^{-1} N_1^3$, we easily get \eqref{est01-prop3.24} by utilizing the $L^4$ Strichartz estimate. Hereafter, we assume $L_{012}^{\max} \ll \alpha^{-1} N_1^3$.
\begin{defn}
Let $k = ( k_{(1)}, \ldots, k_{(d)} ) \in \Z^d$. We define cubes 
$\{ \mathcal{C}_{k}^{\alpha,A}\}_{k \in \Z^d}$ and $\{ \tilde{\mathcal{C}}_k^{\alpha,A}\}_{k \in \Z^d}$ as
\[\mathcal{C}_{k}^{\alpha,A} = \left\{ x=(x_1, \ldots, x_d) \in \R^d \left| \, 
\begin{aligned} & x_1 \in A^{-1} \alpha^{-1} N_1[k_{(1)}, k_{(1)}+1), \\ 
&x_i \in A^{-1} N_1 [k_{(i)}, k_{(i)}+1) \ 
\textnormal{for} \ i=2, \ldots, d,
 \end{aligned} \right.
                                                                            \right\},
\]
and $\tilde{\mathcal{C}}_{k}^{\alpha,A}  : = \R \times \mathcal{C}_{k}^{\alpha,A}$.
Lastly we define $\mathcal{E}_{j,k}^{\alpha,A} = {\overline{\mathcal{S}}}_{j}^A 
\cap \tilde{\mathcal{C}}_k^{\alpha,A}$.
\end{defn}
%%%%%%%%%%%%%%%%%%%%%%%%%%%%%%%%%%%%%
%%%%%%%%%%%%%%%%%%%%%%%%%%%%%%%%%%%%%
%%%%%%%%%%%%%%%%%%%%%%%%%%%%%%%%%%%%%
%%%%%%%%%%%%%%%%%%%%%%%%%%%%%%%%%%%%%
%%%%%%%%%%%%%%%%%%%%%%%%%%%%%%%%%%%%%
\begin{prop}\label{prop3.25}
Assume \textit{Assumption} \textnormal{\ref{assumption2}}. 
Let $16 \leq |j_1-j_2| \leq 32$ and $k_1$, $k_2 \in \Z^d$. Then we get
\begin{equation}
\begin{split}
& \left|\int_{*}{  \ha{w}_{{N_0, L_0}}(\tau, \xi, \boldsymbol{\eta} ) 
\ha{u}_{N_1, L_1}|_{\mathcal{E}_{j_1,k_1}^{\alpha,A} } (\tau_1, \xi_1, \boldsymbol{\eta_1} )  
\ha{v}_{N_2, L_2}|_{\mathcal{E}_{j_2,k_2}^{\alpha,A} } (\tau_2, \xi_2, \boldsymbol{\eta_2} ) 
}
d\sigma_1 d\sigma_2 \right| \\
& \qquad \qquad \qquad \lesssim  \alpha A^{-\frac{d-3}{2}} 
N_1^{\frac{d-6}{2}}  (L_0 L_1 L_2)^{\frac{1}{2}} \|\ha{u}_{N_1, L_1} \|_{L^2} \| \ha{v}_{N_2, L_2} \|_{L^2} 
\|\ha{w}_{{N_0, L_0}} \|_{L^2},
\end{split}\label{est01-prop3.25}
\end{equation}
where $d \sigma_j = d\tau_j d \xi_j d \boldsymbol{\eta}_j$ and $*$ denotes $(\tau, \xi, \boldsymbol{\eta}) = (\tau_1 + \tau_2, \xi_1+ \xi_2, \boldsymbol{\eta_1} + \boldsymbol{\eta_2}).$
\end{prop}
\begin{proof}
We use the same notations as in the proof of Proposition \ref{prop3.7}. 
Similarly to the proof of Proposition \ref{prop3.7}, by applying a suitable rotation to 
$\boldsymbol{\eta_1}$, $\boldsymbol{\eta_2}$, 
we can assume $|\eta_1 \eta_2' - \eta_2 \eta_1'| \sim A^{-1} N_1^2$, 
$|\boldsymbol{\eta_j'}| \lesssim A^{-1} N_1$ and, for fixed $\xi_1$, $\xi_2$, $\boldsymbol{\check{\eta}_1}$, 
$\boldsymbol{\check{\eta}_2}$, we will show
\begin{align}
& \left|\int_{\tilde{*}}{  \ha{w}_{{N_0, L_0}}(\tau, \xi, \boldsymbol{\eta} ) 
\ha{u}_{N_1, L_1}|_{\mathcal{E}_{j_1,k_1}^{\alpha,A} } (\tau_1, \xi_1, \boldsymbol{\eta_1} )  
\ha{v}_{N_2, L_2}|_{\mathcal{E}_{j_2,k_2}^{\alpha,A} } (\tau_2, \xi_2, \boldsymbol{\eta_2} ) 
}
d\tilde{\sigma}_1 d\tilde{\sigma}_2 \right| \notag \\
 \lesssim & \alpha^{\frac32} A^{\frac{1}{2}} N_{1}^{-2}   (L_0 L_1 L_2)^{\frac{1}{2}} 
\|\ha{u}_{N_1, L_1} (\xi_1,\boldsymbol{\check{\eta}_1}) \|_{L^2_{\tau \eta \eta' }} \| \ha{v}_{N_2, L_2} (\xi_2,\boldsymbol{\check{\eta}_2})\|_{L^2_{\tau \eta \eta'  }} 
\|\ha{w}_{{N_0, L_0}}(\xi, \boldsymbol{\check{\eta}}) \|_{L^2_{\tau \eta \eta' }}.\label{est01a-prop3.25}
\end{align}
Let $\ell=(\ell_{(1)},\ell_{(2)}) \in \Z^2$ and
\begin{align*}
& \mathcal{G}_{\ell}^{\alpha,A} : = \{ (\eta, \eta') \in A^{-1}\alpha^{-\frac{1}{2}} N_1 \bigl( 
[ \ell_{(1)}, \ell_{(1)} + 1) \times [ \ell_{(2)}, \ell_{(2)} + 1) \bigr) \},\\
& \tilde{\mathcal{G}}_{\ell}^{\alpha,A}  : = \R^2 \times \mathcal{G}_{\ell}^{\alpha,A} \times \R^{d-3}, \quad 
\mathcal{H}_{j,k,\ell}^{\alpha,A} := {\overline{\mathcal{S}}}_{j}^A \cap  \tilde{\mathcal{C}}_k^{\alpha,A} \cap  \tilde{\mathcal{G}}_\ell^{\alpha,A}.
\end{align*}
Define $\mathcal{L}_{k_1}^{\alpha,A} =
\{\ell \in \Z^2 \, | \, {\tilde{\mathcal{C}}_{k_1}^{\alpha,A}} \cap \tilde{\mathcal{G}}_{\ell}^{\alpha,A} \not= \emptyset \}$. 
We easily observe that the number of $\ell_1 \in \mathcal{L}_{k_1}^{\alpha,A}$ is comparable to $\alpha$. 
Therefore, for fixed $\ell_1 \in \mathcal{L}_{k_1}^{\alpha,A}$, it suffices to show
\begin{align}
& \left|\int_{\tilde{*}}{  \ha{w}_{{N_0, L_0}}(\tau, \xi, \boldsymbol{\eta} ) 
\ha{u}_{N_1, L_1}|_{\mathcal{H}_{j_1,k_1,\ell_1}^{\alpha,A} } (\tau_1, \xi_1, \boldsymbol{\eta_1} )  
\ha{v}_{N_2, L_2}|_{\mathcal{E}_{j_2,k_2}^{\alpha,A} } (\tau_2, \xi_2, \boldsymbol{\eta_2} ) 
}
d\tilde{\sigma}_1 d\tilde{\sigma}_2 \right|\notag \\
& \qquad \qquad \lesssim  \alpha A^{\frac{1}{2}}N_{1}^{-2}   (L_0 L_1 L_2)^{\frac{1}{2}} \|\ha{u}_{N_1, L_1} \|_{L^2_{\tau \eta \eta'}}   \| \ha{v}_{N_2, L_2} \|_{L^2_{\tau \eta \eta'}}  
\|\ha{w}_{{N_0, L_0}} \|_{L^2_{\tau \eta \eta'}} .\label{est02-prop3.25}
\end{align}
Indeed, by using this estimate, we have
\begin{align*}
 \textnormal{(LHS) of } & \eqref{est01a-prop3.25} 
\leq \sum_{\ell \in \mathcal{L}_{k_1}^{\alpha,A}} 
\textnormal{(LHS) of \eqref{est02-prop3.25}}\\
& \lesssim  \alpha  A^{\frac{1}{2}}N_{1}^{-2}  (L_0 L_1 L_2)^{\frac{1}{2}} 
\sum_{\ell_1 \in \mathcal{L}_{k_1}^{\alpha,A}} 
\|\ha{u}_{N_1, L_1}|_{\mathcal{H}_{j_1,k_1\ell_1}^{\alpha,A} } \|_{L^2_{\tau \eta \eta'}}   
\| \ha{v}_{N_2, L_2} \|_{L^2_{\tau \eta \eta'}}  
\|\ha{w}_{{N_0, L_0}} \|_{L^2_{\tau \eta \eta'}}  \\
& \lesssim  \alpha^{\frac32}  A^{\frac{1}{2}}N_{1}^{-2}  (L_0 L_1 L_2)^{\frac{1}{2}} 
\|\ha{u}_{N_1, L_1} \|_{L^2_{\tau \eta \eta'}}  \| \ha{v}_{N_2, L_2} \|_{L^2_{\tau \eta \eta'}}   
\|\ha{w}_{{N_0, L_0}} \|_{L^2_{\tau \eta \eta'}}  .
\end{align*}
We show \eqref{est02-prop3.25}. By the almost orthogonality, we can assume that 
$\ha{v}_{N_2, L_2}$ is restricted to $\mathcal{H}_{j_2,k_2,\ell_2}^{\alpha,A}$ with fixed 
$\ell_2 \in \mathcal{L}_{k_2}^{\alpha,A}$. 
By the same argument as in the proof of Proposition \ref{prop3.7}, we establish
\begin{equation}
\| \tilde{f}_{\xi_1, \boldsymbol{\check{\eta}_1}} |_{S_1} * 
\tilde{g}_{\xi_2, \boldsymbol{\check{\eta}_2}} |_{S_2} \|_{L^2(S_3)} \lesssim \alpha A^{\frac{1}{2}}  
\| \tilde{f}_{\xi_1, \boldsymbol{\check{\eta}_1}} \|_{L^2(S_1)} 
\| \tilde{g}_{\xi_2, \boldsymbol{\check{\eta}_2}} \|_{L^2(S_2)},\label{est03-prop3.25}
\end{equation}
where we used the similar notations that were defined in the proof of Proposition \ref{prop3.7}:
The functions are
\begin{align*}
& \tilde{f}_{\xi_1, \boldsymbol{\check{\eta}_1}}(\tau_1, \eta_1, \eta_1' )  = 
\ha{u}_{N_1, L_1}|_{\mathcal{H}_{j_1,k_1,\ell_1}^{\alpha,A} } (N_1^3\tau_1, \xi_1, N_1 \eta_1, N_1 \eta_1',\boldsymbol{\check{\eta}_1} ), \\
& \tilde{g}_{\xi_2, \boldsymbol{\check{\eta}_2}}(\tau_2, \eta_2, \eta_2')  = 
\ha{v}_{N_2, L_2}|_{\mathcal{H}_{j_2,k_2,\ell_2}^{\alpha,A} } 
                                                                                                                                               (N_1^3\tau_2, \xi_2, N_1\eta_2, N_1 \eta_2', \boldsymbol{\check{\eta}_2}),
\end{align*}
and with $\phi_{\xi_j, \boldsymbol{\check{\eta}_j}, c_j}  (\eta, \eta')  = (\xi_j(\xi_j^2 + |\boldsymbol{\eta}|^2) + c_j, 
\eta, \eta' )$, we define
\begin{align*}
& S_j =  \{  \phi_{\tilde{\xi}_j, \boldsymbol{\overline{\eta}_j}, c_j}  (\eta_j, \eta_j') \in \R^3 \ | \ 
(N_1 \eta_j, N_1 \eta_j')  \in  \mathcal{G}_{\ell_j}^{\alpha,A}\},\; (j=1,2),\\
& S_3 = \left\{ (\psi_{\tilde{\xi},  \boldsymbol{\bar{\eta}} } (\eta, \eta'), \eta, \eta' ) \in \R^3  \ | \  
\psi_{\xi, \boldsymbol{\check{\eta}}}  (\eta, \eta') = \xi(\xi^2 + |\boldsymbol{\eta}|^2 ) + c_0 \right\},
\end{align*}
where $c_0$, $c_1$, $c_2 \in \R$, $\tilde{\xi} = N_1^{-1} \xi$, $\tilde{\xi}_j = N_1^{-1} \xi_j$, 
$\boldsymbol{\overline{\eta}_j} = N_1^{-1}\boldsymbol{\check{\eta}_j}$, 
$\boldsymbol{\overline{\eta}} = N_1^{-1}\boldsymbol{\check{\eta}}$. 
Since $\textnormal{diam} (S_1) \lesssim A^{-1}\alpha^{-1/2}$, 
$\textnormal{diam} (S_2) \lesssim A^{-1}\alpha^{-1/2}$, we may assume 
$\textnormal{diam} (S_3) \lesssim A^{-1}\alpha^{-1/2}$. 
We establish \eqref{est03-prop3.25} by using the nonlinear Loomis-Whitney inequality. 
However, it is observed that the hypersurfaces $S_1$, $S_2$, $S_3$ do not satisfy the necessary diameter condition. 
To be specific, the diameters of the three hypersurfaces are all comparable to $A^{-1} \alpha^{-1/2}$ and the transversality is comparable to $A^{-1} \alpha^{-2}$. To overcome this difficulty, we employ new functions.
\begin{equation*}
\tilde{f}_{\xi_1, \boldsymbol{\check{\eta}_1}}^{\alpha}(\tau_1, \eta_1, \eta_1' ) = 
\tilde{f}_{\xi_1, \boldsymbol{\check{\eta}_1}}
(\tau_1,  \alpha^{\frac{1}{2}} \eta_1,  \alpha^{\frac{1}{2}} \eta_1' ), \quad 
\tilde{g}_{\xi_2, \boldsymbol{\check{\eta}_2}}^{\alpha}(\tau_2, \eta_2, \eta_2')  = 
\tilde{g}_{\xi_2, \boldsymbol{\check{\eta}_2}}(\tau_2, \alpha^{\frac{1}{2}} \eta_2, \alpha^{\frac{1}{2}}\eta_2').
\end{equation*}
Then, \eqref{est03-prop3.25} can be rewritten as
\begin{equation}
\| \tilde{f}_{\xi_1, \boldsymbol{\check{\eta}_1}}^{\alpha} |_{S_1^\alpha} * 
\tilde{g}_{\xi_2, \boldsymbol{\check{\eta}_2}}^{\alpha}|_{S_2^\alpha} \|_{L^2(S_3^\alpha)} 
\lesssim (A\alpha)^{\frac{1}{2}} 
\| \tilde{f}_{\xi_1, \boldsymbol{\check{\eta}_1}}^{\alpha} \|_{L^2(S_1^\alpha)} 
\| \tilde{g}_{\xi_2, \boldsymbol{\check{\eta}_2}}^{\alpha} \|_{L^2(S_2^\alpha)},\label{est04-prop3.25}
\end{equation}
where $\phi_{\xi_j, \boldsymbol{\check{\eta}_j}, c_j}^{\alpha}  (\eta, \eta')  
= (\xi_j(\xi_j^2 + \alpha (\eta^2 + {\eta'}^2)+|\boldsymbol{\check{\eta}}|^2) + c_j, 
\eta, \eta' )$ and
\begin{align*}
& S_j^\alpha 
=  \{  \phi_{\tilde{\xi}_j, \boldsymbol{\overline{\eta}_j}, c_j}^\alpha  (\eta_j, \eta_j') \in \R^3 \ | \ 
( \alpha^{\frac{1}{2}} N_1 \eta_j, \alpha^{\frac{1}{2}}  N_1 \eta_j')  \in  \mathcal{G}_{\ell_j}^{\alpha,A}\}, \; (j=1,2)\\
& S_3^\alpha 
= \left\{ (\psi_{\tilde{\xi},  \boldsymbol{\bar{\eta}} }^\alpha 
(\eta, \eta'), \eta, \eta' ) \in \R^3  \ | \  
\psi_{\xi, \boldsymbol{\check{\eta}}}^\alpha  (\eta, \eta') 
= \xi(\xi^2 +\alpha (\eta^2 + {\eta'}^2)+|\boldsymbol{\check{\eta}}|^2 ) + c_0 \right\},
\end{align*}
Now we verify that the hypersurfaces $S_1^\alpha$, $S_2^\alpha$, $S_3^\alpha$ satisfy the suitable conditions to utilize the nonlinear Loomis-Whitney inequality. 
Let 
\begin{equation*}
\lambda_1=\phi_{\tilde{\xi}_1, \boldsymbol{\overline{\eta}_1}, c_1}^\alpha  (\eta_1, \eta_1') 
\in S_1^\alpha, \ 
\lambda_2 = \phi_{\tilde{\xi}_2, \boldsymbol{\overline{\eta}_2}, c_2}^\alpha  (\eta_2, \eta_2') 
\in S_2^\alpha, \
\lambda_3 = (\psi_{\tilde{\xi}, \boldsymbol{\bar{\eta}} }^\alpha 
(\eta, \eta'), \eta, \eta' )  \in S_3^\alpha.
\end{equation*}
We can write the unit normals ${\mathfrak{n}}_1(\lambda_1)$, ${\mathfrak{n}}_2(\lambda_2)$, 
${\mathfrak{n}}_3(\lambda_3)$ on $\lambda_1$, $\lambda_2$, $\lambda_3$ as
\[ {\mathfrak{n}}_j(\lambda_j) = 
\frac{1}{\sqrt{1+ 4\alpha^2 {\tilde{\xi}_j}^2 |\boldsymbol{\eta_j}|^2}} 
\left(-1, \ 2 \alpha \tilde{\xi}_j \eta_j, \ 2 \alpha \tilde{\xi}_j \eta_j' \right),\; (j=1,2)
\]
and ${\mathfrak{n}}_3(\lambda_3)$ accordingly. 
Since $|\tilde{\xi}_1| \leq 2 \alpha^{-1}$, $|\tilde{\xi}_2| \leq 2 \alpha^{-1}$, 
we easily observe that the hypersurfaces satisfy the necessary regularity conditions, and the diameters of hypersurfaces are all comparable to $A^{-1} \alpha^{-1}$. 
Thus, we consider the transversality here. 
Let $(\widehat{\eta_1}, \widehat{\eta_1}')$, 
$(\widehat{\eta_2}, \widehat{\eta_2}')$, 
$(\widehat{\eta}, {\widehat{\eta}}')$ satisfy 
$(\widehat{\eta_1}, \widehat{\eta_1}') + (\widehat{\eta_2}, \widehat{\eta_2}')  = 
(\widehat{\eta}, {\widehat{\eta}}')$ and
\begin{equation*}
\widehat{\lambda}_1=\phi_{\tilde{\xi}_1, \boldsymbol{\overline{\eta}_1}, c_1}^\alpha 
(\widehat{\eta_1}, \widehat{\eta_1}')
\in S_1^\alpha, \ 
\widehat{\lambda}_2 = \phi_{\tilde{\xi}_2, \boldsymbol{\overline{\eta}_2}, c_2}^\alpha 
(\widehat{\eta_2}, \widehat{\eta_2}')
\in S_2^\alpha, \
\widehat{\lambda}_3 = (\psi_{\tilde{\xi}, \boldsymbol{\bar{\eta}} }^\alpha 
(\widehat{\eta}, {\widehat{\eta}}'), \widehat{\eta}, {\widehat{\eta}}' )  \in S_3^\alpha.
\end{equation*}
It suffices to show
\begin{equation*}
(A\alpha)^{-1} \lesssim |\textnormal{det} N(\widehat{\lambda}_1, \widehat{\lambda}_2, \widehat{\lambda}_3)|.
\end{equation*}
We have
\begin{align*}
 |\textnormal{det} N(\widehat{\lambda}_1, \widehat{\lambda}_2, \widehat{\lambda}_3)|  \gtrsim & 
 \left|\textnormal{det}
\begin{pmatrix}
-1 & -1 & - 1 \\
2 \alpha \tilde{\xi}_1 \widehat{\eta_1}  & 2 \alpha \tilde{\xi}_2 \widehat{\eta_2} & 2 \alpha \tilde{\xi} \widehat{\eta} \\
2 \alpha \tilde{\xi}_1 \widehat{\eta_1}'   & 2 \alpha \tilde{\xi}_2 \widehat{\eta_2}'  & 2 \alpha \tilde{\xi} \widehat{\eta}'
\end{pmatrix} \right| \notag \\
\gtrsim &  \alpha^2\bigl| (\widehat{\eta_1} \widehat{\eta_2}' - \widehat{\eta_2} \widehat{\eta_1}')(\tilde{\xi}_1^2 + \tilde{\xi}_1 
\tilde{\xi}_2 + \tilde{\xi}_2^2 ) \bigr|
\gtrsim  A^{-1} \alpha^{-1}.
\end{align*}
Here we used the assumptions $\alpha^{-1}N_1 \leq \max(|\xi_1|, |\xi_2|) \leq  2 \alpha^{-1}N_1$ which implies 
$|\tilde{\xi}_1^2 + \tilde{\xi}_1 
\tilde{\xi}_2 + \tilde{\xi}_2^2 | \sim \alpha^{-2}$, and $|(\widehat{\eta_1}, \widehat{\eta_1}')| 
\sim |(\widehat{\eta_2}, \widehat{\eta_2}')| \sim \alpha^{-1/2}$, 
$|\eta_1 \eta_2' - \eta_2 \eta_1'| \sim A^{-1} N_1^2$ which imply 
$|\widehat{\eta_1} \widehat{\eta_2}' - \widehat{\eta_2} \widehat{\eta_1}'|\sim (A\alpha)^{-1} $.
\end{proof}
%%%%%%%%%%%%%%%%%%%%%%%%%%%%%%%%%%%%%
%%%%%%%%%%%%%%%%%%%%%%%%%%%%%%%%%%%%%
%%%%%%%%%%%%%%%%%%%%%%%%%%%%%%%%%%%%%
%%%%%%%%%%%%%%%%%%%%%%%%%%%%%%%%%%%%%
%%%%%%%%%%%%%%%%%%%%%%%%%%%%%%%%%%%%%
%%%%%%%%%%%%%%%%%%%%%%%%%%%%%%%%%%%%%
\begin{prop}\label{prop3.26}
Assume \textit{Assumption} \textnormal{\ref{assumption2}}. Let $16 \leq |j_1-j_2| \leq 32$. Then we obtain
\begin{equation}
\begin{split}
& \left|\int_{*}{  |\xi| \, \ha{w}_{{N_0, L_0}}(\tau, \xi, \boldsymbol{\eta} ) 
\ha{u}_{N_1, L_1}|_{{\overline{\mathcal{S}}}_{j_1}^A }(\tau_1, \xi_1, \boldsymbol{\eta_1} ) 
\ha{v}_{N_2, L_2}|_{{\overline{\mathcal{S}}}_{j_2}^A }(\tau_2, \xi_2, \boldsymbol{\eta_2} ) 
}
d\sigma_1 d\sigma_2 \right| \\
& \qquad \quad  \lesssim N_0^{\frac{d-2}{2}}  N_{1}^{-1}   
(L_0 L_1 L_2)^{\frac{1}{2}} \|\ha{u}_{N_1, L_1}|_{{\overline{\mathcal{S}}}_{j_1}^A } \|_{L^2} 
\| \ha{v}_{N_2, L_2}|_{{\overline{\mathcal{S}}}_{j_2}^A }\|_{L^2} 
\|\ha{w}_{{N_0, L_0}} \|_{L^2},\label{est01-prop3.26}
\end{split}
\end{equation}
where $d \sigma_j = d\tau_j d \xi_j d \boldsymbol{\eta}_j$ and $*$ denotes $(\tau, \xi, \boldsymbol{\eta}) = (\tau_1 + \tau_2, \xi_1+ \xi_2, \boldsymbol{\eta_1} + \boldsymbol{\eta_2}).$
\end{prop}
%%%%%%%%%%%%%%%%%%%%%%%%%%%%%%%%%%%%%
%%%%%%%%%%%%%%%%%%%%%%%%%%%%%%%%%%%%%
%%%%%%%%%%%%%%%%%%%%%%%%%%%%%%%%%%%%%
%%%%%%%%%%%%%%%%%%%%%%%%%%%%%%%%%%%%%
%%%%%%%%%%%%%%%%%%%%%%%%%%%%%%%%%%%%%
\begin{proof}
In the case $A \sim 1$, since $|\xi| \lesssim \alpha^{-1}N_1$ and $N_0 \sim N_1 \sim N_2$, Proposition \ref{prop3.25} immediately gives \eqref{est01-prop3.26}. 
Therefore, we assume $A \gg 1$. 
Furthermore, without loss of generality, 
we can assume $|\eta_1 \eta_2' - \eta_2 \eta_1'| \sim A^{-1} N_1^2$, 
$|\boldsymbol{\eta_j'}| \lesssim A^{-1} N_1$. 
Let $M$ be dyadic such that $2 \leq M \leq A$ 
and suppose that $k_1$, $k_2$ satisfy $16 \leq |k_1-k_2| \leq 32$. 
We first show the following inequality.
\begin{equation}
\begin{split}
& \left|\int_{*}{  |\xi| \, \ha{w}_{{N_0, L_0}}(\tau, \xi, \boldsymbol{\eta} ) 
\ha{u}_{N_1, L_1}|_{\tilde{\mathfrak{D}}_{k_1}^{\alpha M}} (\tau_1, \xi_1, \boldsymbol{\eta_1} )  
\ha{v}_{N_2, L_2}|_{\tilde{\mathfrak{D}}_{k_2}^{\alpha M}}  (\tau_2, \xi_2, \boldsymbol{\eta_2} ) 
}
d\sigma_1 d\sigma_2 \right| \\
&  \lesssim  (M^{-\frac{1}{2}}+A^{-\frac{1}{2}} M^{\frac{1}{2}}) N_0^{\frac{d-2}{2}}  N_{1}^{-1}   
(L_0 L_1 L_2)^{\frac{1}{2}} \|\ha{u}_{N_1, L_1} \|_{L^2} \| \ha{v}_{N_2, L_2} \|_{L^2} 
\|\ha{w}_{{N_0, L_0}} \|_{L^2},
\end{split}\label{est02-prop3.26}
\end{equation}
We divide the proof of \eqref{est02-prop3.26} into the two cases 
$|\eta_1+\eta_2| \lesssim M^{-1}N_1$ and $|\eta_1+\eta_2| \gg M^{-1}N_1$.\\
\underline{Case $|\eta_1+\eta_2| \lesssim M^{-1}N_1$}

Let us fix $\boldsymbol{\eta_1'}$, $\boldsymbol{\eta_2'}$. 
\eqref{est02-prop3.26} is verified by showing the following estimate. 
\begin{equation}
\begin{split}
& \left|\int_{\hat{*}}{  |\xi| \, \ha{w}_{{N_0, L_0}}(\tau, \xi, \boldsymbol{\eta} ) 
\ha{u}_{N_1, L_1}|_{\tilde{\mathfrak{D}}_{k_1}^{\alpha M}} (\tau_1, \xi_1, \boldsymbol{\eta_1} )  
\ha{v}_{N_2, L_2}|_{\tilde{\mathfrak{D}}_{k_2}^{\alpha M}}  (\tau_2, \xi_2, \boldsymbol{\eta_2} ) 
}
d \hat{\sigma}_1 d \hat{\sigma}_2 \right| \\
&  \lesssim M^{-\frac{1}{2}}  N_{1}^{-1}   
(L_0 L_1 L_2)^{\frac{1}{2}} \|\ha{u}_{N_1, L_1} (\boldsymbol{\eta_1'}) \|_{L^2_{\tau \xi \eta}} 
\| \ha{v}_{N_2, L_2} (\boldsymbol{\eta_2'})\|_{L^2_{\tau \xi \eta}} 
\|\ha{w}_{{N_0, L_0}}(\boldsymbol{\eta'}) \|_{L^2_{\tau \xi \eta}}.
\end{split}\label{est02a-prop3.26}
\end{equation}
We first observe that $|\eta_1+\eta_2| \lesssim M^{-1}N_1$ provides 
$|\xi_1+\xi_2| \lesssim \alpha^{-1} M^{-1} N_1$. 
To see this, we assume $M \gg 1$ since $M \sim 1$ is a trivial case. 
If we write $(\xi_1, \eta_1) = (r_1 \cos \theta_1, r_1 \sin \theta_1)$, 
$(\xi_2, \eta_2) = (r_2 \cos \theta_2, r_2 \sin \theta_2)$, by the assumptions, we easily check 
$|r_1-r_2| \lesssim M^{-1} N_1$, $|\cos \theta_1+ \cos \theta_2|\lesssim \alpha^{-1} M^{-1}$ and 
$|\cos \theta_1| \lesssim \alpha^{-1}$. Therefore, 
\begin{align*}
|\xi_1+\xi_2| & = |r_1 \cos \theta_1 + r_2 \cos \theta_2|
\leq |(r_1-r_2)\cos \theta_1| + r_2|\cos \theta_1 + \cos \theta_2|\\
& \lesssim \alpha^{-1} M^{-1} N_1.
\end{align*}
Thus, it suffices to show
\begin{equation}
\begin{split}
& \left|\int_{\hat{*}}{   \ha{w}_{{N_0, L_0}}(\tau, \xi, \boldsymbol{\eta} ) 
\ha{u}_{N_1, L_1}|_{\tilde{\mathfrak{D}}_{k_1}^{\alpha M}} (\tau_1, \xi_1, \boldsymbol{\eta_1} )  
\ha{v}_{N_2, L_2}|_{\tilde{\mathfrak{D}}_{k_2}^{\alpha M}}  (\tau_2, \xi_2, \boldsymbol{\eta_2} ) 
}
d \hat{\sigma}_1 d \hat{\sigma}_2 \right| \\
& \quad \lesssim  \alpha M^{\frac{1}{2}}  N_{1}^{-2}   
(L_0 L_1 L_2)^{\frac{1}{2}} \|\ha{u}_{N_1, L_1} (\boldsymbol{\eta_1'}) \|_{L^2_{\tau \xi \eta}} 
\| \ha{v}_{N_2, L_2} (\boldsymbol{\eta_2'})\|_{L^2_{\tau \xi \eta}} 
\|\ha{w}_{{N_0, L_0}}(\boldsymbol{\eta'}) \|_{L^2_{\tau \xi \eta}}.
\end{split}\label{est02b-prop3.26}
\end{equation}
To show \eqref{est02b-prop3.26}, 
we apply a dyadic decomposition to $|\eta_1+\eta_2|$. 
Let $m \in \N_0$ and define 
\begin{equation*}
\mathbb{S}_\delta^m = \{ \eta \in \R \, | \, m \delta^{-1} N_1 \leq |\eta| \leq (m+1) \delta^{-1} N_1\}. 
\end{equation*}
Since $|\eta_1+\eta_2| \lesssim M^{-1} N_1$, we can see 
$\displaystyle{ \{\eta_1+\eta_2\} \subset \bigcup_{m \lesssim \alpha } 
\mathbb{S}_{\alpha M}^m}$. 
Therefore, for fixed $m \in \Z$, we only need to show
\begin{equation}
\begin{split}
& \left|\int_{\hat{*}}{ {\mathbf{1}}_{\mathbb{S}_{\alpha M}^m}(\eta)   \ha{w}_{{N_0, L_0}}(\tau, \xi, \boldsymbol{\eta} ) 
\ha{u}_{N_1, L_1}|_{\tilde{\mathfrak{D}}_{k_1}^{\alpha M}} (\tau_1, \xi_1, \boldsymbol{\eta_1} )  
\ha{v}_{N_2, L_2}|_{\tilde{\mathfrak{D}}_{k_2}^{\alpha M}}  (\tau_2, \xi_2, \boldsymbol{\eta_2} ) 
}
d \hat{\sigma}_1 d \hat{\sigma}_2 \right| \\
& \ \lesssim  (\alpha M)^{\frac{1}{2}}  N_{1}^{-2}   
(L_0 L_1 L_2)^{\frac{1}{2}} \|\ha{u}_{N_1, L_1} (\boldsymbol{\eta_1'}) \|_{L^2_{\tau \xi \eta}} 
\| \ha{v}_{N_2, L_2} (\boldsymbol{\eta_2'})\|_{L^2_{\tau \xi \eta}} 
\|\ha{w}_{{N_0, L_0}}(\boldsymbol{\eta'}) \|_{L^2_{\tau \xi \eta}}.
\end{split}\label{est03-prop3.26}
\end{equation}
By employing the nonlinear Loomis-Whitney inequality, we can establish \eqref{est03-prop3.26} in the same manner as that for 
Proposition \ref{prop3.4}. We omit the details.\\
\underline{Case $|\eta_1+\eta_2| \gg M^{-1}N_1$}

Next we consider the case 
$|\eta_1+\eta_2| \gg  M^{-1}N_1$. It follows from 
$2 N_0 \geq |\eta_1+\eta_2| \gg M^{-1}N_1$ that 
\begin{equation*}
(A^{-1} N_1)^{\frac{d-2}{2}} = (A^{-1} M)^{\frac{d-2}{2}} (M^{-1} N_1)^{\frac{d-2}{2}} \lesssim 
(A^{-1} M)^{\frac{1}{2}} N_0^{\frac{d-2}{2}}.
\end{equation*}
Therefore, since $|\boldsymbol{\eta_j'}| \lesssim A^{-1} N_1$, for fixed $\boldsymbol{\eta_1'}$, 
$\boldsymbol{\eta_2'}$, it suffices to prove
\begin{equation}
\begin{split}
& \left|\int_{*}{  |\xi| \, \ha{w}_{{N_0, L_0}}(\tau, \xi, \boldsymbol{\eta} ) 
\ha{u}_{N_1, L_1}|_{\tilde{\mathfrak{D}}_{k_1}^{\alpha M}} (\tau_1, \xi_1, \boldsymbol{\eta_1} )  
\ha{v}_{N_2, L_2}|_{\tilde{\mathfrak{D}}_{k_2}^{\alpha M}}  (\tau_2, \xi_2, \boldsymbol{\eta_2} ) 
}
d {\sigma}_1 d {\sigma}_2 \right| \\
& \qquad \qquad \qquad  \lesssim  A^{-\frac{d-2}{2}} N_1^{\frac{d-4}{2}}
(L_0 L_1 L_2)^{\frac{1}{2}} \|\ha{u}_{N_1, L_1}  \|_{L^2}
\| \ha{v}_{N_2, L_2} \|_{L^2}
\|\ha{w}_{{N_0, L_0}} \|_{L^2}.
\end{split}\label{est03a-prop3.26}
\end{equation}
Let a dyadic number $\tilde{M}$ satisfies $1 \leq \tilde{M} \ll M$. 
We apply a dyadic decomposition to $|\eta_1+\eta_2|$. 
Suppose that $|\eta_1+\eta_2|$ satisfies $\tilde{M}^{-1} N_1 \leq |\eta_1+\eta_2| \leq 2 \tilde{M}^{-1} N_1$. 
Then, by the same observation as in the previous case, 
we get $|\xi_1+ \xi_2| \lesssim \alpha^{-1} \tilde{M}^{-1} N_1$. 
Therefore, it suffices to show that for $\tilde{M}^{-1} N_1 \leq |\eta_1+\eta_2| \leq 2 \tilde{M}^{-1} N_1$ the 
following holds true.
\begin{equation}
\begin{split}
& \left|\int_{*}{   \ha{w}_{{N_0, L_0}}(\tau, \xi, \boldsymbol{\eta} ) 
\ha{u}_{N_1, L_1}|_{\tilde{\mathfrak{D}}_{k_1}^{\alpha M}} (\tau_1, \xi_1, \boldsymbol{\eta_1} )  
\ha{v}_{N_2, L_2}|_{\tilde{\mathfrak{D}}_{k_2}^{\alpha M}}  (\tau_2, \xi_2, \boldsymbol{\eta_2} ) 
}
d {\sigma}_1 d {\sigma}_2 \right| \\
& \qquad \qquad \lesssim  \alpha \tilde{M}^{\frac{1}{2}} A^{-\frac{d-2}{2}} N_1^{\frac{d-6}{2}}  
(L_0 L_1 L_2)^{\frac{1}{2}} \|\ha{u}_{N_1, L_1}  \|_{L^2}
\| \ha{v}_{N_2, L_2} \|_{L^2}
\|\ha{w}_{{N_0, L_0}} \|_{L^2}.
\end{split}\label{est03b-prop3.26}
\end{equation}
We observe that the condition $\tilde{M}^{-1} N_1 \leq |\eta_1+\eta_2| \leq 2 \tilde{M}^{-1} N_1$ 
yields $|\Phi(\xi_1,\boldsymbol{\eta_1} , \xi_2,\boldsymbol{\eta_2})| \gtrsim (\alpha \tilde{M} )^{-1} N_1^3$. 
To see this, we first observe
\begin{align*}
|\xi_1 \eta_2 (2 \eta_1+\eta_2) + \xi_2 \eta_1 (\eta_1 + 2 \eta_2)| & = 
\bigl| \frac{3}{2} (\xi_1 \eta_2 + \xi_2 \eta_1)(\eta_1+\eta_2) +\frac{\xi_1 \eta_2-\xi_2 \eta_1}{2} 
( \eta_1 - \eta_2)\bigr|\\
& \geq |(\xi_1 \eta_2 + \xi_2 \eta_1)(\eta_1+\eta_2) |-|(\eta_1-\eta_2) (\xi_1 \eta_2-\xi_2 \eta_1)|\\
& \gtrsim (\alpha \tilde{M}^{-1}) N_1^3.
\end{align*}
Here we used $\tilde{M}^{-1} N_1 \leq |\eta_1+\eta_2|$ and 
$|\xi_1 \eta_2-\xi_2 \eta_1| \lesssim (\alpha M)^{-1} N_1^2$ which follows from 
$(\xi_1, \eta_1) \times (\xi_2 , \eta_2) \in {\mathfrak{D}}_{k_1}^{\alpha M} \times  {\mathfrak{D}}_{k_2}^{\alpha M}$ with $|k_1 - k_2| \leq 32$. 
Hence, since $|\boldsymbol{\eta_j'}| \lesssim A^{-1} N_1$, $|\xi_1 + \xi_2| \ll |\eta_1+\eta_2|$ and $\max(|\xi_1|, |\xi_2|) \sim  \alpha^{-1} N_1$, $|\eta_1| \sim |\eta_2| \sim N_1$, we calculate 
\begin{align*}
 |\Phi(\xi_1,\boldsymbol{\eta_1} , \xi_2,\boldsymbol{\eta_2})| \geq{}& |3 \xi_1 \xi_2 (\xi_1+\xi_2) + \xi_1 \eta_2 (2 \eta_1+\eta_2) + \xi_2 \eta_1 (\eta_1 + 2 \eta_2)| + 
\mathcal{O} (\alpha^{-1} A^{-2} N_1^3)\\
\geq{}&  |\xi_1 \eta_2 (2 \eta_1+\eta_2) + \xi_2 \eta_1 (\eta_1 + 2 \eta_2)| - 3|\xi_1 \xi_2 (\xi_1+\xi_2) | + \mathcal{O} (\alpha^{-1} A^{-2} N_1^3) \geq (\alpha \tilde{M})^{-1}  N_1^3.
\end{align*}
This observation also implies that we can assume $|\eta_1+\eta_2| \ll N_1$ since 
$L_{012}^{\max} \ll \alpha^{-1}N_1$. 
Thus we assume $\tilde{M} \gg 1$ hereafter. 
By using the assumptions $\tilde{M}^{-1} N_1 \leq |\eta_1+\eta_2|=|\eta| \leq 2 \tilde{M}^{-1} N_1$ and 
$|\boldsymbol{\eta_j'}| \lesssim A^{-1} N_1$, we show the following bilinear estimates. 
\begin{align}
& \left\| {\mathbf{1}}_{G_{N_0, L_0}} \int  \ha{u}_{N_1, L_1}|_{\tilde{\mathfrak{D}}_{k_1}^{\alpha M}} 
(\tau_1, \xi_1, \boldsymbol{\eta_1} ) 
\ha{v}_{N_2, L_2}|_{\tilde{\mathfrak{D}}_{k_2}^{\alpha M}} (\tau- \tau_1, \xi-\xi_1, \boldsymbol{\eta}- \boldsymbol{\eta_1} ) d\sigma_1 \right\|_{L^2} \notag \\ 
& \qquad \qquad \qquad \qquad 
\lesssim  A^{-\frac{d-2}{2}} N_1^{\frac{d-3}{2}}  (L_1 L_2)^{\frac{1}{2}} 
\| \ha{u}_{N_1, L_1}|_{\tilde{\mathfrak{D}}_{k_1}^{\alpha M}} \|_{L^2}
\|\ha{v}_{N_2, L_2}|_{\tilde{\mathfrak{D}}_{k_2}^{\alpha M}} \|_{L^2},\label{bilinearStrichartz-4}\\
& \left\| {\mathbf{1}}_{G_{N_1, L_1} \cap {\tilde{\mathfrak{D}}_{k_1}^{\alpha M}} } 
\int \ha{v}_{N_2, L_2}|_{\tilde{\mathfrak{D}}_{k_2}^{\alpha M}} 
(\tau_2,\xi_2,\boldsymbol{\eta_2}) 
\ha{w}_{N_0, L_0} (\tau_1+ \tau_2, \xi_1+\xi_2, \boldsymbol{\eta_1}+ \boldsymbol{\eta_2} ) d\sigma_2 
\right\|_{L^2}
\notag \\ 
& \qquad \qquad \qquad \qquad 
\lesssim A^{-\frac{d-2}{2}} \tilde{M}^{-\frac{1}{2}} N_1^{\frac{d-3}{2}} (L_0 L_2)^{\frac{1}{2}} 
\|\ha{v}_{N_2, L_2}|_{\tilde{\mathfrak{D}}_{k_2}^{\alpha M}}  \|_{L^2}
\|\ha{w}_{N_0, L_0} \|_{L^2}, \notag\\
& \left\| {\mathbf{1}}_{G_{N_2,L_2} \cap \tilde{\mathfrak{D}}_{k_2}^{\alpha M} } \int 
\ha{w}_{N_0, L_0} (\tau_1+ \tau_2, \xi_1+\xi_2, \boldsymbol{\eta_1}+ \boldsymbol{\eta_2} )  
\ha{u}_{N_1, L_1}|_{\tilde{\mathfrak{D}}_{k_1}^{\alpha M}} (\tau_1, \xi_1, \boldsymbol{\eta_1} )  
d \sigma_1 
\right\|_{L^2} \notag \\ 
& \qquad \qquad \qquad \qquad 
\lesssim  A^{-\frac{d-2}{2}} \tilde{M}^{-\frac{1}{2}} N_1^{\frac{d-3}{2}} (L_0 L_1)^{\frac{1}{2}} 
\|\ha{w}_{N_0, L_0}  \|_{L^2}
\|\ha{u}_{N_1, L_1}|_{\tilde{\mathfrak{D}}_{k_1}^{\alpha M}} \|_{L^2}.\notag
\end{align}
These estimates, combined with $L_{012}^{\max} \gtrsim |\Phi(\xi_1,\boldsymbol{\eta_1} , \xi_2,\boldsymbol{\eta_2})| \geq (\alpha \tilde{M})^{-1}  N_1^3$, imply \eqref{est03b-prop3.26}. 
We only consider first estimate \eqref{bilinearStrichartz-4} here. The other estimates can be handled in the similar way since $|\eta_1| \sim |\eta_2| \gg |\eta_1+\eta_2|$. 
By the same argument as for the proof of Proposition \ref{prop3.3}, the following estimates establish the claim \eqref{bilinearStrichartz-4}.
\begin{equation}
\sup_{\tilde{M}^{-1} N_1 \leq |\eta| \leq 2 \tilde{M}^{-1} N_1} |E(\tau, \xi, \boldsymbol{\eta})| 
\lesssim A^{-(d-2)} N_1^{d-3} L_1 L_2,\label{est06-prop3.26}
\end{equation}
where 
\begin{align*}
& E(\tau, \xi, \boldsymbol{\eta}) = \{ (\tau_1, \xi_1, \boldsymbol{\eta_1} ) \in G_{N_1, L_1} \cap 
{\tilde{\mathfrak{D}}_{k_1}^{\alpha M}}  
 |  (\tau-\tau_1, \xi- \xi_1, \boldsymbol{\eta}-\boldsymbol{\eta_1} ) \in G_{N_2,L_2} \cap  {\tilde{\mathfrak{D}}_{k_2}^{\alpha M}}  \}.
\end{align*}
We recall the function $\Phi_{\xi,\boldsymbol{\eta}}(\xi_1,\boldsymbol{\eta_1} )$ which was defined in the proof of Proposition \ref{prop3.3} as
\begin{align*}
\max (L_1, L_2) & \gtrsim \bigl| 
\bigl( \tau_1-\xi_1(\xi_1^2 + |\boldsymbol{\eta_1}|^2 )\bigr)  + \bigl( (\tau-\tau_1) - (\xi-\xi_1)((\xi-\xi_1)^2 + |\boldsymbol{\eta}-\boldsymbol{\eta_1}|^2
 )\bigr) \bigr|\\
& = |\bigl( \tau-\xi(\xi^2+|\boldsymbol{\eta}|^2 ) \bigr) + \Phi_{\xi,\boldsymbol{\eta}}(\xi_1,\boldsymbol{\eta_1} )|.
\end{align*}
Let $(\tau_1, \xi_1, \boldsymbol{\eta_1} ) \in E(\tau, \xi, \boldsymbol{\eta})$. 
Since $|\xi| \ll |\eta|$ and $|\boldsymbol{\eta'}| \lesssim A^{-1}N_1$, $|\boldsymbol{\eta_1'}| \lesssim A^{-1}N_1$, for fixed $\boldsymbol{\eta_1}$, it is easily observed
\begin{align*}
|\partial_{\xi_1}\Phi_{\xi,\boldsymbol{\eta}}(\xi_1,\boldsymbol{\eta_1} )| & = |3\xi(\xi-2 \xi_1) + \eta (\eta- 2 \eta_1) |
 + \mathcal{O}(A^{-2} N_1^2)\\
& \gtrsim |\eta| N_1 \sim \tilde{M}^{-1} N_1^2.
\end{align*}
This, $ |\eta| \sim \tilde{M}^{-1} N_1$ and $|\boldsymbol{\eta_1'}| \lesssim A^{-1} N_1$ complete the proof of \eqref{est06-prop3.26}.

Next we assume $|k_1-k_2| \leq 16$ and show
\begin{equation}
\begin{split}
& \left|\int_{*}{  |\xi| \, \ha{w}_{{N_0, L_0}}(\tau, \xi, \boldsymbol{\eta} ) 
\ha{u}_{N_1, L_1}|_{\tilde{\mathfrak{D}}_{k_1}^{\alpha A}} (\tau_1, \xi_1, \boldsymbol{\eta_1} )  
\ha{v}_{N_2, L_2}|_{\tilde{\mathfrak{D}}_{k_2}^{\alpha A}}  (\tau_2, \xi_2, \boldsymbol{\eta_2} ) 
}
d\sigma_1 d\sigma_2 \right| \\
& \qquad \qquad \qquad \lesssim  N_0^{\frac{d-2}{2}}  N_{1}^{-1}   
(L_0 L_1 L_2)^{\frac{1}{2}} \|\ha{u}_{N_1, L_1} \|_{L^2} \| \ha{v}_{N_2, L_2} \|_{L^2} 
\|\ha{w}_{{N_0, L_0}} \|_{L^2},
\end{split}\label{est07-prop3.26}
\end{equation}
Similarly to the proof of \eqref{est02-prop3.26}, we divide the proof into the two cases.\\
\underline{Case $|\eta_1+ \eta_2| \lesssim A^{-1}N_1$}

As we saw above, we may assume $|\xi_1+\xi_2| \lesssim (A\alpha)^{-1} N_1$. Thus Proposition \ref{prop3.25} 
implies \eqref{est07-prop3.26}. \\
\underline{Case $|\eta_1+ \eta_2| \gg A^{-1}N_1$}

We only need to follow the proof of \eqref{est02-prop3.26} in the case 
$|\eta_1+\eta_2| \gg M^{-1}N_1$. We omit the details. 

We now see that the two estimates \eqref{est02-prop3.26} and \eqref{est07-prop3.26} yield \eqref{est01-prop3.26}. For simplicity, we use
\begin{equation*}
I_{k_1,k_2}^{\alpha,M} =  \left|\int_{*}{  |\xi| \, \ha{w}_{{N_0, L_0}}(\tau, \xi, \boldsymbol{\eta} ) 
\ha{u}_{N_1, L_1}|_{\tilde{\mathfrak{D}}_{k_1}^{\alpha M}} (\tau_1, \xi_1, \boldsymbol{\eta_1} )  
\ha{v}_{N_2, L_2}|_{\tilde{\mathfrak{D}}_{k_2}^{\alpha M}}  (\tau_2, \xi_2, \boldsymbol{\eta_2} ) 
}
d\sigma_1 d\sigma_2 \right|.
\end{equation*}
It follows from \eqref{est02-prop3.26} and \eqref{est07-prop3.26} that
\begin{align*}
& \textnormal{(LHS) of \eqref{est01-prop3.26}} \leq \sum_{2 \leq  M \leq A} \sum_{16 \leq |k_1-k_2| \leq 32} I_{k_1,k_2}^{\alpha,M} + \sum_{|k_1-k_2|\leq 16} I_{k_1,k_2}^{\alpha, A}\\
& \lesssim \sum_{2 \leq M \leq A}  (M^{-\frac{1}{2}}+M^{\frac{1}{2}} A^{-\frac{1}{2}}) N_0^{\frac{d-2}{2}}  N_{1}^{-1}   
(L_0 L_1 L_2)^{\frac{1}{2}} \|\ha{u}_{N_1, L_1} \|_{L^2} \| \ha{v}_{N_2, L_2} \|_{L^2} 
\|\ha{w}_{{N_0, L_0}} \|_{L^2}\\ 
& \qquad \qquad \qquad +  N_0^{\frac{d-2}{2}}  N_{1}^{-1}   
(L_0 L_1 L_2)^{\frac{1}{2}} \|\ha{u}_{N_1, L_1} \|_{L^2} \| \ha{v}_{N_2, L_2} \|_{L^2} 
\|\ha{w}_{{N_0, L_0}} \|_{L^2}\\
& \lesssim N_0^{\frac{d-2}{2}}  N_{1}^{-1}   
(L_0 L_1 L_2)^{\frac{1}{2}} \|\ha{u}_{N_1, L_1} \|_{L^2} \| \ha{v}_{N_2, L_2} \|_{L^2} 
\|\ha{w}_{{N_0, L_0}} \|_{L^2},
\end{align*}
which completes the proof of \eqref{est01-prop3.26}.
\end{proof}
%%%%%%%%%%%%%%%%%%%%%%%%%%%%%%%%%%%%%
%%%%%%%%%%%%%%%%%%%%%%%%%%%%%%%%%%%%%
%%%%%%%%%%%%%%%%%%%%%%%%%%%%%%%%%%%%%
In the same manner as in the proof of \eqref{goal01-prop3.2} in the case \textnormal{(Ic)} (see p.\pageref{Proofof(4.1)}), Proposition \ref{prop3.26} gives Proposition \ref{prop3.24}. We omit the proof.
%%%%%%%%%%%%%%%%%%%%%%%%%%%%%%%%%%%%%
%%%%%%%%%%%%%%%%%%%%%%%%%%%%%%%%%%%%%
\section{Proof of the key bilinear estimate: Case 3}\label{sec:proof-bil-3}
Next we deal with $1 \ll N_2 \lesssim N_1 \sim N_0$. 
\begin{ass}\label{assumption3}
Let $\alpha$ be dyadic such that $2^{5} \leq \alpha \leq N_1^3$ and we assume that\\
(1) $1 \ll N_2 \lesssim N_1 \sim N_0$,\\
(2) $\alpha^{-1}N_1 \leq \max(|\xi_1|, |\xi_2|) \leq  2 \alpha^{-1}N_1$.
\end{ass}
%%%%%%%%%%%%%%%%%%%%%%%%%%%%%%%%%%%%%
%%%%%%%%%%%%%%%%%%%%%%%%%%%%%%%%%%%%%
%%%%%%%%%%%%%%%%%%%%%%%%%%%%%%%%%%%%%
\begin{prop}\label{prop3.27}
Assume \textit{Assumption} \textnormal{\ref{assumption3}}. Then we get
\begin{equation}
\begin{split}
& \left|\int_{*}{  |\xi| \, \ha{w}_{{N_0, L_0}}(\tau, \xi, \boldsymbol{\eta} ) 
\ha{u}_{N_1, L_1}(\tau_1, \xi_1, \boldsymbol{\eta_1} )  \ha{v}_{N_2, L_2}(\tau_2, \xi_2, \boldsymbol{\eta_2} ) 
}
d\sigma_1 d\sigma_2 \right| \\
& \qquad \qquad  \lesssim N_2^{\frac{d-4}{2}+2 \e}  N_1^{-\e}    
(L_0 L_1 L_2)^{\frac{1}{2}} \|\ha{u}_{N_1, L_1} \|_{L^2} \| \ha{v}_{N_2, L_2} \|_{L^2} 
\|\ha{w}_{{N_0, L_0}} \|_{L^2},\label{est01-prop3.27}
\end{split}
\end{equation}
where $d \sigma_j = d\tau_j d \xi_j d \boldsymbol{\eta}_j$ and $*$ denotes $(\tau, \xi, \boldsymbol{\eta}) = (\tau_1 + \tau_2, \xi_1+ \xi_2, \boldsymbol{\eta_1} + \boldsymbol{\eta_2}).$
\end{prop}
Since $|\xi|\lesssim \alpha^{-1} N_1$, \eqref{est01-prop3.27} is given by the following estimate.
\begin{prop}\label{prop3.28}
Assume \textit{Assumption} \textnormal{\ref{assumption3}}. Let $16 \leq |j_1-j_2| \leq 32$. Then we get
\begin{align*}
& \left|\int_{*}{  \ha{v}_{N_2, L_2}(\tau, \xi, \boldsymbol{\eta} ) 
\ha{u}_{N_1, L_1}|_{{\overline{\mathcal{S}}}_{j_1}^A }(\tau_1, \xi_1, \boldsymbol{\eta_1} ) 
\ha{w}_{{N_0, L_0}}|_{{\overline{\mathcal{S}}}_{j_2}^A }(\tau_2, \xi_2, \boldsymbol{\eta_2} ) 
}
d\sigma_1 d\sigma_2 \right| \\
& \qquad \  \lesssim \alpha N_2^{\frac{d-4}{2} + 2\e}  N_{1}^{-1-\e}   
(L_0 L_1 L_2)^{\frac{1}{2}} \|\ha{u}_{N_1, L_1}|_{{\overline{\mathcal{S}}}_{j_1}^A } \|_{L^2} 
\| \ha{v}_{N_2, L_2}\|_{L^2} 
\|\ha{w}_{{N_0, L_0}}|_{{\overline{\mathcal{S}}}_{j_2}^A } \|_{L^2},
\end{align*}
where $d \sigma_j = d\tau_j d \xi_j d \boldsymbol{\eta}_j$ and $*$ denotes $(\tau, \xi, \boldsymbol{\eta}) = (\tau_1 + \tau_2, \xi_1+ \xi_2, \boldsymbol{\eta_1} + \boldsymbol{\eta_2}).$
\end{prop}
By exchanging the roles of $\ha{w}_{N_0, L_0}$ and $\ha{v}_{N_2,
  L_2}$, we can establish Proposition \ref{prop3.28} in the same
manner as Proposition \ref{prop3.26}. 
In addition, by following the proof of \eqref{goal01-prop3.2} in the case \textnormal{(Ic)}, Proposition \ref{prop3.28} yields Proposition \ref{prop3.27}. We omit the details. 

Now we show \eqref{goal02-mthm} with the condition
\begin{equation}
\supp \ha{u}_{N_1, L_1} \cup \supp \ha{v}_{N_2, L_2} \subset 
\{(\tau,\xi,\boldsymbol{\eta} ) \in \R \times \R \times \R^{d-1} \, | \, |\xi| \leq 2^{-5} N_{012}^{\max}\}.
\label{assumption-xi}
\end{equation}
\begin{proof}[\underline{Proof of \eqref{goal02-mthm} under \eqref{assumption-xi}}]
By symmetry, we can assume $N_2 \leq N_1$. 
Let us consider $1 \ll N_0 \lesssim N_1 \sim N_2$. 
We define
\begin{align*}
& E_{\alpha}:=\{(\xi_1, \xi_2) \, | \, \alpha^{-1}N_1 \leq \max(|\xi_1|, |\xi_2|) \leq  2 \alpha^{-1}N_1\},\\ 
& F := \{ (\xi_1, \xi_2) \, | \, \max(|\xi_1|, |\xi_2|) \leq  N_1^{-2}\}.
\end{align*}
Applying dyadic decomposition to $\max(|\xi_1|, |\xi_2|)$, we see
\begin{align*}
 & \left|\int{ \left(\partial_{x} w_{N_0, L_0} \right) 
 u_{N_1, L_1}  v_{N_2, L_2}
} 
dt dx d \mathbf{y} \right|\\
& \lesssim \sum_{2^5 \leq \alpha \leq N_1^3} 
 \left|\int_{*}{  |\xi| \ha{w}_{{N_0, L_0}}(\tau, \xi, \boldsymbol{\eta}) {\mathbf{1}}_{E_{\alpha}}(\xi_1,\xi_2)
\ha{u}_{N_1, L_1}(\tau_1, \xi_1, \boldsymbol{\eta_1})  \ha{v}_{N_2, L_2}(\tau_2, \xi_2, \boldsymbol{\eta_2}) 
}
d\sigma_1 d\sigma_2 \right|\\
& + \quad  \left|\int_{*}{ |\xi| \ha{w}_{{N_0, L_0}}(\tau, \xi, \boldsymbol{\eta}) {\mathbf{1}}_{F}(\xi_1,\xi_2)
\ha{u}_{N_1, L_1}(\tau_1, \xi_1, \boldsymbol{\eta_1})  \ha{v}_{N_2, L_2}(\tau_2, \xi_2, \boldsymbol{\eta_2}) 
}
d\sigma_1 d\sigma_2 \right|.
\end{align*}
The first term can be handled by Proposition \ref{prop3.24} as follows.
\begin{align*}
& \sum_{2^5 \leq \alpha \leq N_1^3} 
 \left|\int_{*}{  |\xi| \ha{w}_{{N_0, L_0}}(\tau, \xi, \boldsymbol{\eta}) {\mathbf{1}}_{E_{\alpha}}(\xi_1,\xi_2)
\ha{u}_{N_1, L_1}(\tau_1, \xi_1, \boldsymbol{\eta_1})  \ha{v}_{N_2, L_2}(\tau_2, \xi_2, \boldsymbol{\eta_2}) 
}
d\sigma_1 d\sigma_2 \right|\\
& \lesssim \sum_{2^5 \leq \alpha \leq N_1^3}  N_0^{\frac{d-2}{2}}  N_{1}^{-1+\e}   
(L_0 L_1 L_2)^{\frac{1}{2}} \|\ha{u}_{N_1, L_1} \|_{L^2} \| \ha{v}_{N_2, L_2} \|_{L^2} 
\|\ha{w}_{{N_0, L_0}} \|_{L^2}\\
& \lesssim  N_0^{\frac{d-2}{2}}  N_{1}^{-1+2 \e}   
(L_0 L_1 L_2)^{\frac{1}{2}} \|\ha{u}_{N_1, L_1} \|_{L^2} \| \ha{v}_{N_2, L_2} \|_{L^2} 
\|\ha{w}_{{N_0, L_0}} \|_{L^2},
\end{align*}
which implies \eqref{goal02-mthm} since $s > (d-4)/2$ and $L_{012}^{\max} \ll N_1^3$. 
Next we consider the second term. The inequality $\max(|\xi_1|, |\xi_2|) \leq  N_1^{-2}$ implies 
$|\xi| \leq 2 N_1^{-2}$. Therefore, the $L^4$ Strichartz estimate is enough to verify the claim. 

By using Proposition \ref{prop3.27}, the case $1 \ll N_2 \lesssim N_0 \sim N_1$ can be treated in the similar way. We omit the details.
\end{proof}
%%%%%%%%%%%%%%%%%%%%%%%%%%%%%%%%%%%%%%%%%%%%%%%%%%%%%%%%%%%%%%%%%%%%%%%%%%
%%%%%%%%%%%%%%%%%%%%%%%%%%%%%%%%%%%%%%%%%%%%%%%%%%%%%%%%%%%%%%%%%%%%%%%%%%
%%%%%%%%%%%%%%%%%%%%%%%%%%%%%%%%%%%%%%%%%%%%%%%%%%%%%%%%%%%%%%%%%%%%%%%%%%
%%%%%%%%%%%%%%%%%%%%%%%%%%%%%%%%%%%%%%%%%%%%%%%%%%%%%%%%%%%%%%%%%%%%%%%%%%
%%%%%%%%%%%%%%%%%%%%%%%%%%%%%%%%%%%%%%%%%%%%%%%%%%%%%%%%%%%%%%%%%%%%%%%%%%

\appendix

  \section{Proof of Proposition \ref{bilinear-transversal}}
First we consider \eqref{est01-bilinear-transversal}. 
Set $\tilde{B}_r(p) = \R \times {B}_r(p) $ and $\zeta=(\xi,\boldsymbol{\eta})$, $\zeta_j=(\xi_j,\boldsymbol{\eta}_j)$. By Plancherel's theorem, it suffices to show
\begin{equation}
\begin{split}
 \left\| \int  \ha{u}_{N_1, L_1}|_{\tilde{B}_r(p)}(\tau_1,\zeta_1) 
\ha{v}_{N_2, L_2} (\tau-\tau_1,\zeta- \zeta_1) d\tau_1 d\zeta_1 \right\|_{L^2}  \lesssim r^{\frac{d-1}{2}} K^{-\frac{1}{2}} 
(L_1 L_2)^{\frac{1}{2}} \| \widehat{u}_{N_1, L_1} \|_{L^2} \| \widehat{v}_{N_2, L_2}  \|_{L^2}.
\end{split}\label{est03-bilinear-transversal}
\end{equation}
By performing a harmless decomposition, we may replace $r$ with $r'$ such that 
$r' \ll d^{-1} r$ in the above. 
Furthermore, by the almost orthogonality, we can assume that there exists $p'\in \R^d$ such that 
$\zeta-\zeta_1\in  {B}_{r'}(p')$. 
Since $\varphi$ is a cubic polynomial, we deduce from $N_2 \leq N_1$ that 
\begin{equation*}
\sup_{1 \leq i,j \leq d}(|\partial_i \partial_j \varphi(\zeta_1)| + |\partial_i \partial_j \varphi(\zeta- \zeta_1)|) \lesssim N_1.
\end{equation*} 
Therefore, because $K \gtrsim r N_1$, we easily observe 
\begin{equation*}
|\nabla \varphi(\zeta)-\nabla \varphi(\zeta')| \ll K \quad \textnormal{if } \zeta, \zeta' \in B_{r'}(p).
\end{equation*}
This implies that there exists $j \in \N$ such that $1 \leq j \leq d$ and 
\begin{equation}
|\partial_j \varphi(\zeta_1)-\partial_j \varphi(\zeta_2)| \gtrsim K,\label{est04-bilinear-transversal}
\end{equation} 
for all $\zeta_1$, $\zeta_2$ which satisfy that there exist $\tau_1$ and $\tau_2$ such that 
$(\tau_1,\zeta_1) \in \supp \widehat{u}_{N_1,L_1} \cap \tilde{B}_{r'}(p) $ and $(\tau_2,\zeta_2) \in \supp \widehat{v}_{N_2,L_2} \cap \tilde{B}_{r'}(p')$, respectively. 
Now we turn to show \eqref{est03-bilinear-transversal}. By the Cauchy-Schwarz inequality, we get
\begin{align*}
\biggl\| \int  \ha{u}_{N_1, L_1}|_{\tilde{B}_r(p)}(\tau_1, \zeta_1)  &
\ha{v}_{N_2, L_2}  (\tau- \tau_1, \zeta-\zeta_1 ) d\tau_1 d \zeta_1 \biggr\|_{L^2}  \\
\leq & \left\|   \left(\left| 
\ha{u}_{N_1, L_1}  \right|^2 * 
\left|\ha{v}_{N_2, L_2}  \right|^2
 \right)^{1/2} |E(\tau, \zeta )|^{1/2} \right\|_{L^2} \\
\leq & \sup_{\tau, \zeta} |E(\tau, \zeta )|^{1/2} 
 \left\| \left| 
 \ha{u}_{N_1, L_1} \right|^2 * 
\left|\ha{v}_{N_2, L_2}  \right|^2
\right\|_{L^1}^{1/2}\\
\leq & \sup_{\tau, \zeta} |E(\tau, \zeta )|^{1/2} 
\|\ha{u}_{N_1, L_1}\|_{L^2} 
\| \ha{v}_{N_2, L_2}  \|_{L^2},
\end{align*}
where $E(\tau, \zeta) \subset \R^{d+1}$ is defined by
\begin{equation*}
E(\tau, \zeta) := \{ (\tau_1, \zeta_1) \in G_{N_1,L_1} \cap \tilde{B}_{r'}(p)
\, | \, (\tau-\tau_1, \zeta- \zeta_1) \in G_{N_2,L_2} \}.
\end{equation*}
Thus, it suffices to show
\begin{equation}
|E(\tau, \zeta)| \lesssim r^{d-1} K^{-1} L_1 L_2.\label{est05-bilinear-transversal}
\end{equation}
If we fix $\zeta_1$, it is easily observed that
\begin{equation}
 | \{ \tau_1 \, | \, (\tau_1, \zeta_1) \in E(\tau, \zeta) \}| 
\lesssim \min(L_1, L_2).\label{est06-bilinear-transversal}
\end{equation}
Next, if we fix 
$(\zeta_{1,1}, \ldots, \zeta_{1,j-1}, \zeta_{1,j+1} ,\ldots, \zeta_{1,d})$, since $\max (L_1, L_2)  \gtrsim  |\varphi(\zeta_1)+\varphi(\zeta-\zeta_1)|$, 
the inequality \eqref{est04-bilinear-transversal} implies that $\zeta_{1,j}$ is confined to an interval whose length is comparable to $\max (L_1, L_2)/K$. 
This, combined with \eqref{est06-bilinear-transversal} and $\zeta_1 \in B_{r'}(p)$, 
yields \eqref{est05-bilinear-transversal}.

To see \eqref{est02-bilinear-transversal}, it suffices to show
\begin{equation*}
|\zeta_1| \geq 2 |\zeta_2| \Longrightarrow 
|\nabla \varphi(\zeta_1)-\nabla \varphi(\zeta_2)| \gtrsim |\zeta_1|^2,
\end{equation*}
which is immediately verified by $|\partial_{\xi} \varphi(\xi,\boldsymbol{\eta})|=|3\xi^2 + |\boldsymbol{\eta}|^2| \geq |(\xi,\boldsymbol{\eta})|^2$ and 
$|\partial_{\xi_1} \varphi(\xi_1,\boldsymbol{\eta}_1)| \leq 3 |(\xi_1,\boldsymbol{\eta}_1)|^2$.\qed

\subsection*{Acknowledgment}
Financial support by the
  German Research Foundation (DFG) through the CRC 1283 ``Taming uncertainty and profiting from
  randomness and low regularity in analysis, stochastics and their
  applications'' is acknowledged.

\bibliographystyle{plain}
\bibliography{zk3d}

\end{document}